\documentclass{amsart} 
\usepackage{amssymb}
\usepackage{amsmath}
\usepackage{amsfonts}

\begin{document}
\newtheorem{theo}{Theorem}[section]
\newtheorem{prop}[theo]{Proposition}
\newtheorem{lemma}[theo]{Lemma}
\newtheorem{coro}[theo]{Corollary}
\theoremstyle{definition}
\newtheorem{exam}[theo]{Example}
\newtheorem{defi}[theo]{Definition}
\newtheorem{rem}[theo]{Remark}


\newcommand{\Bb}{{\bf B}}
\newcommand{\Cb}{{\bf C}}
\newcommand{\Nb}{{\bf N}}
\newcommand{\Qb}{{\bf Q}}
\newcommand{\Rb}{{\bf R}}
\newcommand{\Zb}{{\bf Z}}
\newcommand{\Ac}{{\mathcal A}}
\newcommand{\Bc}{{\mathcal B}}
\newcommand{\Cc}{{\mathcal C}}
\newcommand{\Dc}{{\mathcal D}}
\newcommand{\Ec}{{\mathcal E}}
\newcommand{\Fc}{{\mathcal F}}
\newcommand{\Ic}{{\mathcal I}}
\newcommand{\Jc}{{\mathcal J}}
\newcommand{\Kc}{{\mathcal K}}
\newcommand{\Lc}{{\mathcal L}}
\newcommand{\Mx}{{\mathcal M}}
\newcommand{\Nc}{{\mathcal N}}
\newcommand{\Oc}{{\mathcal O}}
\newcommand{\Pc}{{\mathcal P}}
\newcommand{\Qc}{{\mathcal Q}}
\newcommand{\Rc}{{\mathcal R}}
\newcommand{\Sc}{{\mathcal S}}
\newcommand{\Tb}{{\bf T}}
\newcommand{\Tc}{{\mathcal T}}
\newcommand{\TC}{{\mathcal TC}}
\newcommand{\Ub}{{\bf U}}
\newcommand{\Uc}{{\mathcal U}}
\newcommand{\Vb}{{\bf V}}
\newcommand{\Vc}{{\mathcal V}}
\newcommand{\Wb}{{\bf W}}
\newcommand{\Wc}{{\mathcal W}}
\newcommand{\btu}{\bigtriangleup}

\author{Greg Kuperberg and Nik Weaver}

\title [Quantum metrics]
       {A von Neumann algebra approach to quantum metrics}

\address {Department of Mathematics\\
          University of California\\
          Davis, CA 95616}
\address {Department of Mathematics\\
          Washington University in Saint Louis\\
          Saint Louis, MO 63130}

\email {greg@math.ucdavis.edu, nweaver@math.wustl.edu}


\date{May 3, 2010}

\begin{abstract}
We propose a new definition of quantum metric spaces, or W*-metric
spaces, in the setting of von Neumann algebras. Our definition
effectively reduces to the classical notion in the atomic abelian case,
has both concrete and intrinsic characterizations, and admits a wide
variety of tractable examples. A natural application and motivation of
our theory is a mutual generalization of the standard models of
classical and quantum error correction.
\end{abstract}

\maketitle


It has proven to be fruitful in abstract analysis to think of
various structures connected to Hilbert space as ``noncommutative''
or ``quantum'' versions of classical mathematical objects.
This point of view has been emphasized in \cite{Con} (see also \cite{W5}).
For instance, it is well established that C*-algebras and von Neumann
algebras can profitably be thought of as quantum topological and
measure spaces, respectively. The use of the word ``quantum'' is called
for if, for example, the structures in question play a role in modelling
quantum mechanical systems analogous to the role played by the
corresponding classical mathematical structures in classical physics.

Basic examples in noncommutative geometry such as the quantum tori
\cite{Rie2} clearly exhibit a metric aspect in that they carry a natural
noncommutative analog of the algebra of bounded scalar-valued Lipschitz
functions on a metric space. However, the general notion of a quantum
metric has been elusive. Possible definitions have been proposed by Connes
\cite{CL}, Rieffel \cite{Rie3}, and Weaver \cite{W2}. Connes' definition
involves his notion of spectral triples and is patterned
after the Dirac operator on a Riemannian manifold.
Possibly the right interpretation of this definition is ``quantum
Riemannian manifold'' rather than ``quantum metric space''.
Weaver proposed a definition involving
unbounded derivations of von Neumann algebras into dual operator bimodules.
This definition neatly recovers classical Lipschitz algebras in the
abelian case, but it has not led to a deeper structure theory.
Rieffel's definition, which is also called a C*-metric space \cite{Rie4},
generalizes the classical Lipschitz seminorm on functions on a metric space.
This definition has attracted the most interest recently; among other
interesting properties, it leads to a useful model of Gromov-Hausdorff
convergence.

We introduce a new definition of a quantum metric space. To distinguish
between our model and that of Rieffel, it can also be called a W*-metric
space. Recall that an
{\it operator system} is a linear subspace of $\Bc(H)$ that is self-adjoint
and contains the identity operator. We say that
a {\it W*-filtration} of $\Bc(H)$ is a one-parameter
family of weak* closed operator
systems $\Vc_t$, $t \in [0,\infty)$, such that
\begin{quote}
(i) $\Vc_s\Vc_t \subseteq \Vc_{s+t}$ for all $s,t \geq 0$

\noindent (ii) $\Vc_t = \bigcap_{s > t}\Vc_s$ for all $t \geq 0$.
\end{quote}
Notice that $\Vc_0$ is automatically a von Neumann algebra, since the
filtration
condition (i) implies that it is stable under products. We define
a {\it W*-metric} on a von Neumann algebra $\Mx \subseteq \Bc(H)$
to be a W*-filtration $\{\Vc_t\}$ such that $\Vc_0$ is the commutant of
$\Mx$. Since we interpret a W*-metric as a type of quantum metric, and
since it is the main type that we will consider in this article, we
will also just call it a {\it quantum metric}.

We will justify this definition with various constructions and results.
We can begin with a correspondence table between the usual axioms of
a metric space and some of our conditions:
\begin{eqnarray*}
d(x,x) = 0&\longleftrightarrow& I \in \Vc_0\cr
d(x,y) = d(y,x)&\longleftrightarrow& \Vc_t^* = \Vc_t\cr
d(x,z) \leq d(x,y) + d(y,z)&\longleftrightarrow& \Vc_s\Vc_t \subseteq
\Vc_{s+t}
\end{eqnarray*}
The rough intuition is that
$\Vc_t$ consists of the operators that do not displace any mass
more than $t$ units away from where it started.

We will show that quantum metrics on $\Mx$ do not depend on the
representation of $\Mx$ (Theorem \ref{metcor}). We will also
show that any quantum metric on $\Mx$ yields a C*-algebra
$UC(\Mx)$ of uniformly continuous elements and an algebra
${\rm Lip}(\Mx) \subseteq UC(\Mx)$ of (commutation) Lipschitz elements
that are both weak* dense in $\Mx$ (Proposition \ref{ucdense} and
Theorem \ref{lipdensity}).

One motivation for our approach is the standard model of quantum error
correction. Classical error correction is a theory of minimum-distance
sets in metric spaces: if $X$ is a metric space and $C \subseteq X$ is a\
subset with minimum distance $t$ (i.e., $\inf\{d(x,y): x, y \in C, x \neq y\}
= t$), then $C$ is said to be a {\it code} that
detects errors of size less than $t$ and corrects errors of size less
than $t/2$. In particular, the {\it Hamming metric} on
the space $X = \{0,1\}^n$ of $n$-bit words is defined by
letting the distance between two words be the number of bits that
differ. In quantum information theory, the {\it quantum Hamming metric}
is a quantum metric in our sense on $n$ qubits. Here a {\it qubit} is
a quantum system with von Neumann algebra $M_2(\Cb)$; thus
$M_{2^n}(\Cb) \cong M_2(\Cb)^{\otimes n}$ is the von Neumann algebra of
$n$ qubits. The filtration of the quantum Hamming metric is
\begin{eqnarray*}
\Vc_t &=& {\rm span}\{A_1 \otimes \cdots \otimes A_n: A_i \in M_2(\Cb)\cr
&&\phantom{\limsup}\hbox{ and $A_i = I_2$ for all but at most $t$
values of $i$}\} \subseteq M_{2^n}(\Cb).
\end{eqnarray*}
This filtration models the error operators that corrupt at most $t$
qubits for some $t$. (Note that an operator $C \in \Vc_t$ can be a
linear combination, or quantum superposition, of operators that
corrupt different sets of $t$ or fewer qubits.) There is a natural
definition of a code subspace $C \subseteq (\Cb^2)^{\otimes n}$ of
minimum distance $t$, which is then a {\it quantum code}. Quantum codes
both resemble classical codes and are used for the same purposes.
Various generalizations of the quantum Hamming metric on qubits have
been studied; for instance, it is routine to replace qubits by
{\it qudits} with algebra $M_d(\Cb)$. Knill, Laflamme, and Viola considered a
general operator system as an error model \cite{KLV}; this is equivalent
to a {\it quantum graph metric} with $\Mx = \Bc(H)$ (see Section
\ref{graphmetric}). However, more general W*-metrics, even with
$\Mx = \Bc(H)$, have not previously been studied to our knowledge.

A second motivation is the intermediate model of a measurable metric
space due to Weaver \cite{W0}, in which the metric set $X$ is replaced
by a measure space $(X,\mu)$. If $\mu$ is atomic, so that $\Mx =
l^\infty(X)$, then a quantum metric on $\Mx$ is equivalent to a
classical metric on $X$ (Proposition \ref{aa}). In the general
measurable setting, a measurable metric on $(X,\mu)$ is equivalent
to a reflexive quantum metric
on $L^\infty(X,\mu)$ (Theorem \ref{qmeasmet}).

Our new definition is related to the other models of quantum metric
spaces mentioned above. First, every spectral triple in Connes' sense
yields
a W*-metric (Definition \ref{tripdef}). This can be seen as encoding
the purely metric features of the spectral triple, as opposed to its
Riemannian or spinorial structure. Second, every W*-metric yields a
Leibniz Lipschitz seminorm in Rieffel's sense (Definition \ref{comlip}).
Third, as anticipated in
Weaver's earlier work, every W*-metric yields a Lipschitz algebra that
is the domain of a W*-derivation (Definition \ref{opdeLeeuw}). One
twist is that in the noncommutative case, the classical Lipschitz
condition $|f(x)-f(y)| \leq C\cdot d(x,y)$
splits into two distinct conditions, a commutation condition
and a spectral condition. The commutation version is the one with good
algebraic properties, but it is
the spectral version that admits an elegant abstract axiomatization
(Definition \ref{abspeclip}).

We establish several equivalent definitions of a W*-metric space. Our
main definition is the one stated above in terms of W*-filtrations.
This is trivially equivalent to a {\it displacement gauge} on
$\Bc(H)$ (Definition \ref{filt} (b)). A much deeper result gives an
intrinsic characterization of quantum metrics in terms of {\it quantum
distance functions} defined on pairs of projections in
$\Mx \overline{\otimes} \Bc(l^2)$ (Definition \ref{qdfdef}/Theorem \ref{abch}).
This characterization can be attractively recast in terms of
{\it quantum Lipschitz gauges} on the self-adjoint part of
$\Mx \overline{\otimes} \Bc(l^2)$ (Definition \ref{abspeclip}/Corollary \ref{lipequiv}).

We wish to thank David Blecher, Chris
Bumgardner, Renato Feres, Jerry Kaminker, Michael Kapovich, Nets Katz,
Greg Knese, Emmanuel Knill, John McCarthy, Stephen Power, Marc Rieffel,
Zhong-Jin Ruan, David Sherman, and Andr\'{a}s Vasy for helpful conversations.
We also thank the referee for suggesting many minor improvements.

We work with complex scalars throughout. ``Projection'' always means
``orthogonal projection''.

\tableofcontents

\section{Measurable and quantum relations}\label{mrs}

It is convenient to begin with a brief summary of basic results about
measurable and quantum relations. This material will be used sporadically
in subsequent sections. For a fuller treatment see \cite{W6}. The reader
is encouraged to skip this chapter and refer back to it as needed.

We first state the definition of a measurable relation.
A measure space $(X,\mu)$ is {\it finitely decomposable} if it can be
partitioned into a possibly uncountable family of finite measure subspaces
$X_\lambda$ such that a set $S \subseteq X$ is measurable if and only
if its intersection with each $X_\lambda$ is measurable, in which case 
$\mu(S) = \sum \mu(S \cap X_\lambda)$ (\cite{W4}, Definition 6.1.1).
Finitely decomposable measures generalize both $\sigma$-finite measures
and counting measures. The significance of the condition is that it
ensures $L^\infty(X,\mu) \cong L^1(X,\mu)^*$, and hence that the
projections in $L^\infty(X,\mu)$ (equivalently, the measurable subsets of
$X$ up to null sets) constitute a complete Boolean algebra.

\begin{defi}\label{measrel}
(\cite{W6}, Definition 1.2)
Let $(X,\mu)$ be a finitely decomposable measure space.
A {\it measurable relation} on $X$ is a family $\Rc$ of ordered pairs
of nonzero projections in $L^\infty(X,\mu)$ such that
$$\left(\bigvee p_\lambda,\bigvee q_\kappa\right) \in\Rc
\quad\Leftrightarrow\quad
\hbox{some }(p_\lambda, q_\kappa) \in \Rc\eqno{(*)}$$
for any pair of families of nonzero projections $\{p_\lambda\}$ and
$\{q_\kappa\}$.
\end{defi}

Now let $H$ be a complex Hilbert space, not necessarily separable. Recall
(\cite{Tak}, Definition II.2.1) that the {\it weak*} (or {\it $\sigma$-weak
operator}) topology on $\Bc(H)$ is the
weak topology arising from the pairing $\langle A,B\rangle \mapsto
{\rm tr}(AB)$ of $\Bc(H)$ with the trace class operators $\TC(H)$;
that is, it is the weakest topology that makes the map $A \mapsto
{\rm tr}(AB)$ continuous for all $B \in \TC(H)$. The weak* topology
is finer than the weak operator topology but the two agree on bounded
sets.

An {\it operator algebra} is a linear subspace of $\Bc(H)$ that is stable
under products. A subspace of $\Bc(H)$ is {\it self-adjoint} if it
is stable under adjoints and {\it unital} if it contains the identity
operator $I$. A {\it von Neumann algebra} is a weak* closed self-adjoint
unital operator algebra. We will refer to \cite{Tak} for standard facts
about von Neumann algebras. For example, the {\it commutant} of a von
Neumann algebra $\Mx$ is the von Neumann algebra
$$\Mx' = \{A \in \Bc(H): AB = BA\hbox{ for all }B \in \Mx\}$$
and von Neumann's {\it double commutant theorem} states that every
von Neumann algebra equals the commutant of its commutant, $\Mx = \Mx''$
(\cite{Tak}, Theorem II.3.9).

A {\it dual operator space} is a weak* closed subspace $\Vc$ of $\Bc(H)$;
it is a {\it W*-bimodule} over a von Neumann algebra $\Mx \subseteq \Bc(H)$
if $\Mx\Vc\Mx \subseteq \Vc$. A {\it dual operator system} is a self-adjoint
unital dual operator space.

\begin{defi}\label{quantrel}
(\cite{W6}, Definition 2.1)
A {\it quantum relation} on a von Neumann algebra $\Mx \subseteq \Bc(H)$
is a W*-bimodule over its commutant $\Mx'$, i.e., it is a weak* closed
subspace $\Vc \subseteq \Bc(H)$ satisfying $\Mx'\Vc\Mx' \subseteq \Vc$.
\end{defi}

Quantum relations are effectively representation independent.

\begin{theo}\label{repindep}
(\cite{W6}, Theorem 2.7)
Let $H_1$ and $H_2$ be Hilbert spaces and let $\Mx_1 \subseteq \Bc(H_1)$
and $\Mx_2 \subseteq \Bc(H_2)$ be isomorphic von Neumann algebras. Then
any isomorphism induces a 1-1 correspondence between the quantum relations on
$\Mx_1$ and the quantum relations on $\Mx_2$, and this correspondence
respects the conditions $\Vc \subseteq \Wc$, $\Vc = \Mx'$, $\Vc^* =
\Wc$, and $\overline{\Uc\Vc}^{wk^*} = \Wc$.
\end{theo}

If $H_2 = K \otimes H_1$ then the 1-1 correspondence is given by
$\Vc \leftrightarrow \Bc(K) \overline{\otimes} \Vc$, where
$\overline{\otimes}$ is the {\it normal spatial tensor product}, i.e.,
the weak* closure of the algebraic tensor product in $\Bc(K \otimes H)$.
Checking this case suffices to establish the result, because
any two faithful normal unital representations of a von Neumann algebra become
spatially equivalent when each is tensored with a large enough Hilbert space
(\cite{Tak}, Theorem IV.5.5). 

Quantum relations effectively reduce to classical relations in the atomic
abelian case. Let $V_{xy}$ be the rank one operator $V_{xy}: g \mapsto
\langle g,e_y\rangle e_x$ on $l^2(X)$. Here $\{e_x\}$ is the standard
basis of $l^2(X)$.

\begin{prop}\label{atomiccase}
(\cite{W6}, Proposition 2.2)
Let $X$ be a set and let $\Mx \cong l^\infty(X)$ be the von Neumann
algebra of bounded multiplication operators on $l^2(X)$. If
$R$ is a relation on $X$ then
\begin{eqnarray*}
\Vc_R &=& \{A \in \Bc(l^2(X)): (x,y) \not\in R\quad \Rightarrow\quad
\langle Ae_y, e_x\rangle = 0\}\cr
&=& \overline{\rm span}^{wk^*}\{V_{xy}: (x,y) \in R\}
\end{eqnarray*}
is a quantum relation on $\Mx$; conversely, if $\Vc$ is a
quantum relation on $\Mx$ then
$$R_\Vc = \{(x,y) \in X^2: \langle Ae_y, e_x\rangle\neq 0
\hbox{ for some }A \in \Vc\}$$
is a relation on $X$. The two constructions are inverse to each other.
\end{prop}

Before we state the fundamental result relating measurable relations
on $(X,\mu)$ and quantum relations on $L^\infty(X,\mu)$, we need the
notion of (operator) reflexivity:

\begin{defi}\label{refdef}
(\cite{W6}, Definition 2.14)
A subspace $\Vc \subseteq \Bc(H)$ is {\it (operator) reflexive} if
$$\Vc = \{B \in \Bc(H): P\Vc Q = 0\quad \Rightarrow\quad PBQ = 0\},$$
with $P$ and $Q$ ranging over projections in $\Bc(H)$.
\end{defi}

A simple observation is that if $\Vc$ is a quantum relation on $\Mx$
then $P$ and $Q$ can be restricted to range over projections in $\Mx$
in the preceding definition (\cite{W6}, Proposition 2.15).

For $f \in L^\infty(X,\mu)$ let $M_f \in B(L^2(X,\mu))$ be the
corresponding multiplication operator, $M_f: g \mapsto fg$.

\begin{theo}\label{abelianrel}
(\cite{W6}, Theorem 2.9/Corollary 2.16)
Let $(X,\mu)$ be a finitely decomposable measure space and let
$\Mx \cong L^\infty(X,\mu)$ be the von Neumann algebra
of bounded multiplication operators on $L^2(X,\mu)$.
If $\Rc$ is a measurable relation on $X$ then
$$\Vc_\Rc = \{A \in \Bc(L^2(X,\mu)): (p,q) \not\in \Rc
\quad \Rightarrow \quad M_pAM_q = 0\}$$
is a quantum relation on $\Mx$; conversely, if $\Vc$ is a quantum
relation on $\Mx$ then
$$\Rc_\Vc = \{(p,q): M_pAM_q\neq 0\hbox{ for some }A \in \Vc\}$$
is a measurable relation on $X$. We have $\Rc = \Rc_{\Vc_\Rc}$ for
any measurable relation $\Rc$ on $X$ and $\Vc \subseteq
\Vc_{\Rc_\Vc}$ for any quantum relation $\Vc$ on $\Mx$, with equality if
and only if $\Vc$ is reflexive.
\end{theo}

The following basic tool will be used repeatedly.

\begin{lemma}\label{separation}
(\cite{W6}, Lemma 2.8)
Let $\Vc$ be a quantum relation on a von Neumann algebra $\Mx \subseteq
\Bc(H)$ and let $A \in \Bc(H) - \Vc$. Then there is a pair of projections
$P$ and $Q$ in $\Mx \overline{\otimes} \Bc(l^2) \subseteq \Bc(H \otimes l^2)$
such that
$$P(A \otimes I)Q \neq 0$$
but
$$P(B \otimes I)Q = 0$$
for all $B \in \Vc$.
\end{lemma}

We conclude this brief review of quantum relations with a result that
characterizes them intrinsically. Denote the range projection of $A$
by $[A]$.

\begin{defi}\label{abstractqrel}
(\cite{W6}, Definition 2.24)
Let $\Mx$ be a von Neumann algebra and let $\Pc$ be the set of projections
in $\Mx \overline{\otimes} \Bc(l^2)$, equipped with the restriction of the weak
operator topology. An {\it intrinsic quantum relation} on $\Mx$ is an
open subset $\Rc \subset \Pc \times \Pc$ satisfying
\begin{quote}
\noindent (i) $(0,0)\not\in \Rc$

\noindent (ii) $(\bigvee P_\lambda, \bigvee Q_\kappa) \in \Rc$
$\Leftrightarrow$ some $(P_\lambda, Q_\kappa) \in \Rc$

\noindent (iii) $(P,[BQ]) \in \Rc$ $\Leftrightarrow$ $([B^*P],Q) \in \Rc$
\end{quote}
for all projections $P,Q,P_\lambda,Q_\kappa \in \Pc$
and all $B \in I \otimes \Bc(l^2)$.
\end{defi}

\begin{theo}\label{abstractchar}
(\cite{W6}, Theorem 2.32)
Let $\Mx \subseteq \Bc(H)$ be a von Neumann algebra and let $\Pc$ be
the set of projections in $\Mx \overline{\otimes} \Bc(l^2)$. If $\Vc$ is a
quantum relation on $\Mx$ then
$$\Rc_\Vc =
\{(P,Q) \in \Pc^2: P(A \otimes I)Q \neq 0\hbox{ for some }A \in \Vc\}$$
is an intrinsic quantum relation on $\Mx$;
conversely, if $\Rc$ is an intrinsic quantum relation on $\Mx$ then
$$\Vc_\Rc = \{A \in \Bc(H): (P,Q) \not\in \Rc\quad\Rightarrow\quad
P(A \otimes I)Q = 0\}$$
is a quantum relation on $\Mx$. The two constructions are inverse to
each other.
\end{theo}

\section{Quantum metrics}\label{qmsn}

In this chapter we state our new definition of quantum metric spaces,
present some related definitions, and develop their basic properties.
The principal difference between our approach here and earlier work on
quantum metrics is that we do not attempt to directly model a noncommutative
version of the Lipschitz functions on a metric space, in the way that
C*-algebras and von Neumann algebras generalize $C(X)$ and $L^\infty(X,\mu)$
spaces (though we do eventually attain this goal in Corollary \ref{lipequiv}).
Instead, our definition is based on the idea of mass transport.
Operators on $L^2(X,\mu)$ are graded by the maximum distance that they
displace mass supported in localized regions of $X$, and it is this notion
of displacement that replaces Lipschitz number as the fundamental quantity.
Our motivation comes from quantum information theory, where one wants to
recover quantum mechanically encoded information that may have been
corrupted, i.e., displaced from its original state by the introduction of
errors. (We discuss quantum information theory in Section \ref{qhamming}
and Lipschitz numbers in Chapter \ref{lipschitz}.)

Operators that displace mass only a maximum distance have been widely
used in various parts of mathematics, often under the name of finite
propagation operators. Some representative references are \cite{CGT,
Roe, Sch}. It is possible that our approach could shed new light on
some of this work, or that it could point the way to noncommutative
generalizations.

\subsection{Basic definitions}\label{metricdef}
We adopt the convention that metric spaces can have infinite distances.
Thus, a metric on a set $X$ is a function $d: X^2 \to [0,\infty]$ such
that $d(x,y) = 0$ $\Leftrightarrow$ $x = y$, $d(x,y) = d(y,x)$, and
$d(x,z) \leq d(x,y) + d(y,z)$, for all $x,y,z \in X$, and a pseudometric
is defined similarly with the first condition weakened to $d(x,x) = 0$
for all $x \in X$. If $d(x,y) < \infty$ for all $x$ and $y$
then we say that {\it all distances are finite}.

We model quantum metrics using a special class of filtrations of $\Bc(H)$.
Recall that a dual operator system is a weak* closed self-adjoint unital
subspace of $\Bc(H)$.

\begin{defi}\label{filt}
(a) A {\it W*-filtration} of $\Bc(H)$ is a one-parameter family of dual
operator systems $\Vb = \{\Vc_t\}$, $t \in [0,\infty)$, such that
\begin{quote}
\noindent (i) $\Vc_s\Vc_t \subseteq \Vc_{s+t}$ for all $s,t \geq 0$

\noindent (ii) $\Vc_t = \bigcap_{s > t}\Vc_s$ for all $t \geq 0$.
\end{quote}

\noindent (b) A {\it displacement gauge} on $\Bc(H)$ is a function
$D: \Bc(H) \to [0,\infty]$ such that
\begin{quote}
\noindent (i) $D(I) = 0$

\noindent (ii) $D(aA) \leq D(A)$ for $a \in \Cb$

\noindent (iii) $D(A + B) \leq \max\{D(A),D(B)\}$

\noindent (iv) $D(A^*) = D(A)$

\noindent (v) $D(AB) \leq D(A) + D(B)$

\noindent (vi) $A_\lambda \to A$ weak operator implies
$D(A) \leq \liminf D(A_\lambda)$
\end{quote}
for all $A, B, A_\lambda \in \Bc(H)$ with $\sup \|A_\lambda\| < \infty$.
\end{defi}

In part (a) the appropriate convention for $t = \infty$ is
$\Vc_\infty = \Bc(H)$.

The notions of W*-filtration and displacement gauge are equivalent:

\begin{prop}\label{fildis}
If $\Vb$ is a W*-filtration of $\Bc(H)$ then
$$D_\Vb(A) = \inf\{t: A \in \Vc_t\}$$
(with $\inf \emptyset = \infty$)
is a displacement gauge on $\Bc(H)$. Conversely, if $D$ is a displacement
gauge on $\Bc(H)$ then $\Vb_D = \{\Vc_t^D\}$ with
$$\Vc_t^D = \{A \in \Bc(H): D(A) \leq t\}$$
is a W*-filtration. The two constructions are inverse to each other.
\end{prop}

The proof of this proposition is straightforward. (Recall from the Krein-Smulian
theorem that a subspace of $\Bc(H)$ is weak* closed if and only if it
is boundedly weak operator closed.)
The equivalence between W*-filtrations and dispacement gauges is
not technically substantial, but we nonetheless find it convenient
to be able to pass between the two concepts. Broadly speaking,
W*-filtrations tend to be formally simpler but diplacement gauges may be
more intuitive.

For the sake of notational simplicity we will generally suppress the
subscript and simply write $D$ for the displacement gauge associated to a
W*-filtration $\Vb$.

We now present our definition of quantum metrics.

\begin{defi}
A {\it quantum pseudometric} on a von Neumann algebra $\Mx \subseteq
\Bc(H)$ is a W*-filtration $\Vb$ of $\Bc(H)$ satisfying
$\Mx' \subseteq \Vc_0$; it is a {\it quantum metric} if $\Vc_0 = \Mx'$.
\end{defi}

Note that $\Vc_0$ is automatically a von Neumann algebra, so that
any W*-filtration $\Vb$ is a quantum metric on $\Mx = \Vc_0'$.
If $\Vb$ is a quantum pseudometric on a von Neumann algebra
$\Mx$ then $\Vc_0'$ is a von Neumann algebra contained in $\Mx$, and passing
from $\Mx$ to $\Vc_0'$ is the quantum version of factoring out null
distances in order to turn a pseudometric into a metric.

Also, observe that the filtration condition implies that each $\Vc_t$
is a bimodule over $\Vc_0$. Thus if $\Vb$ is a quantum pseudometric on
$\Mx$ then each $\Vc_t$ is a quantum relation on $\Mx$ (Definition
\ref{quantrel}). We can say: a quantum pseudometric on $\Mx$ is a
one-parameter family of quantum relations on $\Mx$ which satisfy
conditions (i) and (ii) in Definition \ref{filt} (a). This allows us
to immediately deduce from Theorem \ref{repindep} the fact, which we
record now, that quantum pseudometrics are representation independent.
We order quantum pseudometrics by inclusion and write $\Vb \leq
\Wb$ if $\Vc_t \subseteq \Wc_t$ for all $t$.

\begin{theo}\label{metcor}
Let $H_1$ and $H_2$ be Hilbert spaces and let $\Mx_1 \subseteq \Bc(H_1)$
and $\Mx_2 \subseteq \Bc(H_2)$ be isomorphic von Neumann algebras. Then
any isomorphism induces an order
preserving 1-1 correspondence between the quantum (pseudo)metrics
on $\Mx_1$ and the quantum (pseudo)metrics on $\Mx_2$.
\end{theo}

We will give intrinsic characterizations of quantum pseudometrics in
Definition \ref{qdfdef}/Theorem \ref{abch} and
Definition \ref{abspeclip}/Corollary \ref{lipequiv} below.

The interpretation of the $\Vc_t$ as quantum relations corresponds
to the classical fact that
a metric $d$ on a set $X$ gives rise to a family of relations
$$R_t = \{(x,y) \in X^2: d(x,y) \leq t\},$$
one for each value of $t \in [0,\infty)$. The usual metric axioms
can be recast as properties of this family of relations:
\begin{quote}
\noindent (i) $R_0$ is the diagonal relation $\Delta$

\noindent (ii) $R_t = R_t^T$ for all $t$

\noindent (iii) $R_sR_t \subseteq R_{s+t}$ for all $s$ and $t$
\end{quote}
corresponding to the metric axioms $d(x,y) = 0$ $\Leftrightarrow$
$x = y$, $d(x,y) = d(y,x)$, $d(x,z) \leq d(x,y) + d(y,z)$. (In
(ii), $R^T$ is the transpose of $R$.) The
relations are also nested such that $R_t = \bigcap_{s > t} R_s$
for all $t$. Conversely, it is easy to check that any family of
relations with the preceding properties arises from a unique
metric defined by
$$d(x,y) = \inf\{t: (x,y) \in R_t\}.$$
Pseudometrics are characterized similarly, with condition (i)
weakened to
\begin{quote}
\noindent (i') $R_0$ contains the diagonal relation $\Delta$.
\end{quote}
Thus classical metrics and pseudometrics have an alternative
axiomatization as one-parameter families of relations satisfying
the above conditions. A moment's thought shows that replacing
classical relations with quantum relations
yields our definitions of quantum metrics and pseudometrics,
the only (inessential) difference being that we do not explicitly specify
that $\Vc_t$ be a bimodule over $\Mx' \subseteq \Vc_0$, because this
follows anyway from the filtration property.

In light of Proposition \ref{atomiccase}, the above
implies that Definition \ref{filt} should reduce to the classical notions
of pseudometric and metric in the atomic abelian case. The proof of the
following proposition is essentially just this observation.
Recall that $V_{xy} \in \Bc(l^2(X))$ is the
rank one operator $V_{xy}: g \mapsto \langle g,e_y\rangle e_x$.

We emphasize that in the following result, although $X$ is in effect given
the discrete topology, the metric $d$ is completely arbitrary. There is
no restriction on the metrics which can be handled by our theory.

\begin{prop}\label{aa}
Let $X$ be a set and let $\Mx \cong l^\infty(X)$ be the von Neumann algebra
of bounded multiplication operators on $l^2(X)$. If $d$ is a pseudometric
on $X$ then $\Vb_d = \{\Vc_t^d\}$ with
\begin{eqnarray*}
\Vc_t^d &=& \{A \in \Bc(l^2(X)): d(x,y) > t\quad\Rightarrow\quad
\langle Ae_y,e_x\rangle = 0\}\cr
&=& \overline{\rm span}^{wk^*}\{V_{xy}: d(x,y) \leq t\}
\end{eqnarray*}
($t \in [0,\infty)$) is a quantum pseudometric on $\Mx$; conversely,
if $\Vb$ is a quantum pseudometric on $\Mx$ then
$$d_\Vb(x,y) = \inf\{t: \langle Ae_y,e_x\rangle \neq 0\hbox{ for some }
A \in \Vc_t\}$$
(with $\inf \emptyset = \infty$) is a pseudometric on $X$. The two
constructions are inverse to each other, and this correspondence between
pseudometrics and quantum pseudometrics restricts to a correspondence
between metrics and quantum metrics.
\end{prop}

\begin{proof}
Let $d$ be a pseudometric on $X$ and for each $t \in [0,\infty)$
let $R_t = \{(x,y) \in X^2: d(x,y) \leq t\}$. Then $\Vc_t^d =
\Vc_{R_t}$, the quantum relation associated to $R_t$ as in Proposition
\ref{atomiccase}. It follows from Proposition \ref{atomiccase} that
the two expressions for $\Vc_t^d$ agree and that
each $\Vc_t^d$ is a weak* closed linear subspace of $\Bc(l^2(X))$
that contains $\Mx$, and since $R_t^T = R_t$ it follows
that $\Vc_t^d$ is self-adjoint. Thus each $\Vc_t^d$
is a dual operator system and $\Vc_0^d$ contains $\Mx = \Mx'$.
Since $R_sR_t \subseteq R_{s+t}$ (the
classical triangle inequality), it follows
that $\Vc_s^d\Vc_t^d \subseteq \Vc_{s+t}^d$. Finally,
$\Vc_t^d = \bigcap_{s > t}\Vc_s^d$ because
$d(x,y) \leq t$ $\Leftrightarrow$ $d(x,y) \leq s$ for all $s > t$.
Thus $\Vb_d$ is a quantum pseudometric on $\Mx$. If $d$ is
a metric then $R_0$ is the diagonal relation, hence $\Vc_0^d = \Mx'$,
hence $\Vb_d$ is a quantum metric.

Next let $\Vb$ be a quantum pseudometric on $\Mx$. Then $I \in \Vc_0$
and $\langle Ie_x,e_x\rangle \neq 0$ imply $d_\Vb(x,x) = 0$, for any
$x \in X$. We have $d_\Vb(x,y) = d_\Vb(y,x)$ because each $\Vc_t$ is
self-adjoint and
$$\langle Ae_y,e_x\rangle \neq 0\quad\Leftrightarrow\quad
\langle A^*e_x,e_y\rangle \neq 0.$$
The triangle inequality holds by the following argument. Suppose
$d_\Vb(x,y) < s$ and $d_\Vb(y,z) < t$. Then there exist
$A \in \Vc_s$ and $B \in \Vc_t$ such that $\langle Ae_y,e_x\rangle
\neq 0$ and $\langle Be_z,e_y\rangle \neq 0$. Since $\Vc_s$ and
$\Vc_t$ are bimodules over $\Mx$ we then have $M_{e_x}AM_{e_y} \in
\Vc_s$ and $M_{e_y}BM_{e_z} \in \Vc_t$. These are nonzero scalar
multiples of the rank one operators $V_{xy}$ and $V_{yz}$,
so $V_{xy} \in \Vc_s$ and
$V_{yz} \in \Vc_t$. Then $V_{xz} = V_{xy}V_{yz} \in \Vc_{s+t}$,
which implies that $d_\Vb(x,z) \leq s + t$. Taking the infimum over
$s$ and $t$ yields $d_\Vb(x,z) \leq d_\Vb(x,y) + d_\Vb(y,z)$. So
$d_\Vb$ is a pseudometric. By similar
reasoning, if $d_\Vb(x,y) = 0$ then $V_{xy} \in \Vc_s$ for all $s > 0$
and therefore $V_{xy} \in \Vc_0$. So if $\Vb$ is a quantum metric,
i.e., $\Vc_0 = \Mx$, then
$$d_\Vb(x,y) = 0\quad\Rightarrow\quad V_{xy} \in \Vc_0
\quad\Rightarrow\quad x = y,$$
and hence $d_\Vb$ is a metric.

Now let $d$ be a pseudometric on $X$, let $\Vb = \Vb_d$, and let
$\tilde{d} = d_\Vb$. Then $\Vc_t$ is the quantum relation
$\Vc_{R_t}$ associated to the relation $R_t = \{(x,y) \in X^2:
d(x,y) \leq t\}$ and $\tilde{R}_t = \{(x,y) \in X^2: \tilde{d}(x,y)
\leq t\}$ is the relation associated to $\Vc_t$, both as in Proposition
\ref{atomiccase}. So $R_t = \tilde{R}_t$ for all $t$ by Proposition
\ref{atomiccase} and we conclude that $d = \tilde{d}$.

Finally, let $\Vb$ be a quantum pseudometric on $\Mx$, let
$d = d_\Vb$, and let $\tilde{\Vb} = \Vb_d$. Then $R_t = \{(x,y)
\in X^2: d(x,y) \leq t\}$ is the relation associated
to $\Vc_t$ and $\tilde{\Vc}_t$ is the quantum relation associated to
$R_t$, both as in Proposition \ref{atomiccase}.
So $\tilde{\Vc}_t = \Vc_t$ for all $t$ by Proposition \ref{atomiccase} and we
conclude that $\tilde{\Vb} = \Vb$.
\end{proof}

Pseudometrics on $X$ correspond to displacement
gauges on $\Bc(l^2(X))$ satisfying $D(A) = 0$ for all $A \in \Mx \cong
l^\infty(X)$ via the formulas
$$D(A) = \sup\{d(x,y): \langle A e_y, e_x \rangle \neq 0\}$$
and
$$d(x,y) = \inf\{D(A): \langle Ae_y, e_x\rangle \neq 0\}.$$

\subsection{More definitions}
The next definition, of distances between pairs of projections in
$\Mx \overline{\otimes} \Bc(l^2)$, is fundamental for later work.
To some extent it replaces classical distances between points. In the
atomic abelian case it corresponds to the usual notion of minimal
distance between sets, $d(S,T) = \inf\{d(x,y): x \in S, y \in T\}$.
There is a version of the triangle inequality which holds in this
setting; see Definition \ref{qdfdef} (v) below.

We give basic properties of the distance function in Proposition
\ref{distances}; we will show later (Theorem \ref{abch}) that these
properties characterize quantum distance functions.

\begin{defi}\label{projdist}
Let $\Vb$ be a quantum pseudometric on a von Neumann algebra $\Mx
\subseteq \Bc(H)$. We define the {\it distance} between any projections
$P$ and $Q$ in $\Mx \overline{\otimes} \Bc(l^2)$ by
\begin{eqnarray*}
\rho_\Vb(P,Q) &=& \inf\{t: P(A \otimes I)Q \neq 0\hbox{ for some }
A \in \Vc_t\}\cr
&=& \inf\{D(A): A \in\Bc(H)\hbox{ and }P(A \otimes I)Q \neq 0\}
\end{eqnarray*}
(with $\inf \emptyset = \infty$).
\end{defi}

By identifying $\Mx$ with $\Mx \otimes I \subseteq \Mx \overline{\otimes}
\Bc(l^2)$ we can consider the restriction of $\rho_\Vb$ to projections in
$\Mx$. Equivalently, for projections $P$ and $Q$ in $\Mx$ we have
\begin{eqnarray*}
\rho_\Vb(P,Q) &=& \inf\{t: PAQ \neq 0\hbox{ for some }
A \in \Vc_t\}\cr
&=& \inf\{D(A): A \in \Bc(H)\hbox{ and } PAQ \neq 0\}.
\end{eqnarray*}

For the sake of notational simplicity we will generally suppress the
subscript and simply write $\rho$ for the distance function
associated to a quantum pseudometric $\Vb$.

It will be convenient later (in Theorem \ref{abch}) to introduce the
following terminology. Recall that $[A]$ denotes
the range projection of $A$.

\begin{defi}\label{qdfdef}
Let $\Mx$ be a von Neumann algebra and let $\Pc$ be the set of
projections in $\Mx \overline{\otimes} \Bc(l^2)$. A {\it quantum
distance function} on $\Mx$ is a function $\rho: \Pc^2 \to [0,\infty]$
such that
\begin{quote}
\noindent (i) $\rho(P,0) = \infty$

\noindent (ii) $PQ \neq 0$ $\Rightarrow$ $\rho(P,Q) = 0$

\noindent (iii) $\rho(P,Q) = \rho(Q,P)$

\noindent (iv) $\rho(P \vee Q,R) = \min\{\rho(P,R),\rho(Q,R)\}$

\noindent (v) $\rho(P,R) \leq \rho(P,Q) + \sup\{\rho(\tilde{Q},R): Q\tilde{Q} \neq 0\}$

\noindent (vi) $\rho(P, [BQ]) = \rho([B^*P],Q)$

\noindent (vii) if $P_\lambda \to P$ and $Q_\lambda \to Q$ weak
operator then $\rho(P,Q) \geq \limsup \rho(P_\lambda,Q_\lambda)$
\end{quote}
for all projections $P,Q,R,P_\lambda,Q_\lambda \in
\Mx\overline{\otimes}\Bc(l^2)$ and all $B \in I \otimes \Bc(l^2)$.
In (v) we take the supremum over all projections $\tilde{Q}$ such
that $Q\tilde{Q} \neq 0$.
\end{defi}

Property (iv) can be strengthened to $\rho(\bigvee P_\lambda,
\bigvee Q_\kappa) = \inf \rho(P_\lambda, Q_\kappa)$. This can
easily be proven directly for $\rho_\Vb$, or it can be deduced from the
stated properties. (Property (iv) implies that $\rho$ is monotone
in the sense that $P \leq \tilde{P}$ implies $\rho(\tilde{P},Q)
\leq \rho(P,Q)$; this plus (iii) yields the inequality $\leq$. For the
reverse inequality, first check that $\rho(\bigvee P_\lambda,
\bigvee Q_\kappa) = \inf \rho(P_\lambda, Q_\kappa)$ holds for
finite joins using (iii) and (iv), and then take limits using
(vii) to pass to arbitrary joins.)

\begin{prop}\label{distances}
Let $\Vb$ be a quantum pseudometric on a von Neumann algebra $\Mx
\subseteq \Bc(H)$. Then $\rho_\Vb$ is a quantum distance function.
\end{prop}

\begin{proof}
Verification of properties (i) through (iv) is easy. Property
(vi) holds because
\begin{eqnarray*}
P(A\otimes I)[BQ] \neq 0&\Leftrightarrow& P(A\otimes I)BQ \neq 0\cr
&\Leftrightarrow& PB(A\otimes I)Q \neq 0\cr
&\Leftrightarrow&[B^*P](A\otimes I)Q \neq 0.
\end{eqnarray*}
One can check property (vii) using the fact that
$P_\lambda \to P$ and $Q_\lambda \to Q$ weak operator implies
$P_\lambda \to P$ and $Q_\lambda \to Q$ strong operator (so
$P_\lambda(A\otimes I)Q_\lambda = 0$ for all $\lambda$ implies
$P(A \otimes I)Q = 0$).

For (v) assume $\rho(P,Q) < \infty$, let
$\epsilon > 0$, and find $A \in \Bc(H)$ such that $P(A \otimes I)Q
\neq 0$ and $D(A) \leq \rho(P,Q) + \epsilon$. Then $Q\tilde{Q} \neq 0$ where
$\tilde{Q}$ is the projection onto the closure of
$$(\Mx'\otimes I)({\rm ran}((A^*\otimes I)P)$$
and $\tilde{Q} \in \Mx \overline{\otimes} \Bc(l^2)$ since its range is
invariant for $\Mx'\otimes I = (\Mx \overline{\otimes} \Bc(l^2))'$.
Now assume $\rho(\tilde{Q}, R) < \infty$ and find $C \in \Bc(H)$ such that
$\tilde{Q}(C \otimes I)R \neq 0$ and $D(C) \leq \rho(\tilde{Q},R) + \epsilon$.
It follows that $R(C^* \otimes I)\tilde{Q} \neq 0$, so
$R(C^*B^*A^* \otimes I)P \neq 0$ for some $B \in \Mx'$, and hence
$P(ABC\otimes I)R \neq 0$ for some $B \in \Mx'$. Then
$$\rho(P,R) \leq D(ABC) \leq D(A) + D(B) + D(C)
\leq \rho(P,Q) + \rho(\tilde{Q},R) + 2\epsilon.$$
Taking $\epsilon \to 0$ yields the desired inequality.
\end{proof}

We now show that distance between projections is representation
independent and that the W*-filtration $\Vb$ can be recovered from
the quantum distance function $\rho_\Vb$. Generally speaking, this means
that any representation independent notion defined in terms of
$\Vb$ will have an equivalent definition in terms of $\rho_\Vb$.

\begin{prop}
Let $\pi_1: \Mx \to \Bc(H_1)$ and $\pi_2: \Mx \to \Bc(H_2)$ be
faithful normal unital representations of a von Neumann algebra $\Mx$,
let $\Vb_1$ be a quantum pseudometric on $\pi_1(\Mx)$, and let $\Vb_2$ be
the corresponding quantum pseudometric on $\pi_2(\Mx)$ as in Theorem
\ref{metcor}. Then the quantum distance functions $\rho_{\Vb_1}$ and
$\rho_{\Vb_2}$ on projections in $\Mx \overline{\otimes} \Bc(l^2)$
associated to $\Vb_1$ and $\Vb_2$ are equal.
\end{prop}

\begin{proof}
As in the proof of Theorem \ref{repindep}, it is sufficient
to consider the case where $\pi_1 = {\rm id}$,
$H_2 = K \otimes H_1$, and $\pi_2: A \mapsto I_K \otimes A$.
Given projections $P,Q \in \Mx \overline{\otimes}\Bc(l^2)$, the
corresponding projections in $\pi_2(\Mx) \overline{\otimes}\Bc(l^2)$
are then $I_K \otimes P$ and $I_K \otimes Q$. Also, if $\Vb_1 =
\{\Vc^1_t\}$ then $\Vb_2 = \{\Vc^2_t\}$ with $\Vc^2_t = \Bc(K)
\overline{\otimes} \Vc^1_t$. Now if $P(A \otimes I_{l^2})Q \neq 0$
for some $A \in \Vc^1_t$ then $I_K \otimes A \in \Vc^2_t$ and
$$(I_K \otimes P)(I_K \otimes A \otimes I_{l^2})(I_K \otimes Q) \neq 0;$$
conversely, if $P(A \otimes I_{l^2})Q = 0$ for all $A \in \Vc^1_t$ then
$$(I_K \otimes P)(B \otimes A \otimes I_{l^2})(I_K \otimes Q) = 0$$
for all $A \in \Vc^1_t$ and all $B \in \Bc(K)$ and hence
$$(I_K \otimes P)(\tilde{A} \otimes I_{l^2})(I_K \otimes Q) = 0$$
for all $\tilde{A} \in \Vc^2_t$. So $\rho_{\Vb_1}(P,Q) =
\rho_{\Vb_2}(I_K \otimes P, I_K \otimes Q)$.
\end{proof}

\begin{prop}\label{recover}
Let $\Vb$ be a quantum pseudometric on a von Neumann algebra $\Mx
\subseteq \Bc(H)$. Then
$$\Vc_t = \{A \in \Bc(H): \rho(P,Q) > t\quad\Rightarrow\quad
P(A\otimes I)Q = 0\}$$
for all $t \in [0,\infty)$, with $P$ and $Q$ ranging over
projections in $\Mx \overline{\otimes} \Bc(l^2)$.
\end{prop}

\begin{proof}
Let $\tilde{\Vc}_t = \{A \in \Bc(H): \rho(P,Q) > t$ $\Rightarrow$
$P(A\otimes I)Q = 0\}$. Then $\Vc_t \subseteq \tilde{\Vc}_t$ is
immediate from the definition of $\rho$.
Conversely, let $A \in \Bc(H) -\Vc_t$; then we must have
$A \not\in \Vc_s$ for some $s > t$ and by Lemma
\ref{separation} there exist projections $P,Q \in
\Mx \overline{\otimes} \Bc(l^2)$ such that $P(B \otimes I)Q = 0$
for all $B \in \Vc_s$, and hence $\rho(P,Q) \geq s > t$, but
$P (A \otimes I)Q \neq 0$. This shows that $A \not\in \tilde{\Vc}_t$.
We conclude that $\Vc_t = \tilde{\Vc}_t$.
\end{proof}

A W*-filtration $\Vb$ generally cannot be recovered from the restriction
of $\rho_\Vb$ to projections in $\Mx$; see Example \ref{counterdiam}.

Next we introduce a von Neumann algebra which detects the
existence of infinite distances.

\begin{defi}
Let $\Vb = \{\Vc_t\}$ be a W*-filtration. For $t \in (0,\infty]$ we
define
$$\Vc_{<t} = \overline{\bigcup_{s < t}\Vc_s}^{wk^*}.$$
In particular,
$$\Vc_{<\infty} = \overline{\bigcup_{t\geq 0} \Vc_t}^{wk^*} =
\overline{\{A: D(A) < \infty\}}^{wk^*}.$$
\end{defi}

Note that $\Vc_{<\infty}$ is a von Neumann algebra that contains
$\Mx' \subseteq \Vc_0$. Thus, it is the commutant of
a von Neumann subalgebra of $\Mx$. We show next that in the atomic
abelian case $\Vc_{<\infty}$ corresponds to the equivalence relation
$x \sim y$ $\Leftrightarrow$ $d(x,y) < \infty$ on $X$. The condition
$\Vc_{<\infty} = \Bc(H)$ is the quantum equivalent of all distances
being finite.

\begin{prop}
Let $X$ be a set and let $\Mx \cong l^\infty(X)$ be the von Neumann algebra
of bounded multiplication operators on $l^2(X)$. Also let $d$ be a
pseudometric on $X$ and define $\Vb_d$ as in Proposition \ref{aa}.
Then $\Vc^d_{<\infty} = \Mx_\infty'$ where
$$\Mx_\infty = \{M_f: f \in l^\infty(X)\quad\hbox{ and }\quad
d(x,y) < \infty \Rightarrow f(x) = f(y)\} \subseteq \Mx.$$
In particular, all distances in $X$ are finite if and only if
$\Vc_{<\infty} = \Bc(l^2(X))$.
\end{prop}

\begin{proof}
$\Mx_\infty'$ is the von Neumann algebra generated by the rank one operators
$V_{xy}$ such that $d(x,y) < \infty$. Since $\Vc_t^d =
\overline{\rm span}^{wk^*}\{V_{xy}: d(x,y) \leq t\}$ it follows that
$\Vc^d_{<\infty} = \overline{\bigcup \Vc_t}^{wk^*}$
is also generated by $\{V_{xy}: d(x,y) < \infty\}$.
So $\Vc^d_{<\infty} = \Mx_\infty'$.
\end{proof}

The condition $\Vc_{<\infty} = \Bc(H)$ can also be characterized
in terms of distances between projections (and hence it is
representation independent). Say that the projections
$P$ and $Q$ in $\Mx \overline{\otimes} \Bc(l^2)$ are {\it unlinkable} if there
exist projections $\tilde{P},\tilde{Q} \in I \otimes \Bc(l^2)$ with
$P \leq \tilde{P}$, $Q \leq \tilde{Q}$, and $\tilde{P}\tilde{Q} = 0$;
otherwise they are {\it linkable}. The motivation for these terms comes
from the following proposition, which shows that $P$ and $Q$ are linkable
if and only if there is an operator $A \in \Bc(H)$ such that
$(A \otimes I)v$ is not orthogonal to the range of $P$, for some vector
$v$ in the range of $Q$.

\begin{prop}\label{findist}
Let $\Vb$ be a quantum pseudometric on a von Neumann algebra $\Mx
\subseteq \Bc(H)$. Two projections $P$ and $Q$ in $\Mx \overline{\otimes}
\Bc(l^2)$ are linkable if and only if there exists $A \in \Bc(H)$ such
that $P(A \otimes I)Q \neq 0$. If $P$ and $Q$ are unlinkable
then $\rho(P,Q) = \infty$. The following are equivalent:
\begin{quote}
\noindent (i) $\Vc_{<\infty} = \Bc(H)$

\noindent (ii) $\rho(P,Q) < \infty$ for any linkable projections $P$ and $Q$
in $\Mx \overline{\otimes} \Bc(l^2)$

\noindent (iii) $\rho(P,Q) < \infty$ for any nonzero projections $P$ and $Q$
in $\Mx$.
\end{quote}
\end{prop}

\begin{proof}
Suppose $P \leq \tilde{P}$ and $Q \leq \tilde{Q}$ where
$\tilde{P},\tilde{Q}$ are projections in $I \otimes \Bc(l^2)$ which satisfy
$\tilde{P}\tilde{Q} = 0$. Then for any $A \in \Bc(H)$ we have $A \otimes I
\in \Bc(H) \otimes I = (I \otimes \Bc(l^2))'$ and therefore
$$P(A \otimes I)Q = P(\tilde{P}(A \otimes I)\tilde{Q})Q =
P((A \otimes I)\tilde{P}\tilde{Q})Q = 0.$$
Conversely, if $P(A \otimes I)Q = 0$ for all $A \in \Bc(H)$ then the
projections $\tilde{P}$ and $\tilde{Q}$ onto the closures of
$(\Bc(H) \otimes I)({\rm ran}(P))$ and $(\Bc(H) \otimes I)({\rm ran}(Q))$
satisfy $P \leq \tilde{P}$, $Q \leq \tilde{Q}$, $\tilde{P}\tilde{Q} = 0$, and
$$\tilde{P},\tilde{Q} \in (\Bc(H) \otimes I)' = I \otimes \Bc(l^2),$$
so $P$ and $Q$ are unlinkable.

It immediately follows that $\rho(P,Q) = \infty$ for any
unlinkable projections $P$ and $Q$ in $\Mx \overline{\otimes} \Bc(l^2)$.

(i) $\Rightarrow$ (ii): Suppose $\Vc_{<\infty} = \Bc(H)$ and
$\rho(P,Q) = \infty$. Then $P(A \otimes I)Q = 0$ for all
$A \in \Bc(H)$ with $D(A) < \infty$, and hence, by taking weak* limits,
for all $A \in \Bc(H)$. Thus $P$ and $Q$ are unlinkable.

(ii) $\Rightarrow$ (iii): Trivial.

(iii) $\Rightarrow$ (i): If $\Vc_{<\infty} \neq \Bc(H)$ then there
exists a nontrivial projection $P$ in its commutant. Then $P$ and
$I - P$ are both nonzero but $\rho(P,I-P) = \infty$.
\end{proof}

Next we present a variety of basic definitions that can be made
directly in terms of the W*-filtration $\Vb$, followed by a series
of propositions giving basic facts about them.

\begin{defi}\label{various}
Let $\Vb$ be a quantum pseudometric on a von Neumann algebra $\Mx
\subseteq \Bc(H)$.

\noindent (a) The {\it diameter} of $\Vb$ is the quantity
$${\rm diam}(\Vb) = \inf\{t: \Vc_t = \Bc(H)\}
= \sup\{D(A): A \in \Bc(H)\}$$
(with $\inf \emptyset = \infty$).

\noindent (b) For $\epsilon > 0$ and $P$ a projection in
$\Mx$, the {\it open $\epsilon$-neighborhood
of $P$} is the projection $(P)_\epsilon$ onto
$$\overline{\Vc_{<\epsilon}({\rm ran}(P))} =
\bigvee_{t < \epsilon} \overline{\Vc_t({\rm ran}(P))}.$$
Equivalently, $(P)_\epsilon = \bigvee \{[AP]: D(A) < \epsilon\}$.

\noindent (c) The {\it closure} of a projection $P \in \Mx$ is the
projection $\overline{P}$ onto
$$\bigcap_{t > 0} \overline{\Vc_t({\rm ran}(P))}$$
and $P$ is {\it closed} if $P = \overline{P}$.

\noindent (d) $\Vb$ is {\it uniformly discrete} if there exists
$t > 0$ such that $\Vc_t = \Vc_0$, or equivalently there exists
$t > 0$ such that
$$D(A) > 0\quad\Rightarrow\quad D(A) \geq t.$$

\noindent (e) $\Vb$ is a {\it path quantum pseudometric} if
$$\bigcap_{\epsilon > 0}
\overline{\Vc_{s+\epsilon}\Vc_{t+\epsilon}}^{wk^*} = \Vc_{s + t}$$
for all $s,t \geq 0$.
\end{defi}

We will show below (Propositions \ref{nbhds} and \ref{closures}) that
$(P)_\epsilon$ and $\overline{P}$ belong to $\Mx$.

The notion of an open $\epsilon$-neighborhood immediately suggests a
definition of Hausdorff distance: if $P$ and $Q$ are projections in $\Mx$,
we can define their {\it Hausdorff distance} to be
$$\inf\{\epsilon: P \leq (Q)_\epsilon\quad{\rm and}\quad
Q \leq (P)_\epsilon\}.$$
Observe that this is an actual pseudometric; it satisfies the triangle
inequality because $((P)_\epsilon)_\delta \subseteq (P)_{\epsilon + \delta}$.
If $\Vb$ and $\Wb$ are quantum pseudometrics on von Neumann algebras
$\Mx \subseteq \Bc(H)$ and $\Nc \subseteq \Bc(K)$ then we can define
their {\it Gromov-Hausdorff distance} to be the infimum of the
Hausdorff distance between $I_H$ and $I_K$ over all quantum
pseudometrics on $\Mx \oplus \Nc$ that restrict to $\Vb$ on $\Mx$
and $\Wb$ on $\Nc$, i.e., W*-filtrations of $\Bc(H \oplus K)$ whose
intersection with $\Bc(H) \subseteq \Bc(H \oplus K)$ equals $\Vb$ and
whose intersection with $\Bc(K) \subseteq \Bc(H \oplus K)$ equals $\Wb$.
(Cf.\ Definition \ref{easycon} (b) and the discussion following it,
which shows that the Gromov-Hausdorff distance is always at most
$\max\{{\rm diam}(\Vb), {\rm diam}(\Wb)\}/2$.) This too is a pseudometric,
the key observation here being that if we are given a W*-filtration on
$\Bc(H \oplus H')$ which restricts to $\Vb$ and $\Vb'$, and a W*-filtration
on $\Bc(H' \oplus H'')$ which restricts to $\Vb'$ and $\Vb''$, then after
embedding both into $\Bc(H \oplus H' \oplus H'')$ their meet (see Definition
\ref{easycon} (c) below) restricts to a W*-filtration on $\Bc(H \oplus H'')$
which restricts to $\Vb$ and $\Vb''$. However, generally speaking
this does not seem to be a good tool for analyzing convergence of quantum
metrics (e.g., the analog of Theorem \ref{cvgs} below fails). A better
candidate may be the notion of local convergence introduced
below in Definition \ref{localconv}.

We first observe that the preceding definitions reduce to the corresponding
classical notions in the atomic abelian case. All parts of the next
proposition are straightforward consequences of the characterization
$\Vc_t^d = \overline{\rm span}^{wk^*}\{V_{xy}: d(x,y) \leq t\}$
(Proposition \ref{aa}). Denote the characteristic function of the
set $S$ by $\chi_S$.

\begin{prop}
Let $X$ be a set and let $\Mx \cong l^\infty(X)$ be the von Neumann algebra
of bounded multiplication operators on $l^2(X)$. Also let $d$ be a
pseudometric on $X$ and define $\Vb_d$ as in Proposition \ref{aa}.

\noindent (a) ${\rm diam}(\Vb_d) = \sup\{d(x,y): x,y \in X\}$.

\noindent (b) For any $S \subseteq X$ we have
$(M_{\chi_S})_\epsilon =  M_{\chi_{N_\epsilon(S)}}$
where $N_\epsilon(S) = \{x \in X: d(x,S) < \epsilon\}$.

\noindent (c) For any $S \subseteq X$ the closure of $M_{\chi_S}$
is $M_{\chi_{\overline{S}}}$.

\noindent (d) $\Vb_d$ is uniformly discrete if and only if there
exists $t > 0$ such that $d(x,y) > 0$ $\Rightarrow$ $d(x,y) \geq t$.

\noindent (e) $\Vb_d$ is a path quantum pseudometric if and only if
$d(x,y)$ is the infimum of the lengths of paths from $x$ to $y$ in
the completion of $X$, for all $x,y \in X$.
\end{prop}

Now we establish some basic properties of the concepts introduced in
Definition \ref{various}. In particular, we provide some
alternative characterizations in terms of projection distances.

\begin{prop}\label{diameter}
Let $\Vb$ be a quantum pseudometric on a von Neumann algebra $\Mx
\subseteq \Bc(H)$. Then ${\rm diam}(\Vb) = \sup\{\rho(P,Q): P$ and $Q$
are linkable projections in $\Mx \overline{\otimes} \Bc(l^2)\}$.
\end{prop}

\begin{proof}
Let $t \geq 0$ and suppose $\Vc_t \neq \Bc(H)$. Let $A \in \Bc(H) - \Vc_t$;
then by Lemma \ref{separation} there exist projections $P$ and
$Q$ in $\Mx \overline{\otimes} \Bc(l^2)$ such that $P(A \otimes I)Q \neq 0$
but $P(B \otimes I)Q = 0$ for all $B \in \Vc_t$. Thus $\rho(P,Q) \geq t$,
and $P(A \otimes I)Q \neq 0$ implies that $P$ and $Q$ are linkable
(Proposition \ref{findist}). Since ${\rm diam}(\Vb) =
\sup\{t: \Vc_t \neq \Bc(H)\}$, taking the supremum over $t$ yields the
inequality $\leq$. For the reverse inequality, let $P$ and $Q$ be
linkable projections in $\Mx \overline{\otimes} \Bc(l^2)$. Then there
exists $A \in \Bc(H)$ such that $P(A \otimes I)Q \neq 0$ by
Proposition \ref{findist}, and we have $\rho(P,Q) \leq D(A)
\leq {\rm diam}(\Vb)$. Taking the supremum over $P$ and $Q$ yields
the inequality $\geq$.
\end{proof}

Restricting the supremum in Proposition \ref{diameter} only to nonzero
projections in $\Mx$ would not suffice in general; see Example
\ref{counterdiam}.

\begin{prop}\label{nbhds}
Let $\Vb$ be a quantum pseudometric on a von Neumann algebra
$\Mx \subseteq \Bc(H)$, let $\epsilon > 0$, and let $P$ be a
projection in $\Mx$. Then
$$(P)_\epsilon = I - \bigvee\{Q: \rho(P,Q) \geq \epsilon\} \in \Mx$$
with $Q$ ranging over projections in $\Mx$. The open $\epsilon$-neighborhood
of the join of any family of projections in $\Mx$
equals the join of their open $\epsilon$-neighborhoods.
\end{prop}

\begin{proof}
First we check that $(P)_\epsilon \in \Mx$. For
each $t < \epsilon$ the subspace $\overline{\Vc_t({\rm ran}(P))}
\subseteq H$ is invariant for $\Mx' \subseteq
\Vc_0$. Hence the join of these subspaces, which is the range
of $(P)_\epsilon$, is also invariant for $\Mx'$, and this shows
that $(P)_\epsilon \in \Mx$.

Define $\tilde{P} = I - \bigvee\{Q \in \Mx: Q$ is a projection and
$\rho(P,Q) \geq \epsilon\}$. If $\rho(P,Q) \geq
\epsilon$ and $D(A) < \epsilon$ then $QAP = 0$;
it follows that $Q(P)_\epsilon = 0$, and this shows that
$(P)_\epsilon \leq \tilde{P}$. Conversely, let $Q = I - (P)_\epsilon$.
Then $QAP = 0$ for all $A \in \Bc(H)$ with $D(A) < \epsilon$,
so $\rho(P,Q) \geq \epsilon$, and this plus the result of the last
paragraph shows that $Q$ belongs to the join
used to define $\tilde{P}$. Thus $I - (P)_\epsilon \leq
I - \tilde{P}$. We conclude that $(P)_\epsilon = \tilde{P}$.

The last assertion follows from the fact that
$$\left[ A\cdot \bigvee P_\lambda\right] = \bigvee [AP_\lambda]$$
for any family of projections $\{P_\lambda\}$ and any $A \in \Bc(H)$.
Taking the join over $D(A) < \epsilon$ yields the open $\epsilon$-neighborhood
of $\bigvee P_\lambda$ on the left and the join of the open
$\epsilon$-neighborhoods of the $P_\lambda$ on the right.
\end{proof}

\begin{prop}\label{closures}
Let $\Vb$ be a quantum pseudometric on a von Neumann algebra
$\Mx \subseteq \Bc(H)$ and let $P$ be a projection in $\Mx$. Then
$$\overline{P} = \bigwedge_{\epsilon > 0} (P)_\epsilon
= I - \bigvee\{Q: \rho(P,Q) > 0\} \in \Mx$$
with $Q$ ranging over projections in $\Mx$.
We have $\rho(P,Q) = \rho(\overline{P},Q)$ for any projection $Q$ in
$\Mx$. The projection
$\overline{P}$ is closed, as is $I - (P)_\epsilon$ for any $\epsilon > 0$.
The meet of any family of closed projections is closed.
\end{prop}

\begin{proof}
For any $\epsilon > 0$ we have
$$\bigvee_{t < \epsilon} \overline{\Vc_t({\rm ran}(P))}
\leq \overline{\Vc_\epsilon({\rm ran}(P))}
\leq \bigvee_{t < 2\epsilon} \overline{\Vc_t({\rm ran}(P))};$$
taking the meet over $\epsilon > 0$, the first inequality yields
$\bigwedge_{\epsilon > 0} (P)_\epsilon \leq \overline{P}$ and the
second yields $\overline{P} \leq \bigwedge_{\epsilon > 0}(P)_\epsilon$.
So $\overline{P} = \bigwedge_{\epsilon > 0} (P)_\epsilon$.
Then by Proposition \ref{nbhds}
\begin{eqnarray*}
I - \overline{P} &=& I - \bigwedge_{\epsilon > 0} (P)_\epsilon\cr
&=& \bigvee_{\epsilon > 0} (I - (P)_\epsilon)\cr
&=& \bigvee_{\epsilon > 0}\bigvee\{Q: \rho(P,Q) \geq \epsilon\}\cr
&=& \bigvee\{Q: \rho(P,Q) > 0\},
\end{eqnarray*}
which proves the second formula for $\overline{P}$.

It is clear that $\rho(\overline{P},Q) \leq \rho(P,Q)$ since $P \leq
\overline{P}$. To prove the reverse inequality suppose
$QA\overline{P} \neq 0$ and let $\delta > 0$.
Since $\overline{P} \leq (P)_\delta$ there must exist $B \in \Vc_{2\delta}$
such that $QABP \neq 0$.
But $D(AB) \leq D(A) + 2\delta$, so taking the infimum over $A$
and letting $\delta \to 0$ yields
$\rho(P,Q) \leq \rho(\overline{P},Q)$, as desired. We conclude that
$\rho(\overline{P},Q) = \rho(P,Q)$. It follows that
$$\overline{\overline{P}} = I - \bigvee\{Q: \rho(\overline{P},Q) > 0\}
= I - \bigvee\{Q: \rho(P,Q) > 0\} = \overline{P},$$
so that $\overline{P}$ is closed.

Next let $\epsilon > 0$; then
$I - (P)_\epsilon = \bigvee\{Q: \rho(P,Q) \geq \epsilon\}$
implies $\rho(I-(P)_\epsilon, P) \geq \epsilon$, and hence
$\rho(\overline{I - (P)_\epsilon}, P) \geq \epsilon$, so
$\overline{I - (P)_\epsilon}$ belongs to the join that defines
$I - (P)_\epsilon$. Since $I - (P)_\epsilon \leq \overline{I - (P)_\epsilon}$
is trivial, this shows that $I - (P)_\epsilon = \overline{I - (P)_\epsilon}$,
i.e., $I - (P)_\epsilon$ is closed.

Finally, let $\{P_\lambda\}$ be any family of closed projections in
$\Mx$ and let $P = \bigwedge P_\lambda$. Let
$$\tilde{Q} = \bigvee\{Q: \rho(P_\lambda, Q) > 0\hbox{ for some }\lambda\}.$$
Every $Q$ that contributes to this join satisfies
$\rho(P, Q) > 0$ and hence is orthogonal to $\overline{P}$. So
$\overline{P} \leq I - \tilde{Q}$. However, since each $P_\lambda$
is closed we have $I - P_\lambda = \bigvee \{Q: \rho(P_\lambda, Q) > 0\}
\leq \tilde{Q}$ for all $\lambda$, and hence $I - \tilde{Q}
\leq \bigwedge P_\lambda = P$. Thus
$\overline{P} \leq P$, and we conclude that $P$ is closed.
\end{proof}

\begin{prop}\label{gaps}
Let $\Vb$ be a quantum pseudometric on a von Neumann algebra
$\Mx \subseteq \Bc(H)$. Then $\Vb$ is uniformly discrete if and
only if there exists $t > 0$ such that
$$\rho(P,Q) > 0\quad\Rightarrow\quad \rho(P,Q) \geq t,$$
with $P$ and $Q$ ranging over projections in $\Mx\overline{\otimes} \Bc(l^2)$.
\end{prop}

\begin{proof}
If $\Vc_0 = \Vc_t$ for some $t > 0$ then it immediately follows from
the definition of $\rho(P,Q)$ that it cannot
lie in the interval $(0,t)$. Conversely, suppose $\Vc_0 \neq \Vc_t$ for
all $t > 0$ and fix $t$. Then there must exist $s \in (0,t)$ such that
$\Vc_s \neq \Vc_t$ because otherwise $\Vc_0 = \bigcap_{s > 0} \Vc_s$
would equal $\Vc_t$. Letting $A \in \Vc_t - \Vc_s$, Lemma
\ref{separation} then yields the existence of projections
$P,Q \in \Mx\overline{\otimes} \Bc(l^2)$ such that $P(A \otimes I)Q
\neq 0$ but $P(B \otimes I)Q = 0$ for all $B \in \Vc_s$. We conclude
that $0 < s \leq \rho(P,Q) \leq t$. Since $t > 0$ was arbitrary, we
have shown that $\inf\{\rho(P,Q): \rho(P,Q) > 0\} = 0$.
\end{proof}

\subsection{The abelian case}
Measurable metric spaces were introduced in \cite{W0} and have
subsequently been studied in connection with derivations
\cite{W1, W2, W3} and local Dirichlet forms \cite{H, H1, H2}.
We recall the basic definition:

\begin{defi}\label{measpm}
(\cite{W4}, Definition 6.1.3)
Let $(X,\mu)$ be a finitely decomposable measure space and let
$\Pc$ be the set of nonzero projections in $L^\infty(X,\mu)$.
A {\it measurable pseudometric} on $(X,\mu)$ is
a function $\rho: \Pc^2 \to [0,\infty]$ such that
\begin{quote}
\noindent (i) $\rho(p, p) = 0$

\noindent (ii) $\rho(p,q) = \rho(q,p)$

\noindent (iii) $\rho(\bigvee p_\lambda, \bigvee q_\kappa) =
\inf_{\lambda,\kappa} \rho(p_\lambda, q_\kappa)$

\noindent (iv) $\rho(p,r) \leq \sup_{q' \leq q} (\rho(p,q') + \rho(q',r))$
\end{quote}
for all $p,q,r,p_\lambda,q_\kappa \in \Pc$. It is a
{\it measurable metric} if for all
disjoint $p$ and $q$ there exist nets $\{p_\lambda\}$ and $\{q_\lambda\}$
such that $p_\lambda \to p$ and $q_\lambda \to q$ weak*
and $\rho(p_\lambda, q_\lambda) > 0$ for all $\lambda$.
\end{defi}

If either $p$ or $q$ is (or both are) the zero projection then the
appropriate convention is $\rho(p,q) = \infty$. (Note that in the
measurable triangle inequality, property (iv), $q'$ ranges over
nonzero projections.) Basic properties of measurable metrics are
summarized in Section 1.5 of \cite{W6}.

In the atomic case measurable metrics reduce to pointwise metrics
in the expected way:

\begin{prop}\label{atomicmm}
(\cite{W4}, Proposition 6.1.4)
Let $\mu$ be counting measure on a set $X$. If $d$ is a pseudometric
on $X$ then
$$\rho_d(\chi_S,\chi_T) = \inf\{d(x,y): x \in S,y \in T\}$$
is a measurable pseudometric on $X$; conversely, if $\rho$ is a
measurable pseudometric on $X$ then
$$d_\rho(x,y) = \rho(e_x,e_y)$$
is a pseudometric on $X$. The two constructions are inverse to
each other, and this correspondence between pseudometrics and
measurable pseudometrics restricts to a correspondence between
metrics and measurable metrics.
\end{prop}

We omit the easy proof. (If $d$ is a metric, we show that $\rho_d$
is a measurable metric by approximating disjoint positive measure
subsets $S,T \subseteq X$ by finite subsets.)

The relation between quantum metrics and measurable metrics is
explained in the following theorem.

\begin{theo}\label{qmeasmet}
Let $(X,\mu)$ be a finitely decomposable measure space and let
$\Mx \cong L^\infty(X,\mu)$ be the von Neumann algebra of bounded
multiplication operators on $L^2(X,\mu)$. If $\rho$ is a measurable
pseudometric on $X$ then $\Vb_\rho = \{\Vc_t^\rho\}$ with
$$\Vc_t^\rho = \{A \in \Bc(L^2(X,\mu)): \rho(p,q) > t\quad\Rightarrow\quad
M_pAM_q = 0\}$$
is a quantum pseudometric on $\Mx$; conversely, if $\Vb$ is a
quantum pseudometric on $\Mx$ then
$$\rho_\Vb(p,q) = \inf\{D(A): M_pAM_q \neq 0\}$$
is a measurable pseudometric on $\Mx$. We have $\rho = \rho_{\Vb_\rho}$
for any measurable pseudometric $\rho$ on $X$ and $\Vb \leq
\Vb_{\rho_\Vb}$ for any quantum pseudometric $\Vb$ on $\Mx$, with
equality if and only if each $\Vc_t$ is reflexive (Definition \ref{refdef}).
A measurable pseudometric $\rho$ is a measurable metric if and only if
$\Vb_\rho$ is a quantum metric.
\end{theo}

\begin{proof}
Let $\rho$ be a measurable pseudometric.
It follows from Definition \ref{measpm} (i) and (iii) that $pq \neq 0$
implies $\rho(p,q) = 0$ (since $p = p \vee pq$ and $q = q \vee pq$).
Thus $\rho(p,q) > t$ $\Rightarrow$ $M_pM_fM_q
= M_{pfq} = 0$ for all $f \in L^\infty(X,\mu)$,
which shows that $\Mx \subseteq \Vc_t^\rho$ for all $t$. Each $\Vc_t^\rho$ is
self-adjoint because $\rho$ is symmetric and is clearly weak operator,
and hence weak*, closed. So each $\Vc_t^\rho$ is a dual operator system
and $\Mx \subseteq \Vc_0^\rho$. Condition (ii) in the definition of a
W*-filtration (Definition \ref{filt}) is easy. For condition (i)
let $\Rc_t$ be the measurable relation $\Rc_t = \{(p,q): \rho(p,q) < t\}$
(\cite{W6}, Lemma 1.16) and observe that the corresponding
quantum relations $\Vc_{\Rc_t} = \{A \in \Bc(L^2(X,\mu)): \rho(p,q) \geq t$
$\Rightarrow$ $M_pAM_q = 0\}$ satisfy $\Vc_t^\rho = \bigcap_{s > t}
\Vc_{\Rc_s}$. Now let $\Vc_{s,t,\epsilon} =
\Vc_{\Rc_{s + \epsilon}}\Vc_{\Rc_{t + \epsilon}}$. We have
$\Rc_{\Vc_{s,t,\epsilon}} \subseteq \Rc_{s + t + 2\epsilon}$
by  Lemma 1.16 and Theorem 2.9 (d) of \cite{W6}, so
$$\Vc_s^\rho\Vc_t^\rho \subseteq
\Vc_{\Rc_{s + \epsilon}}\Vc_{\Rc_{t + \epsilon}} = \Vc_{s,t,\epsilon}
\subseteq \Vc_{\Rc_{\Vc_{s,t,\epsilon}}}
\subseteq \Vc_{\Rc_{s + t + 2\epsilon}};$$
intersecting over $\epsilon > 0$ yields $\Vc_s^\rho\Vc_t^\rho \subseteq
\Vc_{s + t}^\rho$. This completes the proof that $\Vb_\rho$ is a quantum
pseudometric on $\Mx$.

Next, let $\Vb$ be a quantum pseudometric on $\Mx$. Verification of
conditions (i), (ii), and (iii) of Definition \ref{measpm} for $\rho_\Vb$
is straightforward. For (iv), let $p$, $q$, and $r$ be nonzero projections
in $L^\infty(X,\mu)$ and let $\epsilon > 0$. We may assume
$\rho_\Vb(q,r) < \infty$. Find $A \in \Bc(L^2(X,\mu))$ such that
$D(A) \leq \rho_\Vb(q,r) + \epsilon$ and $M_qAM_r \neq 0$. Then
$$Q = \bigvee\{[M_{fq}AM_r]: f \in L^\infty(X,\mu)\}$$
is invariant for $\Mx$ and hence $Q = M_{q'}$ for some nonzero $q' \leq q$.
We may assume $\rho(p,q') < \infty$. Now find $B \in \Bc(L^2(X,\mu))$ such
that $D(B) \leq \rho_\Vb(p,q') + \epsilon$ and $M_pBM_{q'} \neq 0$, so that
$M_pBM_fAM_r \neq 0$ for some $f \in L^\infty(X,\mu)$. Then
$$\rho_\Vb(p,r) \leq D(BM_fA) \leq D(B) + D(A) \leq \rho_\Vb(p,q') +
\rho_\Vb(q',r) + 2\epsilon$$
since $\rho_\Vb(q,r) \leq \rho_\Vb(q',r)$. Taking the infimum over
$\epsilon$ yields
$$\rho_\Vb(p,r) \leq \sup_{q' \leq q}(\rho_\Vb(p,q') +
\rho_\Vb(q',r)).$$
This completes the proof that $\rho_\Vb$ is a measurable metric.

Now let $\rho$ be a measurable pseudometric, let $\Vb = \Vb_\rho$,
and let $\tilde{\rho} = \rho_\Vb$. Applying the formula $\Rc = \Rc_{\Vc_\Rc}$
(Theorem \ref{abelianrel}) to $\Rc_t = \{(p,q): \rho(p,q) < t\}$ yields
\begin{eqnarray*}
\Rc_t &=& \{(p,q): (\exists A \in \Bc(L^2(X,\mu)))\cr
&&\phantom{\liminf\liminf}(\rho(p',q') \geq t
\Rightarrow M_{p'}AM_{q'} = 0\quad{\rm and}\quad M_pAM_q \neq 0)\}.
\end{eqnarray*}
Letting
\begin{eqnarray*}
\tilde{\Rc}_t &=& \{(p,q): \tilde{\rho}(p,q) < t\}\cr
&=& \{(p,q): (\exists A \in \Bc(L^2(X,\mu)))(
D(A) < t\quad{\rm and}\quad M_pAM_q \neq 0)\},
\end{eqnarray*}
we then have
$$\tilde{\Rc_t} \subseteq \Rc_t \subseteq \tilde{\Rc}_{t + \epsilon}$$
for all $t$ and all $\epsilon > 0$. This shows that $\rho = \tilde{\rho}$.

Next let $\Vb$ be any quantum pseudometric on $\Mx$, let
$\rho = \rho_\Vb$, and let $\tilde{\Vb} = \Vb_\rho$. The inequality
$\Vb \leq \tilde{\Vb}$ is straightfoward. Conversely, let
$$\Wc_t = \{A \in \Bc(L^2(X,\mu)): M_p\Vc_tM_q = 0\quad\Rightarrow\quad
M_pAM_q = 0\},$$
so that $\Vc_t \subseteq \Wc_t$ and
$\Vc_t$ is reflexive if and only if $\Vc_t = \Wc_t$. We have
$$\tilde{\Vc}_t = \{A \in \Bc(L^2(X,\mu)): M_p\Vc_{t + \epsilon}M_q = 0
\hbox{ for some } \epsilon > 0\quad\Rightarrow\quad M_pAM_q =0\},$$
and
$$\Wc_t \subseteq \tilde{\Vc_t} \subseteq \Wc_{t + \epsilon}$$
for any $\epsilon > 0$. Thus if each $\Vc_t$ is reflexive then
$$\Vc_t = \Wc_t \subseteq \tilde{\Vc}_t \subseteq \bigcap_{s > t} \Wc_s
= \bigcap_{s > t} \Vc_s = \Vc_t$$
for all $t$,
so that $\Vc_t = \tilde{\Vc}_t$, and if some $\Vc_t$ is not reflexive then
$$\Vc_t \subsetneq \Wc_t \subseteq \tilde{\Vc}_t$$
for that $t$, so that $\Vc_t \neq \tilde{\Vc}_t$.
So $\Vb = \tilde{\Vb}$ if and only if each $\Vc_t$ is reflexive.

Finally, let $\rho$ be a measurable metric. If $A \in \Vc_0^\rho$
then $\rho(p,q) > 0$ implies $M_pAM_q = 0$, so the measurable metric
condition implies that $M_pAM_q = 0$ for any disjoint projections
$p$ and $q$ in $\Mx$. But this implies that $A$ belongs to $\Mx$, so we
have shown that if $\rho$ is a measurable metric then $\Vc_0^\rho =
\Mx$, i.e., $\Vb_\rho$ is a quantum metric. For the converse,
let the {\it closure} of $q$ be $\bar{q} = X -
\bigvee\{p: \rho(p,q) > 0\}$. If $\rho$ is not a measurable metric
then the closed projections in $L^\infty(X,\mu)$ do not generate
$L^\infty(X,\mu)$ as a von Neumann algebra (see Section 1.5 of
\cite{W6}). There must therefore exist
an operator $A \in \Bc(L^2(X,\mu))$ that commutes with $M_q$ for
every closed projection $q$ in $L^\infty(X,\mu)$ but does not belong
to $\Mx$. Now if $A \not\in \Vc_0^\rho$ then there
exist projections $p,q \in L^\infty(X,\mu)$ with $\rho(p,q) > 0$
and $M_pAM_q \neq 0$, but then $A$ cannot commute with $M_{\bar{q}}$,
a contradiction. So we conclude
that $A \in \Vc_0^\rho$, and this shows that $\Vc_0^\rho \neq \Mx$.
So if $\rho$ is not a measurable metric then $\Vb_\rho$ is not a
quantum metric.
\end{proof}


\subsection{Reflexivity and stabilization}\label{refstab}
We have seen the value of working with projections in $\Mx \overline{\otimes}
\Bc(l^2)$, and we will give simple examples in Section \ref{opsys} showing
that projections in $\Mx$ generally do not suffice in the basic results
of the theory. However, by
inflating $\Mx$ to $\tilde{\Mx} = \Mx \overline{\otimes} \Bc(l^2)$
and $\{\Vc_t\}$ to $\{\Vc_t \otimes I\}$ we can ensure that projections
in $\tilde{\Mx}$ do suffice for the basic theory. This is a consequence
of the general principle that projections in $\Mx$ suffice if $\Vb$ is
reflexive.

\begin{defi}\label{rfst}
Let $\Vb = \{\Vc_t\}$ be a quantum pseudometric on a von Neumann algebra $\Mx
\subseteq \Bc(H)$.

\noindent (a) $\Vb$ is {\it reflexive} if each $\Vc_t$ is reflexive
(Definition \ref{refdef}).

\noindent (b) The {\it stabilization} of $\Vb$ is the quantum
pseudometric $\Vb\otimes I = \{\Vc_t \otimes I\}$ on the von Neumann algebra
$\Mx \overline{\otimes} \Bc(l^2)$.
\end{defi}

We just give one illustration of the sufficiency of projections in $\Mx$
when $\Vb$ is reflexive; cf.\ Proposition \ref{recover} and Example
\ref{counterdiam}. The reader will not have any difficulty in supplying
analogous versions of, e.g., Propositions \ref{diameter} and \ref{gaps}.
See also Proposition \ref{refcolip} below.

\begin{prop}
Let $\Vb$ be a reflexive quantum pseudometric on a von Neumann algebra $\Mx
\subseteq \Bc(H)$. Then
$$\Vc_t = \{A \in \Bc(H): \rho(P,Q) > t
\quad\Rightarrow\quad PAQ = 0\}$$
for all $t \in [0,\infty)$, with $P$ and $Q$ ranging over projections in $\Mx$.
\end{prop}

\begin{proof}
Fix $t$ and let $\tilde{\Vc}_t = \{A \in \Bc(H): \rho(P,Q) > t
\Rightarrow PAQ = 0\}$. Then $\Vc_t \subseteq \tilde{\Vc}_t$ follows
immediately from the definition of $\rho$ (Definition \ref{projdist}).
Conversely, let $A \in \tilde{\Vc}_t$. For any $s > t$ we have
$$P\Vc_sQ = 0\quad\Rightarrow\quad\rho(P,Q) \geq s
\quad\Rightarrow\quad PAQ = 0$$
for any projections $P,Q \in \Mx$.
By reflexivity we conclude that $A$ belongs to $\Vc_s$ for all $s > t$,
and hence that $A$ belongs to $\Vc_t$. Thus $\Vc_t = \tilde{\Vc}_t$.
\end{proof}

Next we observe that reflexivity can always be achieved by stabilization.

\begin{prop}
Let $\Vb$ be a quantum pseudometric on a von Neumann algebra $\Mx
\subseteq \Bc(H)$. Then $\Vb \otimes I$ is a reflexive quantum pseudometric
on $\Mx \overline{\otimes} \Bc(l^2)$ and
$\Vb$ is a quantum metric if and only if $\Vb\otimes I$ is a quantum
metric.
\end{prop}

The first assertion follows from (\cite{W6}, Proposition 2.20) and the
second is easy.

We can transfer results and constructions from $\Mx$ to
$\Mx \overline{\otimes} \Bc(l^2)$.
For example, given a quantum pseudometric $\Vb$ on a von Neumann algebra
$\Mx \subseteq \Bc(H)$ and $\epsilon > 0$ we define the open
$\epsilon$-neighborhood of a projection $P \in \Mx \overline{\otimes}
\Bc(l^2)$ to be the projection onto
$$\overline{(\Vc_{<\epsilon}\otimes I)({\rm ran}(P))} =
\bigvee_{t < \epsilon} \overline{(\Vc_t\otimes I)({\rm ran}(P))},$$
which is just its open $\epsilon$-neighborhood in
$\Mx \overline{\otimes} \Bc(l^2)$ relative to the quantum
pseudometric $\Vb\otimes I$. We similarly define the closure of $P$
to be the projection onto
$$\bigcap_{t > 0} \overline{(\Vc_t \otimes I)({\rm ran}(P))},$$
which again is just its closure in $\Mx \overline{\otimes} \Bc(l^2)$
relative to the quantum pseudometric $\Vb\otimes I$. We still say that
$P$ is closed if it equals its closure. Note that Propositions
\ref{nbhds} and \ref{closures} hold for projections in
$\Mx \overline{\otimes} \Bc(l^2)$. Thus the above concepts all have
manifestly representation independent reformulations in terms of
projection distances in $\Mx \overline{\otimes} \Bc(l^2)$.

This notion of closure in $\Mx \overline{\otimes} \Bc(l^2)$ can be
used to give an intrinsic characterization of
the metric/pseudometric distinction.

\begin{prop}
Let $\Vb$ be a quantum pseudometric on a von Neumann algebra
$\Mx \subseteq \Bc(H)$. Then $\Vb$ is a quantum metric if and only if
the closed projections in $\Mx \overline{\otimes} \Bc(l^2)$
generate $\Mx \overline{\otimes} \Bc(l^2)$ as a von Neumann algebra.
\end{prop}

\begin{proof}
Let $\Nc \subseteq \Mx \overline{\otimes} \Bc(l^2)$ be the von Neumann
algebra generated by the closed projections. We will show that
$\Nc' = \Vc_0 \otimes I$; thus $\Nc = \Vc_0' \overline{\otimes} \Bc(l^2)$,
so that $\Nc = \Mx \overline{\otimes} \Bc(l^2)$ if and only if $\Vc_0 =
\Mx'$, i.e., $\Vb$ is a quantum metric.

Observe first that every projection in $I \otimes \Bc(l^2)$ is closed. Thus
$\Nc' \subseteq (I \otimes \Bc(l^2))' = \Bc(H) \otimes I$. Now if
$A \in \Vc_0$ then the range of any closed projection is clearly
invariant for $A \otimes I$. Since $A^*$ also belongs to $\Vc_0$
it follows that $A \otimes I$ commutes with every closed projection,
and therefore $\Vc_0 \otimes I \subseteq \Nc'$. Conversely, let
$A \in \Bc(H) - \Vc_0$. Then there exists $t > 0$ such that $A \not\in
\Vc_t$, and by Lemma \ref{separation} we can then find projections
$P,Q \in \Mx \overline{\otimes} \Bc(l^2)$ such that $P(B \otimes I)Q =0$
for all $B \in \Vc_t$ but $P(A \otimes I)Q \neq 0$. It follows that
$P$ is orthogonal to $\overline{Q}$, and thus the range of
$\overline{Q}$ cannot be invariant for $A \otimes I$. So $A\otimes I$
does not commute with $\overline{Q}$, and hence $A \otimes I
\not\in \Nc'$. This completes the proof that $\Nc' = \Vc_0 \otimes I$.
\end{proof}

\subsection{Constructions with quantum metrics}\label{constqm}
In this section we describe some simple constructions that can be
performed on quantum metrics. We start by identifying the appropriate
morphisms in the category.

\begin{defi}\label{lipmaps}
Let $\Vb$ and $\Wb$ be quantum pseudometrics on von Neumann algebras
$\Mx$ and $\Nc$. A {\it co-Lipschitz morphism} from $\Mx$ to $\Nc$ is
a unital weak* continuous $*$-homomorphism $\phi: \Mx \to \Nc$ for
which there exists a number $C \geq 0$ such that
$$\rho(P,Q) \leq
C\cdot\rho((\phi \otimes {\rm id})(P), (\phi\otimes{\rm id})(Q))$$
for all projections $P,Q \in \Mx\overline{\otimes}\Bc(l^2)$.
The minimum value of $C$ is the {\it co-Lipschitz number} of $\phi$,
denoted $L(\phi)$, and $\phi$ is a {\it co-contraction morphism} if
$L(\phi) \leq 1$. It is a {\it co-isometric morphism} if it is surjective and
$$\rho(\tilde{P},\tilde{Q}) = \sup\{\rho(P,Q): (\phi \otimes {\rm id})(P)
= \tilde{P}, (\phi\otimes{\rm id})(Q) = \tilde{Q}\}$$
for all projections $\tilde{P},\tilde{Q} \in \Nc\overline{\otimes} \Bc(l^2)$.
We set $L(\phi) = \infty$ if $\phi$ is not co-Lipschitz. Thus
$$L(\phi) = \sup_{P,Q}\frac{\rho(P,Q)}{\rho((\phi \otimes {\rm id})(P),
(\phi\otimes{\rm id})(Q))}$$
with $P$ and $Q$ ranging over projections in $\Mx \overline{\otimes} \Bc(l^2)$
and using the convention $\frac{0}{0} = \frac{\infty}{\infty} = 0$.
\end{defi}

We immediately record the most important property of co-Lipschitz
morphisms, which follows directly from their definition:

\begin{prop}\label{colipcomp}
Let $\Vb_1$, $\Vb_2$, and $\Vb_3$ be quantum pseudometrics on von Neumann
algebras $\Mx_1$, $\Mx_2$, and $\Mx_3$ and let $\phi: \Mx_1 \to \Mx_2$
and $\psi: \Mx_2 \to \Mx_3$ be co-Lipschitz morphisms. Then
$\psi\circ\phi: \Mx_1 \to \Mx_3$ is a co-Lipschitz morphism and
$L(\psi\circ\phi) \leq L(\psi)L(\phi)$.
\end{prop}

Definition \ref{lipmaps} is motivated by the atomic abelian case, where
the unital weak* continuous $*$-homomorphisms from $l^\infty(X)$
to $l^\infty(Y)$ are precisely the maps given by composition
with functions from $Y$ to $X$. If $X$ and $Y$ are pseudometric
spaces, let $L(f)$ denote the Lipschitz number of $f: Y \to X$,
$$L(f) = \sup_{x',y' \in Y}\frac{d_X(f(x'), f(y'))}{d_Y(x',y')}$$
(with the convention $\frac{0}{0}= 0$).

\begin{prop}\label{aacolip}
Let $X$ and $Y$ be pseudometric spaces and equip
$l^\infty(X)$ and $l^\infty(Y)$ with the corresponding quantum
pseudometrics (Proposition \ref{aa}).
If $f: Y \to X$ is a Lipschitz function then
$\phi: g \mapsto g\circ f$ is a co-Lipschitz morphism from $l^\infty(X)$
to $l^\infty(Y)$, and $L(\phi) = L(f)$. Every co-Lipschitz morphism from
$l^\infty(X)$ to $l^\infty(Y)$ is of this form.
\end{prop}

\begin{proof}
Let $f$ be any function from $Y$ to $X$ and let $\phi: l^\infty(X) \to
l^\infty(Y)$ be composition with $f$.
The projections in $l^\infty(X) \overline{\otimes} \Bc(l^2)$ can be
identified with projection-valued functions from $X$ into $\Bc(l^2)$,
and similarly for $Y$. Taking $P = e_x\cdot I$ and $Q = e_y \cdot I$
for $x,y \in X$, we have $(\phi\otimes{\rm id})(P) = \chi_{f^{-1}(x)}\cdot I$
and $(\phi\otimes {\rm id})(Q) = \chi_{f^{-1}(y)} \cdot I$. So
$$\frac{d(x,y)}{d(f^{-1}(x),f^{-1}(y))}
=\frac{\rho(P,Q)}{\rho((\phi\otimes{\rm id})(P),(\phi\otimes{\rm id})(Q))}
\leq L(\phi)$$
for all $x,y \in X$, and hence
$$\frac{d(f(x'),f(y'))}{d(x',y')} \leq
\frac{d(f(x'),f(y'))}{d(f^{-1}(f(x')), f^{-1}(f(y')))} \leq L(\phi)$$
for all $x',y' \in Y$; taking the supremum over $x'$ and $y'$ then
yields $L(f) \leq L(\phi)$. Conversely, let $P$ and $Q$ be any
projection-valued functions from $X$ into $\Bc(l^2)$. Then
$$\rho(P,Q) = \inf\{d(x,y): P(x)Q(y) \neq 0\}.$$
We may assume that $\rho((\phi\otimes {\rm id})(P),
(\phi\otimes {\rm id})(Q)) < \infty$.
Given $\epsilon > 0$, find $x',y' \in Y$ such that
$$(\phi\otimes {\rm id})(P)(x') = P(f(x'))$$
and
$$(\phi\otimes {\rm id})(Q)(y') = Q(f(y'))$$
have nonzero product and $d(x',y') \leq \rho((\phi\otimes {\rm id})(P),
(\phi\otimes {\rm id})(Q)) + \epsilon$. If $d(x',y') \geq
\rho((\phi\otimes {\rm id})(P), (\phi\otimes {\rm id})(Q)) >
\epsilon > 0$ then
$$\frac{\rho(P,Q)}{\rho((\phi\otimes {\rm id})(P),
(\phi\otimes {\rm id})(Q))} \leq
\frac{d(f(x'),f(y'))}{d(x',y') - \epsilon},$$
and taking $\epsilon \to 0$ and the supremum over $P$ and $Q$
yields $L(\phi) \leq L(f)$.
If $\rho((\phi\otimes {\rm id})(P), (\phi\otimes {\rm id})(Q)) = 0$
then either $\rho(P,Q) = 0$ and the pair $P,Q$ does not contribute to
$L(\phi)$, or else $\rho(P,Q) > 0$ and the above implies
$$\frac{d(f(x'),f(y'))}{d(x',y')} \geq \frac{\rho(P,Q)}{\epsilon}.$$
In that case taking $\epsilon \to 0$ yields $L(f) = \infty$, which again
implies $L(\phi) \leq L(f)$. So we have
shown that $L(\phi) = L(f)$, and hence $\phi$ is co-Lipschitz if and only
if $f$ is Lipschitz. Since every unital weak* continuous $*$-homomorphism
from $l^\infty(X)$ to $l^\infty(Y)$ is given by composition with
a function $f: Y \to X$, every co-Lipschitz morphism must arise in the
above manner.
\end{proof}

Under the assumption of reflexivity the co-Lipschitz number can be
computed using only projections in $\Mx$. The proof of this result is
notable for its use of the hard direction of Theorem \ref{abstractchar}.

\begin{prop}\label{refcolip}
Let $\Vb$ and $\Wb$ be quantum pseudometrics on von Neumann algebras
$\Mx \subseteq \Bc(H)$ and $\Nc \subseteq \Bc(K)$ and let
$\phi: \Mx \to \Nc$ be a unital weak* continuous $*$-homomorphism.
Suppose $\Vb$ is reflexive (Definition \ref{rfst} (a)). Then
$$L(\phi) = \sup_{P,Q} \frac{\rho(P,Q)}{\rho(\phi(P),\phi(Q))},$$
with $P$ and $Q$ ranging over projections in $\Mx$.
\end{prop}

\begin{proof}
Let $\tilde{L}(\phi) = \sup \rho(P,Q)/\rho(\phi(P),\phi(Q))$ be the
supremum with $P$ and $Q$ ranging over projections in $\Mx$. Then it
is immediate that $\tilde{L}(\phi) \leq L(\phi)$ since the supremum
defining the former is effectively contained in the supremum defining
the latter. For the reverse inequality, fix projections $P$ and $Q$
in $\Mx \overline{\otimes} \Bc(l^2)$. We may assume that
$\rho(P,Q) > 0$ and $\rho((\phi\otimes {\rm id})(P),(\phi \otimes {\rm id})(Q))
< \infty$, so let $0 < s < \rho(P,Q)$ and let
$t > \rho((\phi\otimes {\rm id})(P),(\phi \otimes {\rm id})(Q))$.
Then the pair $((\phi\otimes {\rm id})(P),(\phi \otimes {\rm id})(Q))$
belongs to the intrinsic quantum relation $\Rc_{\Wc_t}$ (Definition
\ref{abstractqrel}) associated to
the quantum relation $\Wc_t$ (Theorem \ref{abstractchar}) and hence
the pair $(P,Q)$ belongs to its pullback $\Rc = \phi^*(\Rc_{\Wc_t})$
(\cite{W6}, Proposition 2.25 (b)). Let $\Vc = \Vc_\Rc$ be the quantum relation
associated to $\Rc$; then since $(P,Q) \in \Rc = \Rc_{\Vc}$ (Theorem
\ref{abstractchar}) there exists $A \in \Vc$ such that $P(A \otimes I)Q
\neq 0$. But $\rho(P,Q) > s$, so $A \not\in \Vc_s$, and this shows that
$\Vc \not\subseteq \Vc_s$. Since $\Vc_s$ is reflexive we can then find
projections $\tilde{P}, \tilde{Q} \in \Mx$ such that
$\tilde{P}\Vc_s\tilde{Q} = 0$ but $\tilde{P}\Vc\tilde{Q}
\neq 0$. Then $\rho(\tilde{P},\tilde{Q}) \geq s$ but $(\tilde{P},\tilde{Q})
\in \Rc$ so $\rho(\phi(\tilde{P}),\phi(\tilde{Q})) \leq t$, so we have
$$\frac{\rho(\tilde{P},\tilde{Q})}{\rho(\phi(\tilde{P}),\phi(\tilde{Q}))}
\geq \frac{s}{t}.$$
Taking $s \to \rho(P,Q)$ and $t \to
\rho((\phi\otimes {\rm id})(P),(\phi \otimes {\rm id})(Q))$
shows that
$$\frac{\rho(P,Q)}{\rho((\phi\otimes {\rm id})(P),(\phi \otimes {\rm id})(Q))}
\leq \tilde{L}(\phi)$$
and taking the supremum over $P$ and $Q$ finally yields
$L(\phi) \leq \tilde{L}(\phi)$.
\end{proof}

The co-Lipschitz number is formulated in terms of the projection
distances introduced in Definition \ref{projdist} but
it has an equivalent version in terms of W*-filtrations.
We use the general form of a unital weak* continuous
$*$-homomorphism $\phi: \Mx \to \Nc$ (\cite{Tak}, Theorem IV.5.5) which
states that every such map can be expressed as an inflation followed
by a restriction followed by an isomorphism. Since this expression is
not unique, if we defined $L(\phi)$ in the concrete way indicated
below then the definition would appear to be ambiguous. But the fact that
this definition is equivalent to the intrinsic one given above means
that there is actually no real ambiguity.

\begin{prop}\label{concretelip}
Let $\Vb$ and $\Wb$ be quantum pseudometrics on von Neumann algebras
$\Mx \subseteq \Bc(H)$ and $\Nc \subseteq \Bc(K)$ and let
$\phi: \Mx \to \Nc$ be a unital weak* continuous $*$-homomorphism.
Let $\tilde{K}$ be a Hilbert space, $R$ a projection in
$\Bc(\tilde{K})\overline{\otimes}\Mx'$, and $U$ an isometry
from $K$ to ${\rm ran}(R)$ such that
$\phi(A) = U^*(I_{\tilde{K}}\otimes A)U$
for all $A \in \Mx$. Then
$$L(\phi) = \inf\{C \geq 0: \Wc_t \subseteq
U^*(\Bc(\tilde{K})\overline{\otimes}\Vc_{Ct})U\hbox{ for all }t\geq 0\}$$
(with $\inf \emptyset = \infty$).
\end{prop}

\begin{proof}
Let $\tilde{L}(\phi) = \inf\{C \geq 0: \Wc_t \subseteq
U^*(\Bc(\tilde{K})\overline{\otimes}\Vc_{Ct})U$ for all $t\geq 0\}$.
Let $P$ and $Q$ be projections in
$\Mx \overline{\otimes} \Bc(l^2)$ and set
$$\tilde{P} = (\phi\otimes {\rm id})(P) =
(U^*\otimes I_{l^2})(I_{\tilde{K}}\otimes P)(U\otimes I_{l^2})$$
and
$$\tilde{Q} = (\phi\otimes {\rm id})(Q) =
(U^*\otimes I_{l^2})(I_{\tilde{K}}\otimes Q)(U\otimes I_{l^2}),$$
both in $\Nc\overline{\otimes} \Bc(l^2)$.

Assuming $\rho(\tilde{P}, \tilde{Q}) < \infty$,
let $t > \rho(\tilde{P},\tilde{Q})$ and find $A \in \Wc_t$ such that
$\tilde{P}(A \otimes I_{l^2})\tilde{Q} \neq 0$. Then $A \in
U^*(\Bc(\tilde{K})\overline{\otimes} \Vc_{\tilde{L}(\phi)t})U$, so
$$UAU^* \in R(\Bc(\tilde{K})\overline{\otimes}\Vc_{\tilde{L}(\phi)t})R
\subseteq \Bc(\tilde{K})\overline{\otimes} \Vc_{\tilde{L}(\phi)t}$$
(since $\Mx'\Vc_{\tilde{L}(\phi)t}\Mx' \subseteq \Vc_{\tilde{L}(\phi)t}$), and
$$(R\otimes I_{l^2})(I_{\tilde{K}}\otimes P)(UAU^* \otimes I_{l^2})
(I_{\tilde{K}}\otimes Q)(R \otimes I_{l^2})
=(U\otimes I_{l^2})\tilde{P}(A\otimes I_{l^2})\tilde{Q}(U^*\otimes I_{l^2})
\neq 0.$$
Thus $\rho(P,Q) \leq \tilde{L}(\phi)t$, and taking $t \to
\rho(\tilde{P}, \tilde{Q})$ and then taking the supremum over $P$ and $Q$
shows that $L(\phi) \leq \tilde{L}(\phi)$.

Conversely, assume $\tilde{L}(\phi) > 0$, let
$C < \tilde{L}(\phi)$, and find $t$ such that
$\Wc_t \not\subseteq U^*(\Bc(\tilde{K})\overline{\otimes}\Vc_{Ct})U$,
or equivalently, such that $U\Wc_tU^* \not\subseteq
\Bc(\tilde{K})\overline{\otimes} \Vc_{Ct}$. By Lemma \ref{separation}
with $I_{\tilde{K}}\otimes \Mx$ in place of $\Mx$ we can find
projections $P,Q \in \Mx\overline{\otimes}\Bc(l^2)$ such that
$$(I_{\tilde{K}}\otimes P)(UAU^* \otimes I_{l^2})
(I_{\tilde{K}}\otimes Q) \neq 0$$
for some $A \in \Wc_t$, and hence $\tilde{P}(A\otimes I_{l^2})\tilde{Q}
\neq 0$ with $\tilde{P} = (\phi \otimes {\rm id})(P)$ and $\tilde{Q} =
(\phi \otimes {\rm id})(Q)$ as in the first part of the proof, but
$$(I_{\tilde{K}}\otimes P)(B \otimes I_{l^2})(I_{\tilde{K}}\otimes Q) = 0$$
for all $B \in \Bc(\tilde{K})\overline{\otimes} \Vc_{Ct}$. It follows
that $\rho(P,Q) \geq Ct$ but $\rho(\tilde{P},\tilde{Q}) \leq t$. So
$L(\phi) \geq C$, and we conclude that $\tilde{L}(\phi) \leq L(\phi)$.
Thus we have shown that $L(\phi) = \tilde{L}(\phi)$.
\end{proof}

The formula for $L(\phi)$ in Proposition \ref{concretelip} may be
more transparent if the map $\phi$ is explicitly decomposed as follows.
Let $\Mx_1 = I_{\tilde{K}}\otimes \Mx \subseteq
\Bc(\tilde{K}\otimes H)$, $\Mx_2 = \Mx_1R$, and $\Nc_1 = U\Nc U^* \subseteq
\Bc({\rm ran}(R))$, equipped with quantum pseudometrics $\Vb_1 =
\{\Vc^1_t\}$, $\Vb_2 = \{\Vc^2_t\}$, and $\Wb_1 = \{\Wc^1_t\}$ where
$$\Vc^1_t = \Bc(\tilde{K})\overline{\otimes} \Vc_t
\qquad \Vc^2_t = R\Vc^1_tR\qquad\Wc^1_t = U\Wc_tU^*.$$
Then $\phi = \phi_r\circ\phi_3\circ\phi_2\circ\phi_1$ where
$\phi_1: \Mx \to \Mx_1$, $\phi_2: \Mx_1 \to \Mx_2$, $\phi_3:
\Mx_2 \to \Nc_1$, and $\phi_4: \Nc_1 \to \Nc$ are defined by
$\phi_1:A \mapsto I_{\tilde{K}}\otimes A$, $\phi_2: A \mapsto AR$,
$\phi_3 = {\rm id}$, and $\phi_4: A \mapsto U^*AU$.
Unless some degeneracy occurs such that $L(\phi) = L(\phi_i) = 0$
for some $i$, we have
\begin{eqnarray*}
L(\phi_1) &=& \inf\{C \geq 0: \Vc^1_t \subseteq
\Bc(\tilde{K})\overline{\otimes}\Vc_{Ct}\hbox{ for all }t\geq 0\} = 1\cr
L(\phi_2) &=& \inf\{C \geq 0: \Vc^2_t \subseteq R\Vc^1_tR
\hbox{ for all }t \geq 0\} = 1\cr
L(\phi_3) &=& \inf\{C \geq 0: \Wc^1_t \subseteq
\Vc^2_{Ct}R\hbox{ for all }t\geq 0\}\cr
L(\phi_4) &=& \inf\{C \geq 0: \Wc_t \subseteq U^*\Wc^1_{Ct}U
\hbox{ for all }t\geq 0\} = 1
\end{eqnarray*}
and $L(\phi) = L(\phi_3)$.

Next we present three easy constructions.

\begin{defi}\label{easycon}
(a) Let $\Vb$ be a quantum pseudometric on a von Neumann algebra
$\Mx \subseteq \Bc(H)$ and let $C \geq 0$. Then the {\it truncation}
of $\Vb$ to $C$ is the quantum pseudometric $\tilde{\Vb} =(\tilde{\Vc}_t)$
defined by
$$\tilde{\Vc}_t =
\begin{cases}
\Vc_t& \hbox{if }t < C\cr
\Bc(H)& \hbox{if }t \geq C.
\end{cases}$$

\noindent (b) Let $\Vb$ and $\Wb$ be quantum pseudometrics on von Neumann algebras
$\Mx \subseteq \Bc(H)$ and $\Nc \subseteq \Bc(K)$. Their {\it direct
sum} is the von Neumann algebra $\Mx \oplus \Nc \subseteq \Bc(H \oplus K)$
equipped with the quantum pseudometric $\Vb \oplus \Wb =
\{\Vc_t \oplus \Wc_t\}$.

\noindent (c) Let $\{\Vb_\lambda\}$ with $\Vb_\lambda = \{\Vc^\lambda_t\}$
be a family of quantum pseudometrics on a von Neumann algebra
$\Mx \subseteq \Bc(H)$. Their {\it meet} is the quantum
pseudometric $\bigwedge \Vb_\lambda = \{\bigcap_\lambda \Vc^\lambda_t\}$.
\end{defi}

In the atomic abelian case truncations reduce to the classical
construction
$$\tilde{d}(x,y) = \min\{d(x,y),C\},$$
direct sums reduce to the disjoint union construction $X \coprod Y$
with $d(x,y) = \infty$ for all $x \in X$ and $y \in Y$, and meets
reduce to the supremum of a family of pseudometrics. In the case of
direct sums note that if ${\rm diam}(\Vb),{\rm diam}(\Wb) \leq C$ then
we can truncate their disjoint union to $C$ without affecting the embedded
copies of $\Vb$ and $\Wb$. This corresponds to setting $d(x,y) = C$
for all $x \in X$ and all $y \in Y$ in the classical case. More generally,
for any $r \geq \max\{{\rm diam}(\Vb), {\rm diam}(\Wb)\}/2$ we can replace
$\Vc_t \oplus \Wc_t$ with
$$\left\{\left[
\begin{matrix}
A&B\cr
C&D
\end{matrix}\right]: A \in \Vc_t, B \in \Bc(K,H), C \in \Bc(H,K),
D \in \Wc_t\right\}$$
for all $t \geq r$. This corresponds to setting $d(x,y) = r$ for
all $x \in X$ and all $y \in Y$ in the classical case.

The meet construction in general is not obtained at the level
of projections by setting $\rho(P,Q) = \sup \rho_\lambda(P,Q)$;
see Example \ref{countersup}. However, truncations and direct
sums do satisfy the obvious formulas at the projection level.

\begin{prop}\label{trunc}
Let $\Vb$ be a quantum pseudometric on a von Neumann algebra
$\Mx \subseteq \Bc(H)$ and let $\tilde{\Vb}$ be its truncation to $C \geq 0$.
Then
$$\rho_{\tilde{\Vb}}(P,Q) = \min\{\rho_{\Vb}(P,Q),C\}$$
for all linkable projections $P,Q \in \Mx \overline{\otimes} \Bc(l^2)$.
\end{prop}

(Recall from Proposition \ref{findist} that the distance between
unlinkable projections is always $\infty$.)

\begin{prop}\label{directsum}
Let $\Vb$ and $\Wb$ be quantum pseudometrics on von Neumann algebras
$\Mx \subseteq \Bc(H)$ and $\Nc \subseteq \Bc(K)$ and let $P = P_1 \oplus P_2$
and $Q = Q_1 \oplus Q_2$
be projections in $(\Mx \oplus \Nc)\overline{\otimes}\Bc(l^2) \cong
(\Mx \overline{\otimes}\Bc(l^2)) \oplus (\Nc \overline{\otimes}\Bc(l^2))$.
Then
$$\rho_{\Vb \oplus \Wb}(P,Q) = \min\{\rho_\Vb(P_1,Q_1), \rho_\Wb(P_2,Q_2)\}.$$
\end{prop}

The proofs are straightforward. In the proof of Proposition \ref{trunc}
we use the fact that if $P$ and $Q$ are linkable then there exists
$A \in \Bc(H)$ such that $P(A \otimes I)Q \neq 0$ (Proposition
\ref{findist}).

We now turn to quotients, subobjects, and products. Quotients
are simplest. If $\phi: \Mx \to \Nc$ is a
surjective unital weak* continuous $*$-homomorphism then ${\rm ker}(\phi)$
is a weak* closed ideal of $\Mx$, and hence ${\rm ker}(\phi) = R\Mx$ for
some central projection $R \in \Mx$. Thus $\Mx = R\Mx \oplus (I-R)\Mx$ with
$(I - R)\Mx \cong \Nc$. So metric quotients are modelled by von Neumann
algebra direct summands.

\begin{defi}\label{quot}
Let $\Vb$ be a quantum pseudometric on a von Neumann algebra $\Mx
\subseteq \Bc(H)$. A {\it metric quotient} of $\Mx$ is a direct
summand $\Nc = R\Mx \subseteq \Bc(K)$ of $\Mx$, where $R$ is a
central projection in $\Mx$ and $K = {\rm ran}(R)$, together with
the quantum pseudometric $\Wb = \{\Wc_t\}$ on $\Nc$ defined by
$$\Wc_t = R\Vc_tR \subseteq \Bc(K).$$
\end{defi}

In this definition note that $\Wc_t \subseteq \Vc_t$ because $R \in \Mx'$
and $\Vc_t$ is a bimodule over $\Mx' \subseteq \Vc_0$.

For example, in Definition \ref{easycon} (b)
$\Mx$ and $\Nc$ are metric quotients of $\Mx \oplus \Nc$,
and this remains true after truncation to $\max\{{\rm diam}(\Vb),
{\rm diam}(\Wb)\}$.

\begin{prop}\label{quotdist}
Let $\Vb$ be a quantum pseudometric on a von Neumann algebra $\Mx
\subseteq \Bc(H)$ and let $\Nc = R\Mx$ be a metric quotient of $\Mx$
with quantum pseudometric $\Wb$.
Then the quantum distance function $\rho_\Wb$ on $\Nc$ is the restriction
of the quantum distance function $\rho_\Vb$ on $\Mx$ to $\Nc \overline{\otimes}
\Bc(l^2) \subseteq \Mx \overline{\otimes} \Bc(l^2)$.
\end{prop}

This proposition follows from the observation
that if $P$ and $Q$ are projections in $\Nc \overline{\otimes}
\Bc(l^2) \subseteq \Mx \overline{\otimes} \Bc(l^2)$ and $A \in
\Vc_t$ then $RAR \in \Wc_t \subseteq \Vc_t$ and
$$P(A \otimes I)Q \neq 0\quad\Leftrightarrow\quad
P(RAR\otimes I)Q \neq 0.$$

\begin{coro}
Let $\Vb$ be a quantum pseudometric on a von Neumann algebra $\Mx
\subseteq \Bc(H)$ and let $\Nc = R\Mx$ be a metric quotient of $\Mx$
with quantum pseudometric $\Wb$.
Then the map $\phi: A \mapsto AR$ is a co-isometric morphism from
$\Mx$ to $\Nc$. Up to isomorphism of the range every co-isometric
morphism has this form.
\end{coro}

\begin{proof}
Let $\tilde{P}$ and $\tilde{Q}$ be projections in $\Nc \overline{\otimes}
\Bc(l^2)$. The projections in $\Mx \overline{\otimes} \Bc(l^2)
\cong (R\Mx \overline{\otimes} \Bc(l^2)) \oplus ((I-R)\Mx
\overline{\otimes} \Bc(l^2))$ that map
to $\tilde{P}$ and $\tilde{Q}$ are just those of the form
$\tilde{P} \oplus P$ and $\tilde{Q} \oplus Q$ for arbitrary projections
$P$ and $Q$ in $(I-R)\Mx \overline{\otimes} \Bc(l^2)$. We have
$$\rho_\Vb(\tilde{P} \oplus P, \tilde{Q} \oplus Q) \leq
\rho_\Vb(\tilde{P} \oplus 0, \tilde{Q} \oplus 0)
= \rho_\Wb(\tilde{P},\tilde{Q}),$$
with equality if $P = Q = 0$. So $\phi$ is a co-isometric morphism.

Any co-isometric morphism is a surjective weak* continuous $*$-homomorphism
and hence up to isomorphism of the range is of the form $\phi: A \mapsto AR$
from $\Mx$ to $\Nc = R\Mx$ where $R$ is a central projection in $\Mx$.
The condition
$$\rho_\Wb(\tilde{P},\tilde{Q}) = \sup \rho_\Vb(\tilde{P}\oplus P,
\tilde{Q}\oplus Q) = \rho_\Vb(\tilde{P} \oplus 0, \tilde{Q}\oplus 0)$$
then implies that $\Nc$ is a metric quotient of $\Mx$ by Proposition
\ref{quotdist}.
\end{proof}

Next we consider subobjects. Even in the classical setting the dual
construction is slightly subtle; this is the metric quotient, discussed
in Section 1.4 of \cite{W4}.

\begin{defi}\label{subobj}
Let $\Vb$ be a quantum pseudometric on a von Neumann algebra $\Mx
\subseteq \Bc(H)$. A {\it metric subobject} of $\Mx$ is a unital von Neumann
subalgebra $\Nc \subseteq \Mx$ together with the quantum pseudometric
$$\Vb_\Nc = \bigwedge\{\Wb: \Vb \leq \Wb\hbox{ and }
\Nc' \subseteq \Wc_0\}$$
where $\Wb$ ranges over W*-filtrations of $\Bc(H)$.
In other words, $\Vb_\Nc$ is the meet of all quantum
pseudometrics on $\Nc$ that dominate $\Vb$.
\end{defi}

Metric subobjects have an obvious universal property:

\begin{prop}\label{univquot}
Let $\Vb$ be a quantum pseudometric on a von Neumann algebra $\Mx
\subseteq \Bc(H)$ and let $\Nc \subseteq \Mx$ be a metric subobject
of $\Mx$.

\noindent (a) The inclusion map $\iota: \Nc \to \Mx$ is a co-contraction
morphism (equipping $\Nc$ with the quantum pseudometric $\Vb_\Nc$).

\noindent (b) If $\Wb$ is any quantum pseudometric on $\Nc$ which makes
the inclusion map a co-contraction morphism then the
identity map from $\Nc$ to itself is a co-contraction morphism from the
pseudometric $\Wb$ to the pseudometric $\Vb_\Nc$.
\end{prop}

This holds because $L(\iota) \leq 1$ if and only if $\Vb \leq \Wb$,
by Proposition \ref{concretelip}. However, even basic questions
such as ``Under what conditions will a metric subobject of a quantum metric
be a quantum metric (not just a pseudometric)?''\ in general have no simple
answer. But this is already true in the classical case for the dual
question about quotients of metric spaces.

In order to define products of quantum metrics we need to be able to
take tensor products of dual operator systems. The most useful
generalization of the spatial tensor product of von Neumann algebras
to dual operator spaces is the {\it normal Fubini tensor product}
\cite{Ham}. If $\Vc \subseteq \Bc(H)$ and $\Wc \subseteq \Bc(K)$ are
dual operator spaces then their normal Fubini tensor product
can be defined concretely as
$$\Vc \overline{\otimes}_\Fc \Wc =
(\Vc \overline{\otimes} \Bc(K)) \cap (\Bc(H) \overline{\otimes} \Wc)$$
where $\overline{\otimes}$ is, as before, the normal spatial tensor
product (i.e., the weak* closure of the algebraic tensor product).
Abstractly, $\Vc \overline{\otimes}_\Fc \Wc$ is characterized as the
dual of the projective tensor product of the preduals of $\Vc$ and
$\Wc$ \cite{Ble, Rua}:
$$\Vc \overline{\otimes}_\Fc \Wc \cong (\Vc_* \hat{\otimes} \Wc_*)^*.$$
The normal spatial tensor product is always contained in the normal
Fubini tensor product but this inclusion may be strict. Thus, the equality
$\Vc \overline{\otimes}_\Fc \Wc =
(\Vc \overline{\otimes} \Bc(K)) \cap (\Bc(H) \overline{\otimes} \Wc)$,
which is crucial for the following proof, fails in general for the normal
spatial tensor product.

\begin{prop}
Let $\Vb = \{\Vc_t\}$ and $\Wb = \{\Wc_t\}$ be W*-filtrations of
$\Bc(H)$ and $\Bc(K)$, repectively. Then $\Vb \overline{\otimes}_\Fc \Wb =
\{\Vc_t \overline{\otimes}_\Fc \Wc_t\}$ is a W*-filtration of
$\Bc(H \otimes K)$.
\end{prop}

\begin{proof}
We have
\begin{eqnarray*}
(\Vc_s\overline{\otimes}_\Fc \Wc_s)(\Vc_t\overline{\otimes}_\Fc \Wc_t)
&=&((\Vc_s\overline{\otimes} \Bc(K))\cap(\Bc(H)\overline{\otimes} \Wc_s))
((\Vc_t\overline{\otimes} \Bc(K))\cap(\Bc(H)\overline{\otimes} \Wc_t))\cr
&\subseteq&
((\Vc_s\overline{\otimes} \Bc(K))(\Vc_t\overline{\otimes} \Bc(K))) \cap
((\Bc(H) \overline{\otimes} \Wc_s)(\Bc(H)\overline{\otimes} \Wc_t))\cr
&\subseteq& (\Vc_{s+t}\overline{\otimes} \Bc(K)) \cap
(\Bc(H)\overline{\otimes} \Wc_{s+t})\cr
&=&\Vc_{s+t}\overline{\otimes}_\Fc \Wc_{s+t}
\end{eqnarray*}
and
\begin{eqnarray*}
\bigcap_{s > t} \Vc_s\overline{\otimes}_\Fc \Wc_s
&=& \bigcap_{s > t} (\Vc_s \overline{\otimes} \Bc(K)) \cap
(\Bc(H) \overline{\otimes} \Wc_s)\cr
&=& \left(\bigcap_{s > t} \Vc_s \overline{\otimes} \Bc(K)\right)
\cap \left(\bigcap_{s > t} \Bc(H) \overline{\otimes} \Wc_s\right)\cr
&=& (\Vc_t \overline{\otimes} \Bc(K)) \cap (\Bc(H) \overline{\otimes} \Wc_t)\cr
&=& \Vc_t \overline{\otimes}_\Fc \Wc_t
\end{eqnarray*}
so both conditions of Definition \ref{filt} (a) are satisfied.
\end{proof}

\begin{defi}
Let $\Vb$ and $\Wb$ be quantum pseudometrics on von Neumann algebras
$\Mx \subseteq \Bc(H)$ and $\Nc \subseteq \Bc(K)$. Their {\it metric
product} is the von Neumann algebra $\Mx \overline{\otimes} \Nc
\subseteq \Bc(H \otimes K)$ equipped with the quantum pseudometric
$\Vb \overline{\otimes}_\Fc \Wb = \{\Vc_t \overline{\otimes}_\Fc \Wc_t\}$.
\end{defi}

\begin{prop}\label{productprop}
Let $\Vb$ and $\Wb$ be quantum pseudometrics on von Neumann algebras
$\Mx \subseteq \Bc(H)$ and $\Nc \subseteq \Bc(K)$.

\noindent (a) The metric product $\Vb \overline{\otimes}_\Fc \Wb$ is the
meet of the quantum pseudometrics $\Vb \overline{\otimes} \Wb_0
= \{\Vc_t \overline{\otimes} \Wc^0_t\}$
and $\Vb_0 \overline{\otimes} \Wb = \{\Vc^0_t \overline{\otimes} \Wc_t\}$
where $\Vb_0 = \{\Vc^0_t\}$
and $\Wb_0 = \{\Wc^0_t\}$ are the trivial quantum pseudometrics with
$\Vc^0_t = \Bc(H)$ and $\Wc^0_t = \Bc(K)$ for all $t$.

\noindent (b) The natural embeddings $\iota_1: A \mapsto A \otimes I_K$
and $\iota_2: B \mapsto I_H \otimes B$ of $\Mx$ and $\Nc$ into
$\Mx \overline{\otimes} \Nc$ realize $\Mx$ and $\Nc$ as metric
subobjects of the metric product.

\noindent (c) $\Vb \overline{\otimes}_\Fc \Wb$ is a quantum metric
if and only if both $\Vb$ and $\Wb$ are quantum metrics.
\end{prop}

\begin{proof}
(a) Trivial.

(b) By symmetry it is enough to consider the embedding of $\Mx$ into
$\Mx \overline{\otimes} \Nc$. The quantum pseudometric on $\Mx \otimes I
\subseteq \Bc(H\otimes K)$ corresponding to $\Vb$ is $\Vb \overline{\otimes}
\Wb_0$ where $\Wb_0$ is the trivial quantum pseudometric as in part (a)
(Theorem \ref{repindep}), so we have to show that
$$\Vb \overline{\otimes} \Wb_0 =
\bigwedge\{\tilde{\Wb}: \Vb \overline{\otimes}_\Fc \Wb
\leq \tilde{\Wb}\hbox{ and }\Mx' \overline{\otimes} \Bc(K) \subseteq
\tilde{\Wc}_0\}$$
where $\tilde{\Wb}$ ranges over W*-filtrations of $\Bc(H \otimes K)$.
It is easy to check that $\Vb \overline{\otimes} \Wb_0$ belongs to the
meet on the right, which verifies the inequality $\leq$.
For the reverse inequality, let $\tilde{\Wb}$ be any W*-filtration
satisfying $\Vb \overline{\otimes}_\Fc \Wb \leq \tilde{\Wb}$
and $\Mx' \overline{\otimes} \Bc(K) \subseteq \tilde{\Wc}_0$. Then
$\Vc_t \otimes I \subseteq \tilde{\Wc}_t$ for all $t$ and
$I \otimes \Bc(K) \subseteq \tilde{\Wc}_0$. Since
$\tilde{\Wc}_0\tilde{\Wc}_t \subseteq \tilde{\Wc}_t$,
it follows that $\Vc_t \overline{\otimes} \Bc(K) \subseteq \tilde{\Wc}_t$
for all $t$, i.e., that $\Vb \overline{\otimes} \Wb_0 \leq
\tilde{\Wb}$. This verifies the inequality $\geq$.

(c) If either $\Vb$ or $\Wb$ is not a quantum metric then either
$\Mx' \subsetneq \Vc_0$ or $\Nc' \subsetneq \Wc_0$,
and it follows that $\Mx' \overline{\otimes} \Nc'
\subsetneq \Vc_0 \overline{\otimes}_\Fc \Wc_0$, so that
$\Vb \overline{\otimes}_\Fc \Wb$ is not a quantum metric. The reverse
implication follows from the fact that
$$\Mx'\overline{\otimes}\Nc' = \Mx' \overline{\otimes}_\Fc \Nc'$$
since $\Mx'$ and $\Nc'$ are von Neumann algebras \cite{ER}.
\end{proof}

The definition of the metric product can be varied. For instance,
an $l^p$ product ($1 \leq p < \infty$) could be defined as the
smallest W*-filtration $\tilde{\Wb} = \{\tilde{\Wc}_t\}$ of
$\Bc(H \otimes K)$ satisfying
$$\Vc_s \overline{\otimes}_\Fc \Wc_t
\subseteq \tilde{\Wc}_{(s^p + t^p)^{1/p}}$$
for all $s,t \geq 0$. This product also has the properties proven
for the metric product in Proposition \ref{productprop} (b) and (c);
the first holds by essentially the same proof given for metric products,
and the second follows from the fact that the $l^p$ product W*-filtration
is contained in the metric product W*-filtration. However, it is not
clear that the $l^p$ product W*-filtration has any more explicit
description than the one just given.

Finally, we note that the quotient, subobject, and product constructions
discussed above reduce to the standard notions in the atomic abelian case.

\begin{prop}
Let $X$ and $Y$ be pseudometric spaces with pseudometrics $d$ and $d'$ and
let $\Mx \cong l^\infty(X)$ and $\Nc \cong l^\infty(Y)$ be the von Neumann
algebras of bounded multiplication operators on $l^2(X)$ and $l^2(Y)$,
equipped with the corresponding quantum pseudometrics (Proposition
\ref{aa}).

\noindent (a) For any subset $X_0$ of $X$ with inherited pseudometric $d_0$,
the von Neumann algebra
$$\Mx_0 = \{M_f: {\rm supp}(f) \subseteq X_0\} \subseteq \Mx$$
equipped with the quantum pseudometric $\Vb_{d_0}$ (Proposition \ref{aa})
is a metric quotient of $\Mx$. Every metric quotient of $\Mx$ is of this
form.

\noindent (b) For any equivalence relation $\sim$ on $X$ with quotient
pseudometric $\tilde{d}$ (\cite{W4}, Definition 1.4.2), the von Neumann
algebra
$$\tilde{\Mx} = \{M_f: f \in l^\infty(X),\quad
x\sim y\Rightarrow f(x) = f(y)\} \subseteq \Mx$$
equipped with  the quantum pseudometric $\Vb_{\tilde{d}}$ is a metric
subobject of $\Mx$. Every metric subobject of $\Mx$ is of this form.

\noindent (c) The metric product of $\Mx$ and $\Nc$ is the von Neumann algebra
$\Mx \overline{\otimes}\Nc \cong l^\infty(X\times Y)$ equipped with
the quantum pseudometric $\Vb_{d_{X \times Y}}$ associated to the
pseudometric
$$d_{X\times Y}((x_1,y_1), (x_2,y_2)) = \max\{d(x_1,x_2), d'(y_1,y_2)\}.$$
\end{prop}

The proof is straightforward. In part (b) we show that
$\Vb_{\tilde{d}}$ is the metric subobject quantum pseudometric by
observing that $\tilde{d}$ has the universal property stated in
Proposition \ref{univquot}; see (\cite{W4}, Proposition 1.4.3).

\subsection{Intrinsic characterization}\label{acr}
We show that quantum pseudometrics can be characterized intrinsically in
terms of quantum distance functions. But first we observe that in finite
dimensions quantum pseudometrics can be encoded as positive operators
in $\Mx \otimes \Mx^{op}$.

\begin{prop}
Let $\Mx \subseteq \Bc(H)$ be a finite dimensional von Neumann algebra
and let $\Vb$ be a quantum pseudometric on $\Mx$. Then there is a
positive operator $X$ in $\Mx \otimes \Mx^{op}$ such that
$$\Vc_t = \{B \in \Bc(H): \Phi_{P_{(t,\infty)}(X)}(B) = 0\}$$
for all $t \geq 0$, where $P_{(t,\infty)}(X)$ is the spectral projection
of $X$ and $\Phi$ is the action of $\Mx \otimes \Mx^{op}$ on
$\Bc(H)$ defined by
$$\Phi_{A \otimes C}(B) = ABC.$$
\end{prop}

\begin{proof}
For each $t$ the set
$$\Ic_t = \{Y \in \Mx \otimes \Mx^{op}: \Phi_Y(B) = 0\hbox{ for all }
B \in \Vc_t\}$$
is a left ideal of $\Mx \otimes \Mx^{op}$, and hence is of the form
$(\Mx \otimes \Mx^{op})P_t$ for some projection $P_t \in \Mx \otimes \Mx^{op}$.
Then $P_t \in \Ic_t$, so $\Phi_{P_t}(B) = 0$ for all $B \in \Vc_t$.
Conversely, $Y = YP_t$ for any $Y \in \Ic_t$, so that $\Phi_{P_t}(B) = 0$
implies $\Phi_Y(B) = \Phi_Y\Phi_{P_t}(B) = 0$ for all $Y \in \Ic_t$
implies $B \in \Vc_t$. Thus $\Vc_t = \{B: \Phi_{P_t}(B) = 0\}$. Finally,
the $P_t$ for $t \geq 0$ constitute a decreasing right continous
one-parameter family of projections in $\Mx \otimes \Mx^{op}$ and hence
are the spectral projections $P_{(t,\infty)}(X)$ for some positive operator
$X \in \Mx \otimes \Mx^{op}$.
\end{proof}

Now we proceed to our main result which gives a general intrinsic
characterization of quantum pseudometrics. Recall the abstract notion of
a ``quantum distance function'' from Definition \ref{qdfdef}. Also
recall that $D = D_\Vb$ is the displacement gauge associated to $\Vb$
(Proposition \ref{fildis}). We will give
a different intrinsic characterization in Corollary \ref{lipequiv}.

\begin{theo}\label{abch}
Let $\Mx \subseteq \Bc(H)$ be a von Neumann algebra. If
$\Vb$ is a quantum pseudometric on $\Mx$ then
$$\rho_\Vb(P,Q) =
\inf\{D(A): A \in \Bc(H)\hbox{ and }P(A \otimes I)Q \neq 0\}$$
(with $\inf \emptyset = \infty$)
is a quantum distance function on $\Mx$; conversely, if $\rho$ is
a quantum distance function on $\Mx$ then $\Vb_\rho =
\{\Vc^\rho_t\}$ with
$$\Vc^\rho_t = \{A \in \Bc(H): \rho(P,Q) > t\quad\Rightarrow\quad
P(A \otimes I)Q = 0\}$$
is a quantum pseudometric on $\Mx$. The two constructions are
inverse to each other.
\end{theo}

\begin{proof}
Let $\Vb$ be a quantum pseudometric on $\Mx$.
The fact that $\rho_\Vb$ is a quantum distance function
was proven in Proposition \ref{distances}.
Now let $\rho$ be any quantum distance function. We have
$\Mx' \subseteq \Vc^\rho_t$ for all $t$ by property (ii) of Definition
\ref{qdfdef} since
$$\rho(P,Q) > t\quad\Rightarrow\quad PQ = 0\quad\Rightarrow\quad
P(A \otimes I)Q = (A \otimes I)PQ = 0$$
for all $A \in \Mx'$. Also $\Vc^\rho_t$ is self-adjoint by property (iii) of
Definition \ref{qdfdef} and it is weak* closed because
$$P(A \otimes I)Q = 0\qquad\Leftrightarrow\qquad
\langle (A\otimes I)w,v\rangle = 0
\hbox{ for all }v \in {\rm ran}(P), w \in {\rm ran}(Q).$$
So each $\Vc^\rho_t$ is a dual operator system and $\Vc^\rho_0$ contains $\Mx'$.
The fact that $\Vc^\rho_t = \bigcap_{s > t} \Vc^\rho_s$ for all $t$
is easy. For the filtration condition, let $A \in \Vc^\rho_s$,
$B \in \Vc^\rho_t$, and $P,R \in \Pc$ and suppose $P(AB \otimes I)R \neq 0$.
We must show that $\rho(P,R) \leq s + t$.
Let $Q$ be the projection onto the closure of
$$(\Mx' \otimes I)({\rm ran}((B \otimes I)R)).$$
The range of $Q$ is invariant for $\Mx' \otimes I$ and hence
$Q$ belongs to $(\Mx' \otimes I)' = \Mx \overline{\otimes} \Bc(l^2)$.
We have $P(A \otimes I)Q \neq 0$ since $[(B \otimes I)R] \leq Q$ and
we have $\tilde{Q}(B \otimes I)R \neq 0$ for any $\tilde{Q} \in \Pc$
such that $Q\tilde{Q} \neq 0$ because
$$\tilde{Q}(B \otimes I)R = 0\quad \Rightarrow\quad
\tilde{Q}(CB \otimes I)R = 0\hbox{ for all }C \in \Mx'
\quad\Rightarrow\quad \tilde{Q}Q = 0.$$
Since $A \in \Vc^\rho_s$ and $B \in \Vc^\rho_t$, the above implies
that $\rho(P,R) \leq s + t$ by property (v) of Definition \ref{qdfdef}.
This shows that $AB \in \Vc^\rho_{s + t}$, and we conclude that
$\Vc^\rho_s \Vc^\rho_t \subseteq \Vc^\rho_{s + t}$. This
completes the proof that $\Vb_\rho$ is a quantum pseudometric on $\Mx$.

Now let $\Vb$ be any quantum pseudometric on $\Mx$, let $\rho = \rho_\Vb$,
and let $\tilde{\Vb} = \Vb_\rho$. The fact that $\Vb = \tilde{\Vb}$ is
just the content of Proposition \ref{recover}.

Finally, let $\rho$ be any quantum distance function, let
$\Vb = \Vb_\rho$, and let $\tilde{\rho} = \rho_\Vb$. Fix $t > 0$ and
define
$$\Rc_t = \{(P,Q) \in \Pc^2: \rho(P,Q) < t\}$$
and
$$\tilde{\Rc}_t = \{(P,Q) \in \Pc^2: \tilde{\rho}(P,Q) \leq t\}.$$
Then $\Rc_t$ is an open subset of $\Pc^2$ because its complement
is closed by property (vii) of
Definition \ref{qdfdef}. We have $(0,0) \not\in \Rc_t$ by
property (i) of Definition \ref{qdfdef},
$(\bigvee P_\lambda, \bigvee Q_\kappa) \in \Rc_t$
$\Leftrightarrow$ some $(P_\lambda, Q_\kappa) \in \Rc_t$ by the
comment following Definition \ref{qdfdef},
and $(P, [BQ]) \in \Rc_t$ $\Leftrightarrow$ $([B^*P],Q) \in \Rc_t$
for all $B \in I \otimes \Bc(l^2)$ by property (vi) of Definition
\ref{qdfdef}. Thus $\Rc_t$ is an intrinsic quantum relation
(Definition \ref{abstractqrel}) and we therefore have
\begin{eqnarray*}
\Rc_t &=& \{(P,Q) \in \Pc^2: (\exists A \in \Bc(H))\cr
&&\phantom{\liminf\liminf}(\rho(P',Q') \geq t \Rightarrow
P'(A \otimes I)Q' = 0\quad{\rm and}\quad P(A \otimes I)Q \neq 0\}
\end{eqnarray*}
by Theorem \ref{abstractchar}. Comparing this with the definition
of $\tilde{\Rc}_t$ then shows that $\Rc_t \subseteq \tilde{\Rc}_t
\subseteq \Rc_{t + \epsilon}$ for all $\epsilon > 0$. It follows that
$\rho(P,Q) = \tilde{\rho}(P,Q)$ for all $P$ and $Q$, i.e.,
$\rho = \tilde{\rho}$.
\end{proof}

\section{Examples}\label{exsn}

Our new definition of quantum metrics supports a wide variety of
examples. This is also true of the previously proposed definitions
mentioned in the introduction, and indeed the main classes of examples
in the different cases substantially overlap. If anything, a complaint
could be made that the previous definitions are too broad. In contrast,
our definition is sufficiently rigid to permit, for example, a simple
classification of all quantum metrics on $M_2(\Cb)$ (Proposition \ref{mtwo})
and a general analysis of translation-invariant quantum metrics on quantum
tori (Theorem \ref{qtori}).

The metric aspect of error correcting quantum codes provided the original
motivation behind our new approach. This connection is explained in
Section \ref{qhamming}. We are able to present a natural common
generalization of basic aspects of classical and quantum error
correction. Our theory also encompasses mixed classical/quantum
settings.

We present a small variety of interesting classes of examples. The
list could obviously be greatly expanded, but we have tried to give
a fair representation of the principal methods of construction.

\subsection{Operator systems}\label{opsys}
We begin our survey of examples with possibly the simplest natural class.
For any dual operator system $\Ac \subseteq \Bc(H)$ define
$$\Vc^\Ac_t =
\begin{cases}
\Cb I&\hbox{ if }0 \leq t < 1\cr
\Ac&\hbox{ if }1 \leq t < 2\cr
\Bc(H)&\hbox{ if }t \geq 2.
\end{cases}$$

The verification that $\Vb_\Ac = \{\Vc^\Ac_t\}$ is a quantum metric on
$\Mx = \Bc(H)$ is trivial. Despite their simplicity, the $\Vb_\Ac$
are already good for producing easy counterexamples.

\begin{exam}\label{counterdiam}
A quantum metric for which
$$\Vc_t \neq \{A \in \Bc(H): \rho(P,Q) > t\quad\Rightarrow\quad
PAQ = 0\},$$
with $P$ and $Q$ ranging over projections in $\Mx$ (cf.\
Proposition \ref{recover}), and furthermore
$${\rm diam}(\Vb) > \sup\{\rho(P,Q): \hbox{ $P$ and $Q$
are nonzero projections in }\Mx\}$$
(cf.\ Proposition \ref{diameter}).
Let $\Ac$ be a dual operator system properly contained in $\Bc(H)$ such
that for any nonzero $v,w \in H$ there exists $A \in \Ac$ with
$\langle Aw,v\rangle \neq 0$. For instance, we could take
$$\Ac = \{A \in \Bc(H): {\rm tr}(AB) = 0\}$$
where $B$ is any nonzero traceless self-adjoint trace class operator.
(Suppose $\langle Aw,v\rangle = 0$ for all $A \in \Ac$, with $v$ and
$w$ nonzero. Then $A \mapsto \langle Aw,v\rangle$ and $A \mapsto
{\rm tr}(AB)$ are nonzero linear functionals with the same kernel,
hence they are scalar multiples of each other. This implies that $B$
is a scalar multiple of the trace class operator
$u \mapsto \langle u,v\rangle w$, which contradicts the assumption that
it is traceless and self-adjoint.) Then $\Vb_\Ac$ has diameter 2 but
$\rho(P,Q) = 1$ for any nonzero projections $P,Q \in \Bc(H)$.
\end{exam}

\begin{exam}\label{countersup}
A pair of quantum metrics $\Vb_\Ac$ and $\Vb_\Bc$ on $\Bc(H)$ such that
the formula
$$\rho_{\Vb_\Ac \wedge \Vb_\Bc}(P,Q) =
\max\{\rho_{\Vb_\Ac}(P,Q),\rho_{\Vb_\Bc}(P,Q)\}$$
fails (cf.\ Propositions
\ref{trunc} and \ref{directsum}, where analogous formulas hold). Fix a pair of
orthogonal unit vectors $v$ and $w$
and find self-adjoint operators $A$ and $B$ such that
$$\langle Aw,v\rangle \neq 0 \neq \langle Bw,v\rangle$$
but $\Ac \cap \Bc = \Cb I$ where $\Ac = {\rm span}\{A,I\}$ and
$\Bc = {\rm span}\{B,I\}$. It is clear that $\Ac$ and $\Bc$ are dual
operator systems, and letting $P$ and $Q$ be the projections
onto $\Cb v$ and $\Cb w$, we have $\rho_{\Vb_\Ac}(P,Q) =
\rho_{\Vb_\Bc}(P,Q) = 1$ but $\rho_{\Vb_\Ac \wedge \Vb_\Bc}(P,Q) = 2$.
\end{exam}

It is worth noting that every quantum metric on $\Bc(H)$ for which the
range of the quantum distance function on linkable projections is contained
in $\{0,1,2\}$ is of the form $\Vb_\Ac$ for some dual operator system
$\Ac$. This follows from an easy and more general fact:

\begin{prop}
Let $\Vb$ be a quantum pseudometric on a von Neumann algebra
$\Mx \subseteq \Bc(H)$, let
$$L = \{0\} \cup \{t \in (0,\infty]: \Vc_{<t} \neq \Vc_t\}$$
(taking $\Vc_\infty = \Bc(H)$), and let
$$R = \{t \in [0,\infty): s > t \quad \Rightarrow\quad \Vc_t \neq \Vc_s\}.$$
Then the range of $\rho$ restricted to linkable pairs of projections
in $\Mx \overline{\otimes} \Bc(l^2)$ equals $L \cup R$.
\end{prop}

\begin{proof}
We proved in Proposition \ref{findist} that $\Vc_{<\infty} \neq \Vc_\infty$,
i.e., $\infty \in L$, if and only if $\rho(P,Q) = \infty$
for some pair of linkable projections $P$ and $Q$. This settles the
case $t = \infty$; for the rest of the proof assume $t$ is finite.

Suppose $t \not\in L \cup R$. Then $\Vc_t = \Vc_{<t}
= \Vc_{t + \epsilon}$ for some $\epsilon > 0$. For any pair
of projections $P,Q \in \Mx \overline{\otimes} \Bc(l^2)$, if
$P(A \otimes I)Q = 0$ for all $A \in \Vc_t$ then
$P(A \otimes I)Q = 0$ for all $A \in \Vc_{t + \epsilon}$ and hence
$\rho(P,Q) \geq t + \epsilon$. On the other hand, if
$P(A \otimes I)Q = 0$ for all $A \in \Vc_s$, for all $s < t$, then
$P(A \otimes I)Q = 0$ for all $A \in \overline{\bigcup_{s < t} \Vc_s}
= \Vc_t$; so $P(A \otimes I)Q \neq 0$ for some $A \in \Vc_t$
implies $P(A \otimes I)Q \neq 0$ for some $A \in \Vc_s$, for some $s < t$.
So in either case $\rho(P,Q) \neq t$, and we conclude that $t$ is not
in the range of $\rho$. This proves one inclusion.

$0$ belongs to the range of $\rho$ by Definition \ref{qdfdef} (ii).
If $t \in L$, $t \neq 0$, then $\Vc_{<t} \neq \Vc_t$ and by Lemma
\ref{separation} we
can find projections $P, Q \in \Mx \overline{\otimes} \Bc(l^2)$
such that $P(A \otimes I)Q \neq 0$ for some $A \in \Vc_t$ but
$P(B \otimes I)Q = 0$ for all $B \in \Vc_{<t}$. Thus
$\rho(P,Q) = t$. This shows that $L$ is contained in the range of $\rho$.

Finally, let $t \in R$ and let $n \in \Nb$. Then $\Vc_t \subsetneq
\Vc_{t + 1/n}$, so by Lemma \ref{separation} we can find projections
$P_n, Q_n \in \Mx \overline{\otimes} \Bc(l^2)$ such that
$P_n(A \otimes I)Q_n \neq 0$ for some $A \in \Vc_{t + 1/n}$ but
$P_n(B \otimes I)Q_n = 0$ for all $B \in \Vc_t$. This implies that
$t \leq \rho(P_n,Q_n) \leq t + 1/n$. Taking countable direct sums,
we get $\bigoplus P_n, \bigoplus Q_n
\in \bigoplus \Mx \overline{\otimes} \Bc(l^2) \cong \Mx \overline{\otimes}
\big(\bigoplus \Bc(l^2)\big) \subseteq \Mx \overline{\otimes}
\Bc(l^2)$ and
$$\rho\left(\bigoplus P_n, \bigoplus Q_n\right) = \inf\{\rho(P_n,Q_n)\} = t.$$
This shows that $R$ is contained in the range of $\rho$.
\end{proof}

\subsection{Graph metrics}\label{graphmetric}
Let $\Mx \subseteq \Bc(H)$ be a von Neumann algebra and let
$\Vc$ be a dual operator system that is a bimodule over $\Mx'$.
(Thus $\Vc$ is a {\it quantum graph} according to Definition 2.6 (d) of
\cite{W6}. For $\Mx$ a matrix algebra this definition appeared in \cite{DSW}.)
Then set $\Vc_t = \overline{\Vc\cdot \ldots \cdot \Vc}^{wk^*}$,
the weak* closure of the
algebraic product taken $[t]$ times, where $[t]$ is the greatest integer
$\leq t$ and with the convention that the empty product
is $\Mx'$. We call $\Vb = \{\Vc_t\}$ the {\it quantum graph metric}
associated to $\Vc$. We are most interested in the case where
$\Mx = \Bc(H)$, $\Vc_0 = \Cb I$, and $\Vc$ is any dual operator system
in $\Bc(H)$.

\begin{prop}
Let $\Mx \subseteq \Bc(H)$ be a von Neumann algebra and let
$\Vc$ be a dual operator system that is a bimodule over $\Mx'$.
Then the quantum graph metric is
the smallest quantum metric $\Vb$ on $\Mx$ such that $\Vc_1 = \Vc$.
\end{prop}

\begin{prop}
Let $X$ be a set and let $\Mx \cong l^\infty(X)$ be the von Neumann algebra
of bounded multiplication operators on $l^2(X)$.
Also let $R$ be a reflexive, symmetric relation on $X$,
let $\Vc = \Vc_R$ (Proposition \ref{atomiccase}), and let $\Vb$ be
the quantum graph metric on $\Mx$ associated to $\Vc$. Then $d_\Vb$
(Proposition \ref{aa}) is the graph metric associated to $R$.
\end{prop}

\begin{proof}
$R$ defines a graph with vertex set $X$ in the obvious way. Now
$$\Vc = \overline{\rm span}^{wk^*}\{V_{xy}: (x,y) \in R\}$$
and
\begin{eqnarray*}
\Vc^n &=& \overline{\rm span}^{wk^*}\{V_{xy}: (x,y) \in R^n\}\cr
&=& \overline{\rm span}^{wk^*}\{V_{xy}: \hbox{there is a path of
length $\leq n$ from $x$ to $y$}\}.
\end{eqnarray*}
Thus $d_\Vb(x,y)$ is the length of the shortest path from $x$ to $y$
(or $\infty$ if there is no such path), as desired.
\end{proof}

\subsection{Quantum metrics on $M_2(\Cb)$}
We analyze the possible quantum metrics on $\Mx = M_2(\Cb) = \Bc(\Cb^2)$.
Let $\sigma_x$, $\sigma_y$, and $\sigma_z$ be the Pauli spin matrices
$$\sigma_x = \left[
\begin{matrix}
1&0\cr
0&-1
\end{matrix}
\right]\quad \sigma_y = \left[
\begin{matrix}
0&1\cr
1&0
\end{matrix}
\right]\quad \sigma_z = \left[
\begin{matrix}
0&i\cr
-i&0
\end{matrix}
\right].$$

The only one dimensional operator system in $M_2(\Cb)$ is $\Cb I$.
Any two dimensional operator system $\Vc$ in $M_2(\Cb)$ must contain
$I$ and some non-scalar self-adjoint matrix $A$ and hence
must equal $\Cb I + \Cb A$. We can then choose an orthonormal
basis $\{f_1,f_2\}$
of $\Cb^2$ so that $A$ is diagonalized; then $\Vc$ will consist of
all diagonal matrices, $\Vc = \Cb I + \Cb \sigma_x$. Now let $\Wc$ be
a three dimensional operator system. It contains a two dimensional
operator system, without loss of generality (by the preceding case)
the diagonal matrices.
It also contains a non-diagonal self-adjoint operator $B$, without
loss of generality zero on the diagonal (since we can subtract off
its diagonal part), and then by replacing $f_2$ with $\alpha f_2$
for a suitable complex number $\alpha$ of modulus 1 we can take
$B$ to be a real scalar multiple of $\sigma_y$. So we have
$\Wc = \Cb I + \Cb \sigma_x + \Cb \sigma_y$. Finally, the only
four dimensional operator system in $M_2(\Cb)$ is $M_2(\Cb) =
\Cb I + \Cb\sigma_x + \Cb\sigma_y + \Cb\sigma_z$ itself. We can
now characterize all quantum metrics on $M_2(\Cb)$.

\begin{prop}\label{mtwo}
Let $0 \leq a \leq b \leq c \leq \infty$ satisfy $c \leq a + b$ and
let $\Vb = \{\Vc_t\}$ where
$$\Vc_t =
\begin{cases}
\Cb I&\hbox{ if }0 \leq t < a\cr
\Cb I + \Cb\sigma_x&\hbox{ if }a \leq t < b\cr
\Cb I + \Cb\sigma_x + \Cb\sigma_y&\hbox{ if }b \leq t < c\cr
\Cb I + \Cb\sigma_x + \Cb\sigma_y + \Cb\sigma_z&\hbox{ if }t \geq c.
\end{cases}$$
Then $\Vb$ is a quantum pseudometric on $M_2(\Cb)$. It is a quantum
metric if and only if $a > 0$ and it is reflexive if and only if
$b = c$. Up to a change of orthonormal
basis every quantum pseudometric on $M_2(\Cb)$ is of this form.
\end{prop}

\begin{proof}
It is elementary to check that $\Vb$ is a quantum pseudometric, that
it is a quantum metric if and only if $a > 0$, and that $\Cb I$ and
$\Cb I + \Cb \sigma_x$ are reflexive but $\Cb I + \Cb\sigma_x +
\Cb\sigma_y$ is not, so that $\Vb$ is reflexive if and only if $b = c$. Now
let $\Wb$ be any quantum pseudometric on $M_2(\Cb)$. Then the discussion
before the proposition shows that after a change of basis $\Wb$ must have
the given form for some $0 \leq a \leq b \leq c \leq \infty$. Furthermore,
we must have $c \leq a + b$ because
$$\Vc_a\Vc_b = (\Cb I + \Cb\sigma_x)(\Cb I + \Cb\sigma_x + \Cb\sigma_y)
\subseteq \Vc_{a+b}$$
and $i\sigma_x\sigma_y = \sigma_z$, so that $\Vc_c \subseteq \Vc_{a+b}$.
\end{proof}

\subsection{Quantum Hamming distance}\label{qhamming}
Fix a natural number $n$ and let $H = \Cb^{2^n} \cong
\Cb^2 \otimes \cdots \otimes \Cb^2$. If $\{e_0,e_1\}$ is the standard
orthonormal basis of $\Cb^2$ then
$$\{e_{i_1}\otimes \cdots \otimes e_{i_n}: \hbox{ each }
i_k = 0\hbox{ or }1\}$$
is an orthonormal basis for $H$. These basis vectors correspond to
binary strings of length $n$. Thus the information represented by
such a string can be encoded in an appropriate physical system as the
state modelled by the corresponding basis vector. For example, a single
photon has two basis polarization states, so a binary string of length
$n$ could be encoded as the polarization of an array of $n$ photons.
The basic difference with classical information is the physical
possibility of superposing base states. Thus $\Cb^2$ models
not a single bit of information, but rather a single ``quantum bit''
or ``qubit''.

In quantum error correction one is concerned with the possibility that
information encoded in this way could be corrupted. The quantum Hamming
metric measures the number of errors that might be introduced. Recall that
$[t]$ denotes the greatest integer $\leq t$.

\begin{defi}
The {\it quantum Hamming metric} on $M_{2^n}(\Cb)$ is the quantum metric
$\Vb_{Ham} = \{\Vc^{Ham}_t\}$ where
\begin{eqnarray*}
\Vc^{Ham}_t &=& {\rm span}\{A_1 \otimes \cdots \otimes A_n:
\hbox{ each } A_i \in M_2(\Cb)\cr
&&\phantom{\limsup}\hbox{ and $A_i = I_2$ for all but at most $[t]$ values of $i$}\}.
\end{eqnarray*}
\end{defi}

The quantum Hamming metric is a quantum graph metric (Section
\ref{graphmetric}). It is also the $n$-fold $l^1$ product (see the
comment following Proposition \ref{productprop}) with itself of the
quantum metric
$$\Vc_t =
\begin{cases}
\Cb I&\hbox{ if }0 \leq t < 1\cr
M_2(\Cb)&\hbox{ if }t \geq 1
\end{cases}$$
on $M_2(\Cb)$. The rough idea is that operators in $\Vc^{Ham}_t$ can corrupt
at most $[t]$ qubits of an $n$-qubit string. In more detail, if a basis
state $e_{i_1} \otimes \cdots \otimes e_{i_n}$ is acted on by an operator
in $\Vc_t^{Ham}$ and the resulting vector is subjected to a measurement
that projects it into a basis state, then the final state will differ
from the initial state in at most $[t]$ factors.

A {\it quantum code} is a subspace $C$ of $H$. It {\it corrects up to $k$
errors} if $PAP$ is a scalar multiple of $P$ for all $A \in \Vc^{Ham}_k$,
where $P$ is the projection onto $C$. Equivalently,
$PAv$ is a scalar multiple of $v$ for all $v \in C$. If the scalar is
nonzero, this means that any state in $C$ can be recovered
by projecting onto $C$.

The {\it volume bound} is a simple upper bound on the dimension of a
quantum code that corrects up to $k$ errors. For such a code write
$PAP = \varepsilon(A)P$ for all $A \in \Vc^{Ham}_k$. Then
$$\langle A,B\rangle = \varepsilon(B^*A)$$
is a positive semidefinite sesquilinear form on $\Vc^{Ham}_{k/2}$,
and if $K$ is the Hilbert space formed by factoring out null vectors
then we have an embedding of $K \otimes C$ into $H$ defined by
$$\bar{A} \otimes v \mapsto Av.$$
Thus
$${\rm dim}(C) \leq {\rm dim}(H)/{\rm dim}(K).$$

We generalize the above to arbitrary quantum metrics. One obvious
application would be to $\Mx = \bigoplus_{i=1}^k M_{2^{n_i}}(\Cb) \subseteq
M_{2^{n_1} + \cdots + 2^{n_k}}(\Cb)$ equipped with the quantum metric
defined by letting $\Vc_t$ be the span of the operators
$$(B_1^1 \otimes \cdots \otimes B_{n_1}^1) \oplus \cdots \oplus
(B_1^k \otimes \cdots \otimes B_{n_k}^k)$$
such that $B^i_j = I_2$ for all but at most $[t]$ values of $i$ and $j$.
This generalizes the quantum Hamming metric to a mixed classical/quantum
system in which information
is encoded in $k$ disentangled quantum packets.

However, we can actually state a version of the volume bound for
arbitrary quantum metrics. A classical code that corrects up to $k$
errors is a subset of the space of all strings any two distinct elements
of which are more than $2k$ units apart; we need a general version of
this condition. First we state it in its most useful form, and then we
prove that our definition is equivalent to two other possibly more
intuitive characterizations.

\begin{defi}
Let $\Vb$ be a quantum metric on a von Neumann algebra $\Mx \subseteq
\Bc(H)$ and let $P$ be a projection in $\Mx$.

\noindent (a) The {\it minimum distance in $P$} is the value
$$\delta(P) = \sup\{t\geq 0: P \Vc_t P = P\Vc_0 P\}.$$

\noindent (b) The {\it induced quantum pseudometric} on $P\Mx P$ is the
smallest quantum pseudometric $\tilde{\Vb} = \{\tilde{\Vc}_t\}$ on
$P\Mx P \subseteq \Bc(K)$, where $K = {\rm ran}(P)$, such that
$P\Vc_t P \subseteq \tilde{\Vc}_t$ for all $t$.
\end{defi}

(The definition of an induced quantum pseudometric generalizes our
definition of metric quotients in Definition \ref{quot}.)

\begin{prop}\label{mindist}
Let $\Vb$ be a quantum metric on a von Neumann algebra $\Mx \subseteq
\Bc(H)$ and let $P$ be a projection in $\Mx$. Suppose $\delta(P) > 0$.
Then
\begin{eqnarray*}
\delta(P)
&=& \sup\{t \geq 0: \tilde{\Vc}_t = \tilde{\Vc}_0\}\cr
&=& \inf\{\rho(P_1,P_2): P_1, P_2 \leq P\otimes I\hbox{ and }
\rho(P_1,P_2) > 0\}
\end{eqnarray*}
where $\tilde{\Vb}$ is the induced pseudometric on $P\Mx P$ and $P_1$
and $P_2$ range over projections in $\Mx \overline{\otimes} \Bc(l^2)$.
\end{prop}

\begin{proof}
Let $K = {\rm ran}(P)$.
Observe first that $\Vc_0 = \Mx'$ implies $P\Vc_0 P = P\Mx'$, and that
$P\Mx' = (P\Mx P)'$ where the commutant on the right is taken in $\Bc(K)$.
(The inclusion $\subseteq$ is easy and the inclusion $\supseteq$
follows from the double commutant theorem.)

Now
$$\tilde{\Wc}_t =
\begin{cases}
P\Vc_0 P&\hbox{ if }0 \leq t < \delta(P)\cr
\Bc(K)&\hbox{ if }t \geq \delta(P)
\end{cases}$$
is a quantum pseudometric on $P\Mx P$ with $P\Vc_tP = P\Vc_0 P
\subseteq \tilde{\Wc}_t$ for all $t < \delta(P)$, which shows that
$\tilde{\Vc}_t \subseteq \tilde{\Wc}_t$ for all $t$ and hence that
$\tilde{\Vc}_t = P\Vc_0 P = \tilde{\Vc}_0$ for all $t < \delta(P)$.
Conversely, if $t > \delta(P)$ then
$$\tilde{\Vc}_0 = P\Vc_0 P \subsetneq P\Vc_t P \subseteq \tilde{\Vc}_t,$$
so $\tilde{\Vc}_0 \neq \tilde{\Vc}_t$.
This proves the first equality.

For the second equality first let $P_1$ and $P_2$ be any two projections
in $\Mx \overline{\otimes} \Bc(l^2)$ with
$P_1, P_2 \leq P \otimes I$ and suppose $\rho(P_1,P_2) < \delta(P)$.
Then there exists $A \in \Bc(H)$ such that $P_1(A \otimes I)P_2 \neq 0$
and $D(A) < \delta(P)$. But then $PAP = PBP$ for some $B \in \Vc_0$
and we have
$$P_1(B \otimes I)P_2 =
P_1(PBP\otimes I)P_2 = P_1(PAP\otimes I)P_2 = P_1(A \otimes I)P_2 \neq 0$$
so that $\rho(P_1,P_2) = 0$. Conversely, if $s < \delta(P) < t$
then $P\Vc_s P = P\Mx' \subsetneq P\Vc_t P$ and so by Lemma \ref{separation}
there exist $P_1, P_2 \in P\Mx P\overline{\otimes} \Bc(l^2)$ satisfying
$P_1(PAP\otimes I)P_2 \neq 0$ for some $A \in \Vc_t$ but
$P_1(PBP\otimes I)P_2 = 0$ for all $B \in \Vc_s$. Regarding
$P_1$ and $P_2$ as elements of $\Mx \overline{\otimes} \Bc(l^2)$, we
then have $P_1,P_2 \leq P \otimes I$ and
$P_1(A\otimes I)P_2 \neq 0$ for some $A \in \Vc_t$ but
$P_1(B \otimes I)P_2 = 0$ for all $B \in \Vc_s$, so
that $s \leq \rho(P_1,P_2) \leq t$. This suffices to prove the
second equality.
\end{proof}

The condition that $\delta(P) > 0$ in Proposition \ref{mindist} is
necessary. Even if $P\Vc_0P \neq P\Vc_tP$ for all $t > 0$ we could
still have $\tilde{\Vc}_0 = \tilde{\Vc}_t$ for some $t > 0$ because
the condition $\tilde{\Vc}_0 = \bigcap_{t > 0} \tilde{\Vc}_t$ could
force $P\Vc_0P \subsetneq \tilde{\Vc}_0$.

Now let $\Vb$ be a quantum metric on a von Neumann algebra $\Mx$ and
let $P$ be a projection in $\Mx$ with minimum distance $\delta(P) > 0$.
We have a positive semidefinite sesquilinear form on $\Vc_{<\delta(P)/2}$
with values in $P \Vc_0 P = P\Mx' = (P\Mx P)'$ defined by
$$\langle A,B\rangle = PB^*AP$$
for $A,B \in \Vc_{<\delta(P)/2}$.
Let $\Ec$ be the vector space formed by factoring out null vectors.
Now $\Vc_{<\delta(P)/2}$ is a right $\Mx'$-module,
and this descends to a right action of $P\Mx'$ on $\Ec$. So $\Ec$
is a right pre-Hilbert $P\Mx'$-module. Letting $K = {\rm ran}(P)$, we
can then define an inner product on $\Ec \otimes_{P\Mx'} K$ by
$$\langle A\otimes v, B\otimes w\rangle =
\big\langle \langle A,B\rangle v, w\big\rangle.$$
Let $\Ec \overline{\otimes}_{P\Mx'} K$ denote the completion of
$\Ec\otimes_{P\Mx'} K$ for this inner product.
The following result gives a general version of the volume bound.

\begin{theo}
Let $\Vb$ be a quantum metric on a von Neumann algebra $\Mx$,
let $P$ be a projection in $\Mx$ with minimum distance $\delta(P) > 0$,
let $\Ec$ be the pre-Hilbert $P\Mx'$-module formed from $\Vc_{\delta(P)/2}$
as above, and let $K = {\rm ran}(P)$. Then the map
$$\bar{A} \otimes v \mapsto Av$$
extends to an isometric isomorphism of $\Ec \overline{\otimes}_{P\Mx'} K$
onto ${\rm ran}((P)_{\delta(P)/2}) \subseteq H$, where $(P)_{\delta(P)/2}$
is the open $\delta(P)/2$-neighborhood of $P$ (Definition \ref{various}
(b)).
\end{theo}

\begin{proof}
The essential computation is
\begin{eqnarray*}
\langle \bar{A}\otimes v, \bar{B}\otimes w\rangle
&=& \big\langle \langle A,B\rangle v,w\big\rangle\cr
&=& \langle PB^*APv, w\rangle\cr
&=& \langle Av, Bw\rangle
\end{eqnarray*}
for $A,B \in \Vc_{<\delta(P)/2}$ and $v,w \in {\rm ran}(P)$.
Taking linear combinations shows that the map $\bar{A} \otimes v
\mapsto Av$ is well-defined and isometric on the uncompleted version of the
construction, and taking completions then yields an isometry from the
completed tensor product onto $\bigvee_{s < \delta(P)/2}
\overline{\Vc_s({\rm ran}(P))} = {\rm ran}((P)_{\delta(P)/2})$.
\end{proof}

\subsection{Quantum tori}
We formulate a notion of translation invariant quantum pseudometrics
on quantum tori. Using the analysis of translation invariant quantum
relations on quantum tori developed in Section 2.7 of \cite{W6}, it is
then straightforward to deduce strong structural information about
translation invariant quantum pseudometrics.

Quantum tori are the simplest examples of noncommutative manifolds.
They are related to the quantum plane, which plays the role of the
phase space of a spinless one-dimensional particle.
The classical version of such a system has phase space $\Rb^2$, with
the point $(q,p) \in \Rb^2$ representing a state with position $q$ and
momentum $p$, so that the position and momentum observables are just
the coordinate functions on phase space. When such a system is quantized
the position and momentum observables are modelled by unbounded
self-adjoint operators $Q$ and $P$ satisfying $QP - PQ = i\hbar I$.
Polynomials in $Q$ and $P$ can then be seen as a quantum
analog of polynomial functions on $\Rb^2$. The quantum analog of the
continuous functions on the torus --- equivalently, the
$(2\pi, 2\pi)$-periodic continuous functions on the plane --- is the
C*-algebra generated by the unitary operators $e^{iQ}$ and $e^{iP}$,
which satisfy the commutation relation $e^{iQ}e^{iP} =
e^{-i\hbar}e^{iP}e^{iQ}$. For more background see \cite{Rie2} or
Sections 4.1, 4.2, 5.5, and 6.6 of \cite{W5}.

Let $\Tb = \Rb/2\pi\Zb$ and fix $\hbar \in \Rb$. Let $\{e_{m,n}\}$ be
the standard basis of $l^2(\Zb^2)$. We model the quantum
tori on $l^2(\Zb^2)$ as follows.

\begin{defi}\label{qtdf}
Let $U_\hbar$ and $V_\hbar$ be the unitaries in $\Bc(l^2(\Zb^2))$ defined by
\begin{eqnarray*}
U_\hbar e_{m,n} &=& e^{-i\hbar n/2} e_{m+1,n}\cr
V_\hbar e_{m,n} &=& e^{i\hbar m/2} e_{m,n+1}.
\end{eqnarray*}
The {\it quantum torus von Neumann algebra} for the given value of
$\hbar$ is the von Neumann algebra
$W^*(U_\hbar,V_\hbar)$ generated by $U_\hbar$ and $V_\hbar$.
\end{defi}

The commutant of $W^*(U_\hbar,V_\hbar)$ is $W^*(U_{-\hbar}, V_{-\hbar})$
(\cite{W6}, Corollary 2.38).

If $\hbar$ is an irrational multiple of $\pi$ then $W^*(U_\hbar,V_\hbar)$ is a
hyperfinite $II_1$ factor. We will not need this fact.

Conjugating $U_\hbar$ and $V_\hbar$ by the Fourier transform
$\Fc: L^2(\Tb^2) \to l^2(\Zb^2)$ yields the operators
\begin{eqnarray*}
\hat{U}_\hbar f(x,y) &=& e^{ix} f\left(x,y-\frac{\hbar}{2}\right)\cr
\hat{V}_\hbar f(x,y) &=& e^{iy} f\left(x + \frac{\hbar}{2},y\right)
\end{eqnarray*}
on $L^2(\Tb^2)$, with $W^*(\hat{U}_\hbar,\hat{V}_\hbar)$
reducing to the algebra of bounded
multiplication operators when $\hbar = 0$. However, for our purposes
the $l^2(\Zb^2)$ picture is more convenient.

\begin{defi}\label{fdef}
Let $A \in \Bc(l^2(\Zb^2))$.

\noindent (a) For $x,y \in \Tb$ define
$$\theta_{x,y}(A) = M_{e^{i(mx + ny)}} A M_{e^{-i(mx + ny)}}.$$

\noindent (b) For $k,l \in \Zb$ define
$$A_{k,l} = \frac{1}{4\pi^2} \int_0^{2\pi} \int_0^{2\pi}
e^{-i(kx + ly)} \theta_{x,y}(A)\, dxdy.$$
We call $A_{k,l}$ the {\it $(k,l)$ Fourier term} of $A$.

\noindent (c) For $k,l \in \Nb$ define
$$S_{k,l}(A) = \sum_{|k'| \leq k, |l'| \leq l} A_{k',l'}$$
and for $N \in \Nb$ define
$$\sigma_N(A) = \frac{1}{N^2} \sum_{0 \leq k,l \leq N-1} S_{k,l}(A).$$
\end{defi}

In the $L^2(\Tb^2)$ picture the operator $M_{e^{i(mx + ny)}}$ on
$l^2(\Zb^2)$ becomes translation by $(-x,-y)$, so that $\theta_{x,y}$ is
conjugation by a translation.

The integral used to define $A_{k,l}$ can be understood in a weak
sense: for any vectors $v,w \in l^2(\Zb^2)$ we take
$\langle A_{k,l}w,v\rangle$ to be $\frac{1}{4\pi^2} \int_0^{2\pi}\int_0^{2\pi}
e^{-i(kx + ly)} \langle \theta_{x,y}(A)w,v\rangle\, dxdy$.
In particular, if $w = e_{m,n}$ and $v = e_{m',n'}$ then we have
$$\langle A_{k,l} e_{m,n}, e_{m',n'}\rangle =
\begin{cases}
\langle Ae_{m,n}, e_{m',n'}\rangle&\hbox{if $m' = m + k$ and
$n' = n + l$}\cr
0&\hbox{otherwise}.
\end{cases}\eqno{(*)}$$
The $A_{k,l}$ are something like Fourier coefficients, the $S_{k,l}(A)$
like partial sums of a Fourier series, and the $\sigma_N(A)$ like
Ces\`{a}ro means.

\begin{defi}\label{qtdefs}
Let $\Mx \cong l^\infty(\Zb^2)$ be the von Neumann
algebra of bounded multiplication operators in $\Bc(l^2(\Zb^2))$.

(a) For any weak* closed subspace $\Ec$ of $\Mx$ that is invariant
under the natural action of $\Zb^2$, define $\Vc_\Ec$ by
$$\Vc_\Ec = \{A \in \Bc(l^2(\Zb^2)):
A_{k,l} \in \Ec\cdot U_{-\hbar}^kV_{-\hbar}^l\hbox{ for all }k,l \in \Zb\}.$$

(b) For any closed subset $S \subseteq \Tb^2$ define
\begin{eqnarray*}
\Ec_0(S) &=& \{M_f: f \in l^\infty(\Zb^2)\hbox{ and }\sum f\bar{g} = 0
\hbox{ for all } g \in l^1(\Zb^2)\hbox{ such that }\hat{g}|_S = 0\}\cr
&=& \overline{\rm span}^{wk^*}\{M_{e^{i(mx + ny)}}: (x,y) \in S\}
\end{eqnarray*}
and
\begin{eqnarray*}
\Ec_1(S) &=& \{M_f: f \in l^\infty(\Zb^2)\hbox{ and }\sum f\bar{g} = 0
\hbox{ for all } g \in l^1(\Zb^2)\hbox{ such that }\hat{g}|_T = 0\cr
&&\phantom{\limsup\limsup}\hbox{for some neighborhood $T$ of $S$}\}\cr
&=& \bigcap_{\epsilon > 0} \Ec_0(N_\epsilon(S))
\end{eqnarray*}
where $N_\epsilon(S)$ is the open $\epsilon$-neighborhood of $S$.
\end{defi}

In part (a), $\Vc_\Ec$ is a quantum relation on $W^*(U_\hbar, V_\hbar)$
satisfying $\theta_{x,y}(\Vc_\Ec) = \Vc_\Ec$ for all $x,y \in \Tb$ by
(\cite{W6}, Theorem 2.41). This suggests the following definition.

\begin{defi}
A quantum pseudometric $\Vb = \{\Vc_t\}$ on
$W^*(U_\hbar,V_\hbar)$ is {\it translation invariant} if
$\theta_{x,y}(\Vc_t) = \Vc_t$ for all $x,y \in \Tb$ and all $t \in
[0,\infty)$.
\end{defi}

We can characterize translation invariance intrinsically as follows.

\begin{prop}
Let $\Vb$ be a quantum pseudometric on $W^*(U_\hbar,V_\hbar)$
and let $\rho$ be the associated quantum distance function (Definition
\ref{projdist}). Then $\Vb$ is translation invariant if and only if
$$\rho(P,Q) = \rho((\theta_{x,y} \otimes I)(P), (\theta_{x,y}\otimes I)(Q))$$
for all $x,y \in \Tb$ and all projections
$P,Q \in W^*(U_\hbar,V_\hbar) \overline{\otimes} \Bc(l^2)$.
\end{prop}

\begin{proof}
The forward implication follows from (\cite{W6}, Proposition 2.40)
and the fact that
$$\rho(P,Q) =  \inf\{t: (P,Q) \in \Rc_t\}$$
where $\Rc_t$ is the intrinsic quantum relation corresponding to $\Vc_t$
(Theorem \ref{abstractchar}). For the converse, suppose $\Vb$ is not
translation invariant and find $t \in [0,\infty)$, $x,y \in \Tb$, and
$A \in \Vc_t$ such that $\theta_{x,y}(A) \not\in \Vc_t$. Then
$\theta_{x,y}(A) \not\in \Vc_s$ for some $s > t$, and
by Lemma \ref{separation} we can find projections $P, Q \in
W^*(U_\hbar,V_\hbar) \overline{\otimes} \Bc(l^2)$ such that
$P(\theta_{x,y}(A)\otimes I)Q \neq 0$ but $P(B \otimes I)Q = 0$ for
all $B \in \Vc_s$. Then $(P,Q) \not\in \Rc_s$ but
$((\theta_{-x,-y}\otimes I)(P),(\theta_{-x,-y}\otimes I)(Q)) \in \Rc_t$
because
$$(\theta_{-x,-y}\otimes I)(P)(A \otimes I)(\theta_{-x,-y}\otimes I)(Q)
= (\theta_{-x,-y}\otimes I)(P(\theta_{x,y}(A)\otimes I)Q) \neq 0.$$
Thus
$$\rho((\theta_{-x,-y}\otimes I)(P),(\theta_{-x,-y}\otimes I)(Q))
\leq t < s \leq \rho(P,Q).$$
This proves the reverse implication.
\end{proof}

Say that a pseudometric $d$ on $\Tb^2$ is {\it closed} if
$$\overline{N}_\epsilon(x,y) = \{(x',y') \in \Tb^2:
d((x,y), (x',y')) \leq \epsilon\}$$
is closed in the usual topology, for every $(x,y) \in \Tb^2$ and
every $\epsilon > 0$. The following basic structural result for
translation invariant quantum pseudometrics on quantum tori
follows immediately from Theorem 2.41 and Corollary 2.42 of \cite{W6}.

\begin{theo}\label{qtori}
Let $\Mx \cong l^\infty(\Zb^2)$ be the von Neumann
algebra of bounded multiplication operators in $\Bc(l^2(\Zb^2))$.

\noindent (a) If $\Vb_0 = \{\Vc^0_t\}$ is a translation invariant quantum
pseudometric on $W^*(U_0, V_0) \cong L^\infty(\Tb^2)$ then $\Vb_\hbar =
\{\Vc^\hbar_t\}$ is a translation invariant quantum pseudometric on
$W^*(U_\hbar,V_\hbar)$, where $\Vc^\hbar_t = \Vc_{\Ec_t}$
with $\Ec_t = \Vc^0_t \cap \Mx$. Every translation
invariant quantum pseudometric on $W^*(U_\hbar, V_\hbar)$ is of this form.

\noindent (b) Let $\Vb = \{\Vc_t\}$ be a translation invariant quantum
pseudometric on $W^*(U_\hbar, V_\hbar)$. Then
$$d((x,y), (x',y')) = \inf\{t: M_{e^{i(m(x-x') + n(y-y'))}} \in \Vc_t\}$$
is a closed translation invariant pseudometric on $\Tb^2$ and
$\Vb_0 = \{\Vc^0_t\}$ and $\Vb_1 = \{\Vc^1_t\}$ are translation
invariant quantum pseudometrics on $W^*(U_\hbar, V_\hbar)$ where
\begin{eqnarray*}
\Vc^0_t &=& \Vc_{\Ec_0(S_t)}\cr
\Vc^1_t &=& \Vc_{\Ec_1(S_t)}
\end{eqnarray*}
with $S_t = \{(x,y) \in \Tb^2: d((0,0), (x,y)) \leq t\}$. We have
$\Vb_0 \leq \Vb \leq \Vb_1$.
\end{theo}

We would like to say that the W*-filtration $\Vb_\hbar$ in Theorem
\ref{qtori} (a) converges to the W*-filtration $\Vb_r$ as $\hbar \to r$.
It is easy to see that convergence does
not occur in the Gromov-Hausdorff sense (see the comment following Definition
\ref{various}). The right notion of convergence seems to be the following.
Denote the closed unit ball of any Banach space $\Vc$ by $[\Vc]_1$.

\begin{defi}\label{localconv}
Let $\{\Vb_\lambda\}$ be a net of W*-filtrations of $\Bc(H)$. We say
that $\{\Vb_\lambda\}$ {\it locally converges} to a W*-filtration
$\Vb$ of $\Bc(H)$ if for every $0 \leq s < t$ and every weak* open
neighborhood $U$ of $0 \in \Bc(H)$ we eventually have
$$[\Vc^\lambda_s]_1 \subseteq [\Vc_t]_1 + U \qquad{\rm and}\qquad
[\Vc_s]_1 \subseteq [\Vc^\lambda_t]_1 + U.$$
\end{defi}

Equivalently, for any $\epsilon > 0$ and any vectors $v_1, \ldots, v_n,
w_1, \ldots, w_n \in H$ the sets $[\Vc_s]_1$
and $[\Vc^\lambda_s]_1$ are, respectively, eventually
within the $\epsilon$-neighborhoods of the sets
$[\Vc^\lambda_t]_1$ and $[\Vc_t]_1$ for the seminorm
$$| \! | \! |A| \! | \! | = \sum |\langle Aw_i, v_i\rangle|.$$
Indeed, it would be enough to show this with the $v_i$ and $w_i$ ranging
over a spanning subset of $H$.
The next result is an easy consequence of this characterization.

\begin{prop}
Let $\{d_\lambda\}$ be a net of pseudometrics on a set $X$. Then the
corresponding quantum pseudometrics $\Vb_\lambda$ (Proposition \ref{aa})
locally converge to the quantum pseudometric $\Vb$ corresponding to a
pseudometric $d$ on $X$ if and only if $d_\lambda(x,y) \to d(x,y)$ for
all $x,y \in X$.
\end{prop}

\begin{proof}
Suppose $d_\lambda(x,y) \not\to d(x,y)$ for some $x,y \in X$. Then
either $\limsup d_\lambda(x,y) > d(x,y)$ or $\liminf d_\lambda(x,y)
< d(x,y)$. Suppose the former and let $s = d(x,y)$ and $s < t <
\limsup d_\lambda(x,y)$. Also let $v_1 = e_x$ and $w_1 = e_y$. Then
the rank one operator $V_{xy}$ belongs to $[\Vc_s]_1$, but for any
$\lambda$ with $d_\lambda(x,y) > t$ we have $\langle Ae_y,e_x\rangle =0$
for all $a \in \Vc^\lambda_t$; this shows that $V_{xy}$ is frequently
not approximated by operators in $\Vc^\lambda_t$ for the seminorm
$|\langle Aw_1,v_1\rangle|$. So $\Vb_\lambda$ does not locally converge
to $\Vb$. In the second case ($\liminf d_\lambda(x,y) < d(x,y)$) the
same proof works, now interchanging the roles of $\Vb$ and $\Vb_\lambda$.

Conversely, suppose $d_\lambda(x,y) \to d(x,y)$ for all $x,y \in X$.
We verify the condition stated just before the proposition with the
vectors $v_i$ and $w_i$ ranging over the standard basis $\{e_x\}$ of
$l^2(X)$. So let $v_1, \ldots, v_n, w_1, \ldots, w_n$ be basis vectors
and find a finite set $S \subseteq X$ on which they are all supported.
Fix $0 \leq s < t$. Then eventually we have $d_\lambda(x,y) \leq t$
for all $x,y \in S$ with $d(x,y) \leq s$, so that if $A \in [\Vc_s]_1$
then $M_{\chi_S}AM_{\chi_S} \in [\Vc^\lambda_t]_1$, and $| \! | \! |A -
M_{\chi_S}AM_{\chi_S}| \! | \! | = 0$ for the seminorm $| \! | \! |\cdot| \! | \! |$. A
similar argument shows that $[\Vc^\lambda_s]_1$ is eventually within
the $\epsilon$-neighborhood of $[\Vc_t]_1$ for the seminorm $| \! | \! |\cdot| \! | \! |$,
for all $\epsilon > 0$. We conclude that $\Vb_\lambda$ locally converges
to $\Vb$.
\end{proof}

Now we show that translation invariant quantum pseudometrics on quantum
torus von Neumann algebras converge as the parameter $\hbar$ varies.

\begin{theo}\label{cvgs}
Let $\Vb_0$ be a translation invariant quantum pseudometric on
$W^*(U_0,V_0)$ and for each $\hbar \in \Rb$ let $\Vb_\hbar$ be the
corresponding translation invariant quantum pseudometric on
$W^*(U_\hbar, V_\hbar)$ (Theorem \ref{qtori} (a)). Let $r \in \Rb$.
Then $\Vb_\hbar$ locally converges to $\Vb_r$ as $\hbar \to r$.
\end{theo}

\begin{proof}
We use the alternative characterization of local convergence given
following Definition \ref{localconv}, with the $v_i$ and $w_i$ ranging
over the standard basis $\{e_{m,n}\}$ of $l^2(\Zb^2)$. We will show
that $[\Vc^r_s]_1$ is eventually within the $\epsilon$-neighborhood
of $[\Vc^\hbar_s]_1$ for the seminorm $| \! | \! |\cdot| \! | \! |$; the corresponding
assertion with $r$ and $\hbar$ switched is proven similarly.

Let $v_1, \ldots, v_n, w_1, \ldots, w_n$ be basis vectors and find a
finite subset $S \subset \Zb^2$ on which they are all supported.
Then the inner products
$\langle U_{-\hbar}^kV_{-\hbar}^lw_i, v_i\rangle$
will converge to the inner products
$\langle U_{-r}^kV_{-r}^lw_i, v_i\rangle$
as $\hbar \to r$, uniformly in $k$ and $l$ since these inner products
are zero for sufficiently large $k$ and $l$.

Next let $\epsilon > 0$ and observe that the inner products
$\langle \sigma_N(A) w_i, v_i\rangle$ converge to the inner
products $\langle A w_i, v_i\rangle$ as $N \to \infty$, uniformly in
$A \in [\Vc^r_s]_1$. Also, $A \in [\Vc^r_s]_1$ implies $\sigma_N(A) \in
[\Vc^r_s]_1$ by (\cite{W6}, Proposition 2.35).
By the preceding observation, if we write $\sigma^\hbar_N(A) =
\sum a_{k,l}U^k_{-\hbar}V^l_{-\hbar} \in \Vc^\hbar_s$ where $\sigma_N(A) =
\sum a_{k,l}U^k_{-r}V^l_{-r}$ then we will have
$$| \! | \! |\sigma^\hbar_N(A) - A| \! | \! | \leq
| \! | \! |\sigma^\hbar_N(A) - \sigma_N(A)| \! | \! | +
| \! | \! |\sigma_N(A) - A| \! | \! |$$
with the first term on the right going to zero as $\hbar \to r$
and the second going to zero as $N \to \infty$, both
uniformly in $A \in [\Vc^r_s]_1$. We just need
$\limsup \|\sigma^\hbar_N(A)\| \leq 1$, uniformly in $A \in
[\Vc^\hbar_s]_1$. This follows from the fact that the C*-algebras
generated by $U_\hbar$ and $V_\hbar$ form a continuous field
\cite{Rie1}, so that the norms of polynomials in $U_\hbar$ and
$V_\hbar$ vary continuously in $\hbar$.
\end{proof}

\subsection{H\"older metrics}\label{hold}
Let $\Vb = \{\Vc_t\}$ be any quantum pseudometric on a von Neumann
algebra $\Mx \subseteq \Bc(H)$ and let $0 < \alpha < 1$. Then
$\Vb^\alpha = \{\Vc_t^\alpha\}$ where $\Vc_t^\alpha = \Vc_{t^{1/\alpha}}$ is
also a quantum pseudometric on $\Mx$; the filtration condition
follows from the inequality
$s^{1/\alpha} + t^{1/\alpha} \leq (s + t)^{1/\alpha}$ for $s,t \geq 0$.
We call $\Vb^\alpha$ a {\it H\"older} or {\it snowflake} quantum
pseudometric. It is easy to see that this construction reduces to
the usual one in the atomic abelian case:

\begin{prop}
Let $X$ be a set and let $\Mx \cong l^\infty(X)$ be the von Neumann
algebra of bounded multiplication operators on $l^2(X)$. If $d$ is a
pseudometric on $X$ with corresponding quantum pseudometric $\Vb_d$
(Proposition \ref{aa}) and $0 < \alpha < 1$ then the quantum pseudometric
$\Vb_{d^\alpha}$ corresponding to the pseudometric $d^\alpha$ equals
$\Vb^\alpha_d$.
\end{prop}

As in the classical case, H\"older metrics have some pathological
qualities. For example, they are essentially never path metrics
(Definition \ref{various} (e)).

\begin{prop}
Let $\Vb$ be a quantum pseudometric on a von Neumann algebra $\Mx$
and let $0 < \alpha < 1$. Suppose that $\Vc_t \neq \Vc_0$ for
some $t > 0$. Then $\Vb^\alpha$ is not a path quantum pseudometric.
\end{prop}

\begin{proof}
Choose $t > 0$ such that $\Vc_t \neq \Vc_0$ and let $t_0 =
\inf\{s: \Vc_s = \Vc_t\}$. So $t_0 > 0$ and $\Vc_{t_0} \neq \Vc_s$
for any $s < t_0$. Now
$$\Vc^\alpha_{t_0^\alpha/2}\Vc^\alpha_{t_0^\alpha/2}
= \Vc_{(t_0^\alpha/2)^{1/\alpha}}\Vc_{(t_0^\alpha/2)^{1/\alpha}}
\subseteq \Vc_{2(t_0^\alpha/2)^{1/\alpha}}
= \Vc_{2^{1-1/\alpha}t_0}$$
and $2^{1-1/\alpha}t_0 < t_0$, so
for sufficiently small $\epsilon > 0$ we will have
$$\Vc^\alpha_{(t_0^\alpha/2) + \epsilon}
\Vc^\alpha_{(t_0^\alpha/2) + \epsilon} \subseteq \Vc_s$$
for some $s < t_0$. Thus
$$\Vc^\alpha_{(t_0^\alpha/2) + \epsilon}\Vc^\alpha_{(t_0^\alpha/2) + \epsilon}
\subseteq \Vc_s \subsetneq \Vc_{t_0} = \Vc^\alpha_{t_0^\alpha}$$
and this shows that $\Vb^\alpha$ is not a path quantum pseudometric.
\end{proof}

The following result gives a more general version of the H\"older
construction.

\begin{prop}
Let $\Vb = \{\Vc_t\}$ be a quantum pseudometric on a von Neumann
algebra $\Mx \subseteq \Bc(H)$ and let $f: [0,\infty) \to [0,\infty]$
be a right continuous nondecreasing function such that
$f(s) + f(t) \leq f(s + t)$ for all $s,t \geq 0$.
Then $\Vb^f = \{\Vc^f_t\}$  where $\Vc^f_t = \Vc_{f(t)}$ is a
quantum pseudometric on $\Mx$
(taking $\Vc_\infty = \Bc(H)$). If $\rho$ is the quantum distance function
associated to $\Vb$ and $\rho^f$ is the quantum distance function
associated to $\Vb^f$ then we have
$$\rho^f(P,Q) = f(\rho(P,Q))$$
for all projections $P,Q \in \Mx \overline{\otimes} \Bc(l^2)$.
\end{prop}

The proof is trivial. The truncation of a quantum pseudometric (Definition
\ref{easycon} (a)) is also a special case of this construction. (Define
$f(t) = t$ for $t < C$ and $f(t) = \infty$ for $t \geq C$.)

\subsection{Spectral triples}
A {\it spectral triple} is a triple $(\Ac, H, D)$ consisting of a Hilbert
space $H$, a unital $*$-algebra $\Ac \subseteq \Bc(H)$, and a self-adjoint
operator $D$ with compact resolvent, such that $[D,A]$ is bounded for
all $A \in \Ac$ (\cite{Con}, Definition IV.2.11). 

We will not actually use the hypothesis that $D$ has compact resolvent,
and indeed there are natural examples where this fails. (For example,
$H = L^2(\Rb)$, $\Ac = {\rm Lip}(\Rb)$ acting by multiplication on $H$,
and $D = id/dx$.) So this requirement can be dropped in the following
discussion.

\begin{defi}\label{tripdef}
Let $(\Ac, H, D)$ be a spectral triple, possibly minus the assumption
that $D$ has compact resolvent. The {\it quantum pseudometric
associated to $(\Ac, H, D)$} is the smallest quantum pseudometric
$\Vb$ on $\Ac''$ such that $e^{itD} \in \Vc_t$ for all $t > 0$.
\end{defi}

In other words, $\Vb$ is the smallest W*-filtration such that
$\Ac' \subseteq \Vc_0$ and $e^{itD} \in \Vc_t$ for all $t > 0$.
This definition makes sense because arbitrary meets of quantum
pseudometrics exist (Definition \ref{easycon} (c)).

We can also characterize the quantum pseudometric associated to a spectral
triple internally:

\begin{lemma}\label{sptrint}
Let $\Vb = \{\Vc_t\}$ be the quantum pseudometric associated to a spectral
triple $(\Ac, H, D)$. For each $t > 0$ let $\Wc_t$ be the
weak* closure of the span of the operators of the form
$$e^{is_1D}A_1e^{is_2D}A_2\cdots e^{is_nD}A_n$$
with $n \in \Nb$, each $A_i \in \Ac'$, each $s_i \in \Rb$,
and $\sum |s_i| \leq t$.
Then $\Vc_t = \bigcap_{s > t} \Wc_s$ for all $t \geq 0$.
\end{lemma}

\begin{proof}
Since $e^{\pm itD}A$ must belong to $\Vc_t$ for all $t \geq 0$
and all $A \in \Ac'$, it follows from the filtration property
that $\Wc_t \subseteq \Vc_t$ for all $t$, and then the fact
that $\Vc_t = \bigcap_{s > t} \Vc_s$ shows that $\bigcap_{s > t}\Wc_s
\subseteq \Vc_t$ for all $t$. It is evident that $\Ac' \subseteq
\bigcap_{s > t} \Wc_s$ and $e^{itD} \in \bigcap_{s > t} \Wc_s$ for
all $t \geq 0$. To complete the proof we must check that
$\Vc_t = \bigcap_{s > t} \Wc_s$ defines a quantum pseudometric.
Condition (i) of Definition \ref{filt} holds because $\Wc_s\Wc_t
\subseteq \Wc_{s+t}$ for all $s$ and $t$, and condition (ii) is easy.
\end{proof}

Next we note that quantum pseudometrics associated to spectral
triples are always path quantum pseudometrics (Definition
\ref{various} (e)).

\begin{prop}
The quantum pseudometric associated to any spectral triple is a
path quantum pseudometric.
\end{prop}

\begin{proof}
Let $\Vb$ be the quantum pseudometric associated to the spectral
triple $(\Ac, H, D)$. Fix $s,t \geq 0$. For any $\epsilon > 0$ we have
$\Vc_{s+\epsilon}\Vc_{t+\epsilon} \subseteq \Vc_{s + t + 2\epsilon}$,
and it follows that
$$\bigcap_{\epsilon > 0}
\overline{\Vc_{s+\epsilon}\Vc_{t+\epsilon}}^{wk^*} \subseteq
\Vc_{s + t}.$$
(This will be true of any quantum pseudometric.)
For the reverse inclusion, by Lemma \ref{sptrint}
it will be enough to show that any
operator of the form $e^{is_1D}A_1e^{is_2D}A_2\cdots e^{is_nD}A_n$
with $n \in \Nb$, each $A_i \in \Ac'$, each $s_i \in \Rb$,
and $\sum |s_i| \leq s + t + \epsilon$, belongs to $\Vc_s\Vc_{t+\epsilon}$.
This is true because we can write
$$e^{is_1D}A_1e^{is_2D}A_2\cdots e^{is_nD}A_n =
(e^{is_1D}A_1 \cdots e^{is_{j-1}D}A_{j-1}e^{is_j'D})
(e^{is_j''D}A_j\cdots e^{is_nD}A_n)$$
with $s_j' + s_j'' = s_j$, $|s_1| + \cdots + |s_j'| \leq s$, and
$|s_j''| + \cdots + |s_n| \leq t + \epsilon$.
\end{proof}

We will now show that the construction in
Definition \ref{tripdef} recovers the standard quantum metric for
spectral triples involving the Hodge-Dirac operator (see \cite{Cno})
on a complete
connected Riemannian manifold with positive injectivity radius
(i.e., the infimum of the injectivity radius over all points in the
manifold is nonzero; see \cite{Kli}).

Let $M$ be a complete connected Riemannian manifold equipped with
volume measure. It carries an intrinsic metric $d$ and a corresponding
intrinsic measurable metric defined by
\begin{eqnarray*}
\rho(\chi_S,\chi_T)
&=& \inf\{d(x,y): x\hbox{ is a Lebesgue point of }S\cr
&&\phantom{\liminf\liminf}\hbox{and $y$ is a Lebesgue point of }T\}.
\end{eqnarray*}
It is straightforward to verify that $\rho$ is indeed a measurable
metric. (To verify the join condition, Definition \ref{measpm} (iii),
use the fact that if $(X,\mu)$ is $\sigma$-finite then any join
$\bigvee p_\lambda$ of projections in $L^\infty(X,\mu)$ equals the
join of some countable subcollection.)

In Theorem \ref{qmeasmet} we identified a quantum metric $\Vb_\rho$
corresponding to $\rho$ which lives on the Hilbert space $L^2(M,\mu)$.
Here we need to work on $H =$ the complexification of the space of $L^2$
differential forms on $M$, with $L^\infty(M,\mu)$ acting by multiplication;
the corresponding quantum metric $\tilde{\Vb}_\rho = \{\tilde{\Vc}^\rho_t\}$
for this representation is defined by
$$\tilde{\Vc}^\rho_t = \{A \in \Bc(H): \rho(p,q) > t\quad
\Rightarrow\quad M_p AM_q = 0\}$$
with $p$ and $q$ ranging over projections in $L^\infty(M,\mu)$.

\begin{theo}
Let $M$ be a complete connected Riemannian manifold with positive
injectivity radius, let $\mu$ be volume measure on $M$, let $H$ be the
complexification of the space of $L^2$ differential forms,
and let $\Ac = C^\infty_c(M)$, acting on $H$ by multiplication. Let
$D = d + d^*$ be the Hodge-Dirac operator. Then the
quantum pseudometric $\Vb$ associated to $(\Ac, H, D)$ (Definition
\ref{tripdef}) equals the quantum metric $\tilde{\Vb}_\rho$
associated to the intrinsic measurable metric $\rho$.
\end{theo}

\begin{proof}
Note that $D$ has compact resolvent if $M$ is compact, but not in general.

The inequality $\Vb \leq \tilde{\Vb}_\rho$ is easy. Let $S$ and $T$ be positive
measure subsets of $M$ and suppose $\rho(\chi_S,\chi_T) > t$; we must show that
$M_{\chi_S}AM_{\chi_T} = 0$ for any operator $A \in \Vc_t$. By Lemma
\ref{sptrint} it is enough to show this for operators $A$ of the form
$e^{is_1D}A_1e^{is_2D}A_2\cdots e^{is_nD}A_n$ with each $A_i \in \Ac'$ and
$\sum |s_i| \leq t$. We may assume that $S$ and $T$ are open; otherwise,
replace them with the open $\epsilon$-neighborhoods of their Lebesgue sets
for some $\epsilon < (\rho(\chi_S,\chi_T) - t)/2$. Since the $A_i$ belong to
$\Ac'$, it will suffice to show that if $f \in H$ is supported
in $S$ then $e^{isD}f$ is supported in the $s$-neighborhood
of $S$. By continuity we may assume $f$ is smooth.
The desired conclusion now follows from (\cite{Roe}, Proposition 5.5).
Thus, we have shown that $\Vb \leq \tilde{\Vb}_\rho$ (without assuming positive
injectivity radius).

For the reverse inclusion, we need to assume that $M$ has positive
injectivity radius $\delta$. We first show that $\tilde{\Vb}_\rho$ is a path
quantum metric. Let $s,t \geq 0$ and let $A \in \tilde{\Vc}^\rho_{s+t}$. Fix
$\epsilon > 0$ and let $\delta' = \min\{\delta, \epsilon\}$; we want
to show that $A$ is in the weak* closure of
$\tilde{\Vc}^\rho_{t + \epsilon}\tilde{\Vc}^\rho_{s + \epsilon}$.
Let $\{x_n\}$ be a $\delta'$-net in $M$ and
let $\{S_n\}$ be a measurable partition of $M$ with $S_n \subseteq
{\rm ball}(x_n,\delta')$; then $A = \sum_{m,n} M_{\chi_{S_m}}AM_{\chi_{S_n}}$,
and each summand belongs to $\tilde{\Vc}^\rho_{s + t}$, so we may restrict
attention to a single summand and assume $A = M_{\chi_{S_m}}AM_{\chi_{S_n}}$.
Now if $A \neq 0$ then $d(x_m,x_n) < s + t + 2\epsilon$, so
let $\gamma: [0,r] \to M$ be a unit speed geodesic from $x_n$ to $x_m$ with
$r < s + t + 2\epsilon$, let $y = \gamma(s + \epsilon)$, and
let $V \in \Bc(H)$ be the partial isometry from $M_pH$ to $M_qH$, where
$p = \chi_{{\rm ball}(x_m,\delta')}$ and $q = \chi_{{\rm ball}(y,\delta')}$,
induced via weighted composition from the homeomorphism
of ${\rm ball}(x_m,\delta')$ with ${\rm ball}(y,\delta')$
arising from the exponential maps. Then $d(x_m,y) < t + \epsilon$ and
$$A = (V^*)(VA) \in
\tilde{\Vc}^\rho_{t + 3\epsilon}\tilde{\Vc}^\rho_{s + 3\epsilon}.$$
This suffices to establish that $\tilde{\Vb}_\rho$ is a path quantum metric.

To verify $\tilde{\Vb}_\rho \leq \Vb$ it will now be enough to show that
$\tilde{\Vc}^\rho_t \subseteq \Vc_t$ for $t < \delta$. Fix $t < \delta$ and
let $x, y \in M$ satisfy $d(x,y) < t$; we will show that there exist
neighborhoods $S$ and $T$ of $x$ and $y$ respectively such that any
operator $A \in \Bc(H)$ satisfying $A = M_{\chi_S}AM_{\chi_T}$
belongs to $\Vc_t$. By an argument similar to the one given in the last
paragraph, this will suffice.

Observe that $\Ac' = L^\infty(M,\mu)'$ can be identified with
the bounded measurable sections of ${\rm End}(\Lambda^*T^*M)$.
Identify $L^2(M,\mu)$ with the $L^2$ 0-forms as a subspace of $H$.
It will be enough to find neighborhoods $S$ and $T$ of $x$ and $y$
respectively such that every operator $A \in \Bc(L^2(M,\mu))
\subseteq \Bc(H)$ satisfying $A = M_{\chi_S}AM_{\chi_T}$ belongs to
$\Vc_t$; the corresponding statement will then also be true for
arbitrary operators in $\Bc(H)$ because $\Vc_t$ is a bimodule over $\Ac'$.

The Hodge Laplacian is $\Delta = dd^* + d^*d$. $L^2(M,\mu)$
is an invariant subspace for $\Delta$ and $\Delta_0 =
\Delta|_{L^2(M,\mu)}$ is the scalar Laplacian. Also note that
since the power series expansion of $\cos z$ involves only even
powers of $z$, we have $\cos(sD) = \cos(s\sqrt{\Delta})$. Thus
for any $\phi \in L^1(\Rb)$ the operator
$$\int_0^t \phi(s)\cos(s\sqrt{\Delta})\, ds
= \frac{1}{2}\int_{-t}^t \phi(|s|)e^{isD}\, ds$$
belongs to $\Vc_t$, as does $\int_0^t \phi(s)\cos(s\sqrt{\Delta_0})\, ds
= P\int_0^t \phi(s)\cos(s\sqrt{\Delta})\, ds$ where $P \in \Ac'
\subseteq \Bc(H)$ is the projection onto $L^2(M,\mu)$.

In the remainder of the proof we work with scalar functions
on $M$. For $0 \leq s \leq t$ let $u_x(s)$ be the
distribution $u_x(s) = \cos(s\sqrt{\Delta_0})\delta_x$ and
note that it satisfies the wave equation with initial conditions
$u_x(0) = \delta_x$ and $u_x'(0) = 0$. By examining the wavefront
set (see Theorem 4 of \cite{Vas}) we see that $y$ is in the support
of $u_x(s)$ for
$s = d(x,y)$. Then let $\phi \in C^\infty(\Rb)$ satisfy $\phi = 1$
in an interval about $s = d(x,y)$ and $\phi = 0$ outside a slightly
larger interval, and let $B = \int_0^t \phi(s)\cos(s\sqrt{\Delta_0})\, ds
\in \Vc_t$. Then $B$ is an integral operator with continuous
kernel $k(x',y')$, and if $\phi$ is suitably chosen there will exist
neighborhoods $S$ and $T$ of $x$ and $y$ such that $k(x',y') \neq 0$
for all $x' \in S$ and $y' \in T$.
It follows that $M_fBM_g$ belongs to $\Vc_t$ for all $f \in
C_c(S)$ and $g \in C_c(T)$; this is the integral operator with kernel
$k(x',y')f(x')g(y')$. By the Stone-Weierstrass theorem we conclude that
$\Vc_t$ contains all integral operators with continuous kernel supported
in $S \times T$, and hence, by weak* closure, all operators $A \in
\Bc(L^2(M,\mu))$ satisfying $A = M_{\chi_S}AM_{\chi_T}$, as desired.
\end{proof}

In the above case the quantum pseudometric associated to the
spectral triple $(\Ac,H,D)$ is in fact a quantum metric, not just a
quantum pseudometric, on $\Ac''$. A good problem for future work
will be to identify natural conditions that ensure the quantum pseudometric
associated to a spectral triple is a quantum metric.

\section{Lipschitz operators}\label{lipschitz}

The algebra ${\rm Lip}(X)$ of bounded scalar-valued
Lipschitz functions on a metric
space $X$ has been studied extensively (see \cite{W4}). Under
appropriate hypotheses one can recover the space $X$ as the normal
spectrum of ${\rm Lip}(X)$, with metric inherited from the dual of
${\rm Lip}(X)$; moreover, basic metric properties of $X$ are reflected
in algebraic properties of ${\rm Lip}(X)$ (closed subsets correspond
to weak* closed ideals, etc.). Thus the relation between metric spaces
and Lipschitz algebras is strongly analogous to the relation between
topological and measure spaces and the algebras $C(X)$ and $L^\infty(X,\mu)$.

We find that there are two distinct, but related, versions of the
notion of a ``Lipschitz operator'' associated to a quantum pseudometric.
There is a spectral version with good lattice properties but poor
algebraic properties and a commutation version with the opposite
qualities. (Something similar happens with lower semicontinuity, which
bifurcates in the C*-algebra setting into two notions, ``lsc'' and
``q-lsc'', with, respectively, good algebraic and lattice properties
\cite{Ake}.)

We can abstractly characterize the spectral Lipschitz gauge
(Theorem \ref{abstractlip}), which
provides another intrinsic characterization of quantum pseudometrics,
perhaps more elegant than the one given in Theorem \ref{abch}.
Our other main result, Theorem \ref{lipineq}, states that commutation
Lipschitz number is always less than or equal to spectral Lipschitz number.
We prove this using a powerful result of Browder and Sinclair
(\cite{BD2}, Corollary 26.6) which equates
the norm and spectral radius of a Hermitian element of a complex unital
Banach algebra. Our inequality
is valuable because we have general techniques for constructing
operators with finite spectral Lipschitz number (see Lemma \ref{distfn}).

\subsection{The abelian case}\label{abcase}
We start by reviewing the measurable version of Lipschitz number.
Say that the {\it essential range} of a function
$f \in L^\infty(X,\mu)$ is the set of all $a \in \Cb$ such that
$f^{-1}(U)$ has positive measure for every open neighborhood $U$
of $a$. Equivalently, it is the spectrum of the operator
$M_f \in \Bc(L^2(X,\mu))$. If $p \in L^\infty(X,\mu)$ is a projection
then we denote the essential range of $f|_{{\rm supp}(p)}$ by
${\rm ran}_p(f)$.

\begin{defi}\label{measlipnum}
(\cite{W4}, Definition 6.2.1)
Let $(X,\mu)$ be a finitely decomposable measure space and let
$\rho$ be a measurable pseudometric on $X$. The {\it Lipschitz
number} of $f \in L^\infty(X,\mu)$ is the quantity
$$L(f) = \sup\left\{\frac{d({\rm ran}_p(f), {\rm ran}_q(f))}{\rho(p,q)}\right\},$$
where the supremum is taken over all nonzero projections $p,q \in
L^\infty(X,\mu)$ and we use the convention $\frac{0}{0} = 0$. (Note that
line 4 of Definition 6.2.1 of \cite{W4} should say ``essential infimum of
the function $|f(p) - f(q)|$''.) Here $d$
is the minimum distance between compact subsets of $\Cb$.
We call $L$ the {\it Lipschitz gauge} associated to $\rho$ and we
define ${\rm Lip}(X,\mu) = \{f \in L^\infty(X,\mu): L(f) < \infty\}$.
\end{defi}

We first observe that this definition generalizes the atomic abelian case.

\begin{prop}\label{atomicln}
Let $\mu$ be counting measure on a set $X$, let $d$ be a pseudometric
on $X$, and let $f \in l^\infty(X)$. Let $\rho$ be the associated
measurable pseudometric (Proposition \ref{atomicmm}). Then
$$L(f) =
\sup\left\{\frac{|f(x) - f(y)|}{d(x,y)}: x,y \in X, d(x,y) > 0\right\}.$$
\end{prop}

\begin{proof}
The inequality $\geq$ follows immediately from the definition of $L(f)$
in Definition \ref{measlipnum} by taking $p = \chi_{\{x\}}$,
$q = \chi_{\{y\}}$. For the reverse inequality let $p,q \in l^\infty(X)$,
say $p = \chi_S$ and $q = \chi_T$, with $S$ and $T$ both nonempty.
Find sequences $\{x_n\} \subseteq S$ and $\{y_n\} \subseteq T$ such
that $d(x_n,y_n) \to \rho(p,q)$; then either $\liminf |f(x_n) - f(y_n)| = 0$,
in which case $d({\rm ran}_p(f), {\rm ran}_q(f)) = 0$, or else
$$\frac{d({\rm ran}_p(f), {\rm ran}_q(f))}{\rho(p,q)}
\leq \liminf\frac{|f(x_n) - f(y_n)|}{\rho(p,q)}
\leq \sup\frac{|f(x_n) - f(y_n)|}{d(x_n,y_n)}.$$
This proves the inequality $\leq$.
\end{proof}

Next, in anticipation of our axiomatization of quantum Lipschitz gauges
in Section \ref{slnsec}, we give an alternative axiomatization of measurable
pseudometrics in terms of Lipschitz numbers.
Many versions of this result have appeared in the literature;
probably the closest is Example 6.2.5 of \cite{W4}.

\begin{defi}\label{ablipdef}
Let $(X,\mu)$ be a finitely decomposable measure space. An {\it abstract
Lipschitz gauge} on $L^\infty(X,\mu)$ is a function $\Lc$ from the real
part of $L^\infty(X,\mu)$ to $[0,\infty]$ satisfying
\begin{quote}
\noindent (i) $\Lc(1) = 0$

\noindent (ii) $\Lc(af) = |a|\Lc(f)$

\noindent (iii) $\Lc(f + g) \leq \Lc(f) + \Lc(g)$

\noindent (iv) $\Lc(\bigvee f_\lambda) \leq \sup \Lc(f_\lambda)$
\end{quote}
for any $a \in \Rb$ and $f,g, f_\lambda \in L^\infty(X,\mu)$ real-valued
with $\sup \|f_\lambda\|_\infty < \infty$.
\end{defi}

\begin{theo}\label{abablip}
Let $(X,\mu)$ be a finitely decomposable measure space. If $\rho$ is a
measurable pseudometric on $X$ then the restriction of the associated
Lipschitz gauge $L$ to the real part of $L^\infty(X,\mu)$ is an abstract
Lipschitz gauge. If $\Lc$ is an abstract Lipschitz gauge then
$$\rho_\Lc(p,q) =
\sup\{d({\rm ran}_p(f), {\rm ran}_q(f)): \Lc(f) \leq 1\}$$
is a measurable pseudometric on $X$. The two constructions are
inverse to each other.
\end{theo}

\begin{proof}
Let $\rho$ be a measurable pseudometric on $X$ and let $L$ be the
associated Lipschitz gauge. The fact that $L$ is an abstract Lipschitz gauge
follows from (\cite{W6}, Corollary 1.21). We verify that $\rho = \rho_L$.
The inequality $\rho(p,q) \geq
\rho_L(p,q)$ for all $p, q$ follows immediately from the definition of $L$.
Conversely, if $\rho(p,q) < \infty$ then let $c = \rho(p,q)$ and define
$$f = \bigvee \min\{\rho(r,q), c\}\cdot r,$$
taking the join in $L^\infty(X,\mu)$ over all nonzero projections $r$.
We have ${\rm ran}_q(f) = \{0\}$ and ${\rm ran}_p(f) = \{c\}$, so that
$d({\rm ran}_p(f), {\rm ran}_q(f)) = c = \rho(p,q)$. Also, $L(f) \leq 1$
by (\cite{W6}, Lemma 1.22). This proves the inequality $\rho(p,q) \leq
\rho_L(p,q)$. If $\rho(p,q) = \infty$ then take $c \to \infty$ in the
preceding argument.

Now let $\Lc$ be an abstract Lipschitz gauge.
We will show that $\rho_\Lc$ is a measurable pseudometric and that
$\Lc(f) = L(f)$ for all real-valued $f \in L^\infty(X,\mu)$,
where $L$ is the Lipschitz gauge associated to $\rho_\Lc$.

We verify the conditions in Definition \ref{measpm}. Conditions
(i) and (ii) are trivial. The inequality $\leq$ in condition (iii) is
easy; for the reverse inequality suppose $\rho_\Lc(p_\lambda,q_\kappa) > a$
for all $\lambda, \kappa$ and find $f_{\lambda\kappa} \in L^\infty(X,\mu)$
real-valued such that $\Lc(f_{\lambda\kappa}) \leq 1$ and
$d({\rm ran}_{p_\lambda}(f_{\lambda\kappa}),
{\rm ran}_{q_\kappa}(f_{\lambda\kappa})) \geq a$
for all $\lambda,\kappa$. Then define
$$g_{\lambda\kappa}= \bigwedge_{c \in {\rm ran}_{p_\lambda}(f_{\lambda\kappa})}
\left(|f_{\lambda\kappa} - c\cdot 1| \wedge a\cdot 1\right);$$
we still have $\Lc(g_{\lambda\kappa}) \leq 1$, and
${\rm ran}_{p_\lambda}(g_{\lambda\kappa}) = \{0\}$
and ${\rm ran}_{q_\kappa}(g_{\lambda\kappa}) = \{a\}$. So
$$g = \bigwedge_\lambda \bigvee_\kappa g_{\lambda\kappa}$$
satisfies $\Lc(g) \leq 1$,
${\rm ran}_{\bigvee p_\lambda}(g_{\lambda\kappa}) = \{0\}$,
and ${\rm ran}_{\bigvee q_\kappa}(g_{\lambda\kappa}) = \{a\}$. Thus
$\rho_\Lc(\bigvee p_\lambda, \bigvee q_\kappa) \geq a$, which is enough.

To verify condition (iv), let $p,q,r \in L^\infty(X,\mu)$ be
nonzero projections and let $f \in L^\infty(X,\mu)$ be real-valued
and satisfy $\Lc(f) \leq 1$. Given $\epsilon > 0$, choose
$a \in {\rm ran}_q(f)$ and find $q' \leq q$ such that
$${\rm ran}_{q'}(f) \subseteq [a - \epsilon , a + \epsilon].$$
Then
\begin{eqnarray*}
d({\rm ran}_p(f), {\rm ran}_r(f))
&\leq& d({\rm ran}_p(f), {\rm ran}_{q'}(f)) +
d({\rm ran}_{q'}(f), {\rm ran}_r(f)) + 2\epsilon,\cr
&\leq& \rho_\Lc(p,q') + \rho_\Lc(q',r) + 2\epsilon
\end{eqnarray*}
and taking $\epsilon \to 0$ and the supremum over $f$ yields the
measurable triangle inequality. So $\rho_\Lc$ is a measurable pseudometric.

Now let $f \in L^\infty(X,\mu)$ be real-valued. First suppose $\Lc(f) <
\infty$ and let $L(f)$ be the Lipschitz number of $f$ with respect to
the pseudometric $\rho_\Lc$.
For any $0 < a < 1/\Lc(f)$, we then have $\Lc(af) < 1$ and so
$$\rho_\Lc(p,q) \geq d({\rm ran}_p(af), {\rm ran}_q(af))
= a\cdot d({\rm ran}_p(f), {\rm ran}_q(f))$$
for any $p$ and $q$. Taking $a \to 1/\Lc(f)$, this shows that
$$\frac{d({\rm ran}_p(f), {\rm ran}_q(f))}{\rho_\Lc(p,q)} \leq \Lc(f),$$
and taking the supremum over $p$ and $q$ yields $L(f) \leq \Lc(f)$. For
the reverse inequality, suppose $L(f) < 1$; we will show that $\Lc(f) \leq 1$.
For any $a, b \in \Rb$ with $a < b$ let $p = \chi_{f^{-1}(-\infty,a]}$ and
$q = \chi_{f^{-1}[b,\infty)}$ and find $f_{ab}$ such that
$\Lc(f_{ab}) \leq 1$ and
$$d({\rm ran}_p(f_{ab}), {\rm ran}_q(f_{ab})) =
d({\rm ran}_p(f), {\rm ran}_q(f))\geq b - a.$$
We can do this because $d({\rm ran}_p(f), {\rm ran}_q(f)) \leq L(f)\rho_\Lc(p,q)
< \rho_\Lc(p,q)$. Now define
$$g_{ab} = \bigwedge_{c \in {\rm ran}_p(f_{ab})} \big(|f_{ab} - c\cdot 1|
\wedge (b-a)\cdot 1\big);$$
then $\Lc(g_{ab}) \leq 1$, ${\rm ran}_p(g_{ab}) = \{0\}$, and
${\rm ran}_q(g_{ab}) = \{b-a\}$. So
$$f = \bigwedge_{|a| \leq \|f\|_\infty} \bigvee_{b > a} (g_{ab} + a\cdot 1),$$
and this shows that $\Lc(f) \leq 1$. We conclude that $\Lc(f) \leq L(f)$.
This completes the proof.
\end{proof}

Lemma 6.2.4 of \cite{W4} can be used to prove a similar result with
$\Lc$ defined on all of $L^\infty(X,\mu)$. It is interesting to note
that although Lipschitz numbers do satisfy the seminorm condition
in Definition \ref{ablipdef} (iii), it is only used in the proof
to ensure that $\Lc(f + c\cdot 1) = \Lc(f)$. In the noncommutative
setting only this weaker version holds (see Definition \ref{abspeclip}
and Example \ref{countersum}).

\subsection{Spectral Lipschitz numbers}\label{slnsec}
Now we introduce the spectral Lipschitz number of a self-adjoint
operator in a von Neumann algebra $\Mx$ equipped with a quantum
pseudometric, or more generally
a self-adjoint operator in $\Mx \overline{\otimes} \Bc(l^2)$.
Recall that $P_S(A)$ denotes the spectral projection of the self-adjoint
operator $A$ for the Borel set $S \subseteq \Rb$.

\begin{defi}\label{speclip}
Let $\rho$ be a quantum distance function on a von Neumann algebra
$\Mx$ (Definition \ref{qdfdef}) and let
$A \in \Mx \overline{\otimes} \Bc(l^2)$ be self-adjoint.
The {\it spectral Lipschitz number} of $A$ is the quantity
$$L_s(A) = \sup \left\{\frac{b - a}{\rho(P_{(-\infty, a]}(A),
P_{[b, \infty)}(A))}: a,b \in \Rb, a < b\right\},$$
with the convention $\frac{0}{0} = 0$,
and $A$ is {\it spectrally Lipschitz} if $L_s(A) < \infty$.
We call the function $L_s$ the {\it spectral Lipschitz gauge}.
\end{defi}

We immediately note an alternative formula for $L_s(A)$.

\begin{prop}\label{altform}
Let $\rho$ be a quantum distance function on a von Neumann algebra
$\Mx$ and let $A \in \Mx \overline{\otimes} \Bc(l^2)$ be self-adjoint.
Then
$$L_s(A) = \sup \left\{\frac{d(S,T)}{\rho(P_S(A), P_T(A))}:
S,T \subseteq \Rb\hbox{ Borel}\right\},$$
with the convention $\frac{0}{0} = 0$.
\end{prop}

\begin{proof}
Here $d(S,T) = \inf\{d(x,y): x \in S, y \in T\}$. The inequality $\leq$
follows by taking $S = (-\infty, a]$ and $T = [b,\infty)$ for arbitrary
$a,b \in \Rb$, $a < b$. For the reverse inequality, let $S, T \subset \Rb$
be Borel and (assuming $d(S,T) > 0$, otherwise the pair makes no
contribution to the supremum) partition them as $S = \bigcup S_i$,
$T = \bigcup T_j$ such that each $S_i$ and each $T_j$ has diameter at
most $d(S,T)$. We can do this so that $P_{S_i}(A) = P_{T_j}(A) = 0$ for
all but finitely many $i$ and $j$, since the interval $[-\|A\|, \|A\|]$
has finite length. Then $P_S(A) = \bigvee P_{S_i}(A)$ and $P_T(A) =
\bigvee P_{T_j}(A)$ implies that we must have $\rho(P_{S_i}(A), P_{T_j}(A)) =
\rho(P_S(A), P_T(A))$ for some $i$ and $j$. Since $d(S_i,T_j) \geq
d(S,T) \geq \max\{{\rm diam}(S_i), {\rm diam}(T_j)\}$, without loss
of generality we may assume that $a < b$ where $a = \sup S_i$ and
$b = \inf T_j$ for this choice of $i$ and $j$. We then have
$$\frac{d(S,T)}{\rho(P_S(A), P_T(A))} \leq
\frac{d(S_i,T_j)}{\rho(P_{S_i}(A), P_{T_j}(A))} \leq
\frac{b-a}{\rho(P_{(-\infty, a]}(A), P_{[b,\infty)}(A))} \leq L_s(A).$$
Taking the supremum over $S$ and $T$ yields the desired inequality.
\end{proof}

Spectral Lipschitz numbers generalize measurable Lipschitz numbers
(Definition \ref{measlipnum}).

\begin{prop}\label{seqab}
Let $(X,\mu)$ be a finitely decomposable measure space, let
$\rho$ be a measurable pseudometric on $X$, and let $\Vb_\rho$
be the associated quantum pseudometric on the von Neumann algebra
$\Mx \cong L^\infty(X,\mu)$ of bounded multiplication operators on
$L^2(X,\mu)$ (Theorem \ref{qmeasmet}). Then for any real-valued
$f \in L^\infty(X,\mu)$ we have $L_s(M_f) = L(f)$.
\end{prop}

\begin{proof}
Let $\tilde{\rho}$ be the quantum distance function associated to
$\Vb_\rho$ (Definition \ref{projdist}) and recall that we have
$\tilde{\rho}(M_p,M_q) = \rho(p,q)$ for all nonzero projections
$p,q \in L^\infty(X,\mu)$ (Theorem \ref{qmeasmet}). Now
the inequality $L_s(M_f) \leq L(f)$ is proven by taking $p =
\chi_{f^{-1}((-\infty,a])}$ and $q = \chi_{f^{-1}([b,\infty)}$
for arbitrary $a,b \in \Rb$, $a < b$
(so that $P_{(-\infty, a]}(M_f) = M_p$ and $P_{[b, \infty)}(M_f) = M_q$)
in Definition \ref{measlipnum},
since $b - a \leq d({\rm ran}_p(f), {\rm ran}_q(f))$ and therefore
$$\frac{b-a}{\tilde{\rho}(P_{(-\infty, a]}(M_f), P_{[b, \infty)}(M_f))}
= \frac{b-a}{\rho(p,q)}
\leq \frac{d({\rm ran}_p(f),{\rm ran}_q(f))}{\rho(p,q)} \leq L(f).$$
Taking the supremum over $a$ and $b$ shows that $L_s(M_f) \leq L(f)$.
For the reverse inequality, let $p,q \in L^\infty(X,\mu)$ be nonzero
projections. If $\epsilon = d({\rm ran}_p(f),{\rm ran}_q(f)) = 0$ then
the pair makes no contribution to $L(f)$, so assume $\epsilon > 0$.
Now partition $p$ and $q$ as $p = \sum_1^m p_i$ and $q = \sum_1^n q_j$
such that ${\rm ran}_{p_i}(f)$ and ${\rm ran}_{q_j}(f)$ have diameter
at most $\epsilon$ for all $i$ and $j$. Then for some choice of $i$ and $j$
we have $\rho(p_i,q_j) = \rho(p,q)$, and as in the proof of
Proposition \ref{altform} we may assume $a < b$
where $a = \sup {\rm ran}_{p_i}(f)$ and $b = \inf {\rm ran}_{q_j}(f)$.
Thus $M_{p_i} \leq P_{(-\infty,a]}(M_f)$ and
$M_{q_j} \leq P_{[b,\infty)}(M_f)$, so that
\begin{eqnarray*}
\frac{d({\rm ran}_p(f), {\rm ran}_q(f))}{\rho(p,q)} &\leq&
\frac{d({\rm ran}_{p_i}(f), {\rm ran}_{q_j}(f))}{\rho(p_i,q_j)}\cr
&\leq& \frac{b-a}{\tilde{\rho}(P_{(-\infty, a]}(M_f), P_{[b, \infty)}(M_f))}\cr
&\leq& L_s(M_f).
\end{eqnarray*}
Taking the supremum over $p$ and $q$ shows that $L(f) \leq L_s(M_f)$.
\end{proof}

Together with Proposition \ref{atomicln} this implies that the spectral
Lipschitz number reduces to the classical Lipschitz number in the
atomic abelian case.

\begin{coro}\label{aasplp}
Let $d$ be a pseudometric on a set $X$ and equip the von Neumann algebra
$\Mx \cong l^\infty(X)$ of bounded multiplication operators on $l^2(X)$
with the associated quantum pseudometric (Proposition \ref{aa}).
Then for any real-valued $f \in l^\infty(X)$ we have $L_s(M_f) = L(f)$.
\end{coro}

The fact that spectral Lipschitz numbers effectively reduce to classical
Lipschitz numbers in the abelian case is evidence that Definition
\ref{speclip} is reasonable. We can also point to the following easy result.

\begin{prop}\label{compo}
Let $\rho$ be a quantum distance function on a von Neumann algebra $\Mx$,
let $A \in \Mx \overline{\otimes} \Bc(l^2)$ be self-adjoint and spectrally
Lipschitz, and let $f: \Rb \to \Rb$ be Lipschitz. Then $f(A)$ is spectrally
Lipschitz and $L_s(f(A)) \leq L(f)L_s(A)$.
\end{prop}

\begin{proof}
Let $a,b \in \Rb$, $a < b$, and let
$S = f^{-1}((-\infty, a])$ and $T = f^{-1}([b,\infty))$. Then
$$P_{(-\infty, a]}(f(A)) = P_S(A)\hbox{ and }P_{[b,\infty)}(f(A)) = P_T(A).$$
Also $b - a \leq L(f)\cdot d(S,T)$, so
$$\frac{b-a}{\rho(P_{(-\infty, a]}(f(A)), P_{[b,\infty)}(f(A)))}
\leq L(f)\cdot\frac{d(S,T)}{\rho(P_S(A), P_T(A))} \leq L(f)\cdot L_s(A)$$
by Proposition \ref{altform}. Taking the supremum over $a$ and $b$
proves that $f(A)$ is spectrally Lipschitz and yields the stated inequality.
\end{proof}

We will give other basic properties of $L_s$, in particular its
compatibility with spectral joins and meets, in Lemma \ref{joinmeet}
and Theorem \ref{abstractlip}.

The spectral Lipschitz condition can be related to the notion of co-Lipschitz
number introduced in Definition \ref{lipmaps} (and consequently Corollary
\ref{aasplp} can also be deduced from Proposition \ref{aacolip}).
Recall that if $A \in \Mx$ is self-adjoint then the von Neumann algebra
$W^*(A)$ it generates is $*$-isomorphic to $L^\infty(X,\mu)$ where $X$ is
the spectrum of $A$ and $\mu$ is some finite measure on $X$
(\cite{Tak}, Theorem III.1.22).
Let $\phi: L^\infty(X,\mu) \cong W^*(A) \subseteq \Mx$ be such an isomorphism
and define a measurable metric $\tilde{\rho}$ on $X$ by setting
$$\tilde{\rho}(p,q) = d({\rm ran}_p(\zeta), {\rm ran}_q(\zeta))$$
where $\zeta = \phi^{-1}(A)$.

\begin{prop}\label{seqco}
Let $\rho$ be a quantum distance function on a von Neumann algebra
$\Mx$ and let $A \in \Mx$ be self-adjoint. Let $\phi: L^\infty(X,\mu)
\cong W^*(A)$ be a *-isomorphism and equip $X$ with the quantum metric
$\Vb_{\tilde{\rho}}$ (Theorem \ref{qmeasmet}) associated to the measurable
metric $\tilde{\rho}$ defined above. Then the spectral Lipschitz
number $L_s(A)$ of $A$ equals the co-Lipschitz number $L(\phi)$ of $\phi$.
\end{prop}

\begin{proof}
One inequality is easy: given $a, b \in \Rb$ take
$p = \phi^{-1}(P_{(-\infty, a]}(A))$ and $q = \phi^{-1}(P_{[b,\infty)}(A))$
in Definition \ref{lipmaps}. Then $\tilde{\rho}(p,q) \geq b - a$ and so
$$\frac{b-a}{\rho(P_{(-\infty, a]}(A), P_{[b,\infty)}(A))}
\leq \frac{\tilde{\rho}(p,q)}{\rho(\phi(p),\phi(q))} \leq L(\phi).$$
Taking the supremum over $a$ and $b$ yields $L_s(A) \leq L(\phi)$.
For the reverse inequality, since $\Vb_{\tilde{\rho}}$ is reflexive Proposition
\ref{refcolip} implies that we can restrict attention to projections
in $L^\infty(X,\mu)$ when evaluating $L(\phi)$. So let
$p,q \in L^\infty(X,\mu)$ be projections. As in the proof of Proposition
\ref{seqab}, we can find projections $p' \leq p$ and $q' \leq q$ such that
$\rho(\phi(p'), \phi(q')) = \rho(\phi(p), \phi(q))$ and ${\rm ran}_{p'}(\zeta)$
and ${\rm ran}_{q'}(\zeta)$ have diameter less than $\tilde{\rho}(p,q)$,
where $\zeta = \phi^{-1}(A)$. Without loss of generality $a < b$ where
$a = \sup {\rm ran}_{p'}(\zeta)$ and $b = \inf {\rm ran}_{q'}(\zeta)$. Then
$\phi(p') \leq P_{(-\infty,a]}(A)$ and $\phi(q') \leq P_{[b,\infty)}(A)$,
so that
$$\frac{\tilde{\rho}(p,q)}{\rho(\phi(p),\phi(q))} \leq
\frac{\tilde{\rho}(p',q')}{\rho(\phi(p'),\phi(q'))} \leq
\frac{b-a}{\rho(P_{(-\infty,a]}(A), P_{[b,\infty)}(A))} \leq L_s(A).$$
Taking the supremum over $p$ and $q$ yields $L(\phi) \leq L_s(A)$.
\end{proof}

By stabilization (see Section \ref{refstab}) we can therefore equate the
spectral Lipschitz number of any self-adjoint operator
$A \in \Mx \overline{\otimes} \Bc(l^2)$ with the co-Lipschitz number of a
$*$-isomorphism $\phi: L^\infty(X, \mu) \cong W^*(A) \subseteq
\Mx \overline{\otimes} \Bc(l^2)$. We can also prove a counterpart to
Proposition \ref{compo}.

\begin{coro}
Let $\Mx$ and $\Nc$ be von Neumann algebras equipped with quantum
distance functions, let $A \in \Mx \overline{\otimes} \Bc(l^2)$ be
self-adjoint and spectrally Lipschitz, and let $\psi: \Mx \to \Nc$
be a co-Lipschitz morphism. Then $(\psi \otimes I)(A)$ is spectrally
Lipschitz and $L_s((\psi \otimes I)(A)) \leq L(\psi)L_s(A)$.
\end{coro}

\begin{proof}
Let $\phi: L^\infty(X,\mu) \cong W^*(A)$ be a $*$-isomorphism as in
Proposition \ref{seqco}. Then $(\psi \otimes I)\circ\phi$ restricts to
an isomorphism from $L^\infty(S,\mu|_S)$ to $W^*((\psi\otimes I)(A))$
for some $S \subseteq X$, and we can take this to be the $*$-isomorphism
used in Proposition \ref{seqco} for the operator $(\psi\otimes I)(A)$. So
$$L_s((\phi \otimes I)(A)) = L(\psi\circ\phi|_S) \leq L(\psi)L(\phi|_S)
\leq L(\psi)L_s(A)$$
by Proposition \ref{colipcomp}.
\end{proof}

Next we want to show that the quantum distance function $\rho$ associated to
a quantum pseudometric can be recovered from the spectral Lipschitz gauge.
In order to prove this we need to show that there are sufficiently
many spectrally Lipschitz operators. The basic tool is the following
analog of (\cite{W6}, Lemma 1.22).

The {\it spectral join} of a bounded family of self-adjoint operators
$\{A_\lambda\}$ is the self-adjoint operator $\bigvee A_\lambda$ whose
spectral projections satisfy
$$P_{(a,\infty)}\left(\bigvee A_\lambda\right) =
\bigvee_\lambda P_{(a,\infty)}(A_\lambda)$$
for all $a \in \Rb$. Their {\it spectral meet} $\bigwedge A_\lambda$ has
spectral projections
$$P_{[a, \infty)}\left(\bigwedge A_\lambda\right) =
\bigwedge P_{[a,\infty)}(A_\lambda).$$
Equivalently, $\bigwedge A_\lambda = -\bigvee(-A_\lambda)$. 

\begin{lemma}\label{distfn}
Let $\rho$ be a quantum distance function on a von Neumann algebra $\Mx$,
let $R \in \Mx \overline{\otimes} \Bc(l^2)$ be a nonzero projection,
and let $c > 0$. Then
$$\bigvee \min\{\rho(P,R), c\}\cdot P,$$
taking the spectral join over all projections $P$ in $\Mx \overline{\otimes}
\Bc(l^2)$, has spectral Lipschitz number at most 1.
\end{lemma}

\begin{proof}
Let $A$ be this spectral join. Then $P_{(a,\infty)}(A)$ is the join
of the projections whose distance from $R$ is greater than $a$, for any
$a < c$. Now let $a,b \in \Rb$, $a < b$; we must show that
$b - a \leq \rho(P_{(-\infty, a]}(A), P_{[b, \infty)}(A))$. We
may assume that $b \leq c$ as otherwise $P_{[b,\infty)}(A) = 0$ and
so the right side is infinite. Let $\epsilon > 0$,
$P = P_{(b - \epsilon, \infty)}(A)$, and $Q = P_{(-\infty, a]}(A)$ and
observe that if $\tilde{Q}$ is any projection such that $Q\tilde{Q}
\neq 0$ then $\rho(\tilde{Q}, R) \leq a$ (as otherwise we would have
$\tilde{Q} \leq P_{(a, \infty)}(A)$ by the definition of $A$). Also
$\rho(P,R) \geq b - \epsilon$ since $P$ is a join of projections
whose distance from $R$ is greater than $b - \epsilon$. Thus
$$b - \epsilon \leq \rho(P, R) \leq \rho(P, Q) + \sup\{\rho(\tilde{Q}, R):
Q\tilde{Q} \neq 0\} \leq \rho(P,Q) + a$$
by Definition \ref{qdfdef} (v). Thus $b - a \leq
\rho(P_{(-\infty, a]}(A), P_{[b,\infty)}(A)) + \epsilon$, and taking
$\epsilon \to 0$ yields the desired inequality.
\end{proof}

This lets us recover the quantum distance function from the spectral Lipschitz
gauge.

\begin{theo}\label{distrecov}
Let $\rho$ be a quantum distance function on a von Neumann algebra
$\Mx$. Then
\begin{eqnarray*}
\rho(P,Q) &=& \sup\{a \geq 0: \hbox{some self-adjoint
$A \in \Mx \overline{\otimes} \Bc(l^2)$ satisfies}\cr
&&\phantom{\limsup} L_s(A) \leq 1, P \leq P_{(-\infty, 0]}(A),
\hbox{ and }Q \leq P_{[a, \infty)}(A)\}
\end{eqnarray*}
for any projections $P,Q \in \Mx \overline{\otimes} \Bc(l^2)$.
\end{theo}

\begin{proof}
Let $\tilde{\rho}(P,Q)$ be the displayed supremum.
Given $a$ and $A$ satisfying the conditions in the definition of
$\tilde{\rho}(P,Q)$, we must have
$$a \leq \rho(P_{(-\infty, 0]}(A), P_{[a,\infty)}(A))
\leq \rho(P,Q)$$
(the first inequality because $L_s(A) \leq 1$, the second because
$P \leq P_{(-\infty, a]}(A)$ and $Q \leq P_{[a,\infty)}(A)$). So
$\tilde{\rho}(P,Q) \leq \rho(P,Q)$.
Conversely, suppose $\rho(P,Q) < \infty$ and
let $A$ be the operator defined in Lemma \ref{distfn}
with $R = P$ and $c = \rho(P,Q)$. Then it is clear that
$Q \leq P_{\{c\}}(A)$, and we have $P \leq P_{\{0\}}(A)$
since $P$ is orthogonal to any projection whose
distance from $P$ is nonzero (Definition \ref{qdfdef} (ii)).
According to Lemma \ref{distfn} we have $L_s(A) \leq 1$, so this
shows that $\rho(P,Q) = c \leq \tilde{\rho}(P,Q)$. Thus
$\tilde{\rho}(P,Q) = \rho(P,Q)$. If $\rho(P,Q) = \infty$ then take
$c \to \infty$ in the preceding argument.
\end{proof}

We now proceed to the main result of this section, which abstractly
characterizes spectral Lipschitz gauges. Because of the mutual
recoverability of spectral Lipschitz numbers and quantum distance
functions and the
equivalence of quantum distance functions with quantum pseudometrics
(Theorem \ref{abch}), this gives us a second intrinsic characterization
of quantum pseudometrics.

In order to state the relevant definition we need the following notion.
Let $\Mx$ be a von Neumann algebra and let $A,B \in \Mx$, 
$A \geq 0$. Then $A_{[B]} \in \Mx$ will denote the positive operator with
spectral subspaces
$$P_{(a,\infty)}(A_{[B]}) = [BP_{(a,\infty)}(A)]$$
for $a > 0$, where the bracket on the right side denotes range projection.

\begin{defi}\label{abspeclip}
A {\it quantum Lipschitz gauge} on a von Neumann algebra $\Mx$ is a function
$\Lc$ from the self-adjoint part of $\Mx \overline{\otimes} \Bc(l^2)$
to $[0,\infty]$ which satisfies
\begin{quote}
(i) $\Lc(A + I) = \Lc(A)$

\noindent (ii) $\Lc(aA) = |a|\Lc(A)$

\noindent (iii) $\Lc(A \vee \tilde{A}) \leq
\max\{\Lc(A), \Lc(\tilde{A})\}$

\noindent (iv) $\Lc(A_{[B]}) \leq \Lc(A)$ if $A \geq 0$

\noindent (v) if $A_\lambda \to A$ weak operator then
$\Lc(A) \leq \sup \Lc(A_\lambda)$
\end{quote}
for any $a \in \Rb$, any self-adjoint $A, \tilde{A}, A_\lambda \in
\Mx \overline{\otimes} \Bc(l^2)$ with $\sup \|A_\lambda\| < \infty$, and
any $B \in I \otimes \Bc(l^2)$.
(In (i), $I$ is the unit in $\Mx \overline{\otimes} \Bc(l^2)$; in
(iii), $A \vee \tilde{A}$ is the spectral join of $A$ and $\tilde{A}$.)
\end{defi}

We emphasize that we do not assume $\Lc(A + \tilde{A}) \leq \Lc(A) +
\Lc(\tilde{A})$; see Example \ref{countersum} below.

\begin{lemma}\label{joinmeet}
Let $\Lc$ be a quantum Lipschitz gauge on a von Neumann algebra $\Mx$
and let $\{A_\lambda\}$ be a bounded family of self-adjoint operators
in $\Mx \overline{\otimes} \Bc(l^2)$. Then
$$\Lc\left(\bigvee A_\lambda\right), \Lc\left(\bigwedge A_\lambda\right)
\leq \sup \Lc(A_\lambda).$$
\end{lemma}

(Again, $\bigvee A_\lambda$ and $\bigwedge A_\lambda$ are the spectral
join and meet. The desired inequality holds for finite joins by property
(iii), and then for arbitrary joins by property (v), taking the weak
operator limit of the net of finite joins; the inequality for
meets then follows from the identity $\bigwedge A_\lambda =
- \bigvee - A_\lambda$.)

\begin{theo}\label{abstractlip}
Let $\Mx$ be a von Neumann algebra. If $\rho$ is a quantum distance
function on $\Mx$ (Definition \ref{qdfdef}) then the associated
spectral Lipschitz gauge $L_s$ is a quantum Lipschitz gauge.
Conversely, if $\Lc$ is a quantum Lipschitz gauge then 
\begin{eqnarray*}
\rho_\Lc(P,Q) &=& \sup\{a \geq 0: \hbox{some self-adjoint
$A \in \Mx \overline{\otimes} \Bc(l^2)$ satisfies}\cr
&&\phantom{\limsup} L_s(A) \leq 1, P \leq P_{(-\infty, 0]}(A),
\hbox{ and }Q \leq P_{[a, \infty)}(A)\}
\end{eqnarray*}
is a quantum distance function. The two constructions are inverse
to each other.
\end{theo}

\begin{proof}
Let $\rho$ be a quantum distance function on $\Mx$ and let
$L_s$ be the associated spectral Lipschitz gauge. We verify properties
(i) -- (v) of Definition \ref{abspeclip}. Properties (i) and (ii)
are easy. For (iii), let $\epsilon > 0$ and let $a, b \in \Rb$, $a < b$,
and observe that
$$P_{(-\infty, a]}(A \vee \tilde{A}) = P_{(-\infty, a]}(A) \wedge
P_{(-\infty, a]}(\tilde{A})$$
and
$$P_{[b, \infty)}(A \vee \tilde{A}) \leq
P_{(b-\epsilon,\infty)}(A \vee \tilde{A}) = P_{(b-\epsilon, \infty)}(A) \vee
P_{(b-\epsilon, \infty)}(\tilde{A}).$$
So
\begin{eqnarray*}
&&\rho(P_{(-\infty, a]}(A \vee \tilde{A}),
P_{[b, \infty)}(A \vee \tilde{A}))\cr
&&\phantom{\limsup}\geq
\min\{\rho(P_{(-\infty, a]}(A \vee \tilde{A}), P_{(b-\epsilon, \infty)}(A)),
\rho(P_{(-\infty, a]}(A \vee \tilde{A}), P_{(b-\epsilon, \infty)}(\tilde{A})\}\cr
&&\phantom{\limsup}\geq
\min\{\rho(P_{(-\infty, a]}(A), P_{[b-\epsilon, \infty)}(A)),
\rho(P_{(-\infty, a]}(\tilde{A}), P_{[b-\epsilon, \infty)}(\tilde{A})\},
\end{eqnarray*}
and hence
\begin{eqnarray*}
&&\frac{b-a}{\rho(P_{(-\infty, a]}(A \vee \tilde{A}), P_{[b,\infty)}(A \vee \tilde{A}))}\cr
&&\phantom{\limsup}\leq
\max\left\{\frac{b-a}{\rho(P_{(-\infty, a]}(A),P_{[b-\epsilon,\infty)}(A))},
\frac{b-a}{\rho(P_{(-\infty, a]}(\tilde{A}),P_{[b-\epsilon,\infty)}(\tilde{A}))}\right\}\cr
&&\phantom{\limsup}\leq
\frac{b-a}{b-a-\epsilon}\max\{L_s(A), L_s(\tilde{A})\}.
\end{eqnarray*}
Taking $\epsilon \to 0$ and the supremum over $a$ and $b$ then yields
property (iii).

Next, suppose $A \geq 0$ and let $B \in I \otimes \Bc(l^2)$. Let $a,b \in \Rb$,
$0 \leq a < b$, and let $\epsilon > 0$. To verify property (iv), as in the
proof of property (iii) it will suffice to show that
$$\rho(P_{(-\infty, a]}(A_{[B]}), P_{[b, \infty)}(A_{[B]})) \geq
\rho(P_{(-\infty, a]}(A), P_{(b-\epsilon, \infty)}(A)).$$
Thus let $P = P_{(-\infty, a]}(A)$ and $Q = P_{(b-\epsilon,\infty)}(A)$
and observe that $P_{(b-\epsilon,\infty)}(A_{[B]}) =[BQ]$ and
\begin{eqnarray*}
P_{(-\infty,a]}(A_{[B]}) &=& I - P_{(a,\infty)}(A_{[B]})\cr
&=& I - [BP_{(a,\infty)}(A)]\cr
&=& I - [B(I-P)].
\end{eqnarray*}
We claim that $[B^*(I - [B(I-P)])] \leq P$. To see this suppose
$v \perp {\rm ran}(P)$. Then $Bv \in {\rm ran}(B(I-P))$, so that
$\langle Bv,w\rangle = 0$ for any $w \in {\rm ran}(I - [B(I-P)])$.
That is, $\langle v, B^*w\rangle = 0$, so we have shown that
$v \perp {\rm ran}(B^*(I - [B(I-P)]))$. This proves the claim.
It now follows from Definition \ref{qdfdef} (vi) that
\begin{eqnarray*}
\rho(P_{(-\infty, a]}(A_{[B]}), P_{(b-\epsilon,\infty)}(A_{[B]}))
&=& \rho(I - [B(I-P)], [BQ])\cr
&=& \rho([B^*(I - [B(I-P)])], Q)\cr
&\geq& \rho(P,Q),
\end{eqnarray*}
as desired.

Finally, to prove property (v), suppose $A_\lambda \to A$ boundedly
weak operator and let $a,b \in \Rb$, $a < b$. Let $\epsilon > 0$. Then
applying (\cite{W6}, Lemma 2.31) to
$(A - aI) \oplus (bI - A) \in (\Mx\overline{\otimes} \Bc(l^2)) \oplus
(\Mx \overline{\otimes} \Bc(l^2))$, we get nets of projections
$\{P_\kappa\}$ and $\{Q_\kappa\}$ in $\Mx \overline{\otimes} \Bc(l^2)$
such that $P_\kappa \to P_{(-\infty, a]}(A)$,
$Q_\kappa \to P_{[b,\infty)}(A)$, and for each $\kappa$ we have
$P_\kappa \leq P_{(-\infty, a + \epsilon]}(A_\lambda)$ and
$Q_\kappa \leq P_{[b - \epsilon,\infty)}(A_\lambda)$ for some
$\lambda$. By Definition \ref{qdfdef} (vii)
\begin{eqnarray*}
\rho(P_{(-\infty, a]}(A), P_{[b,\infty)}(A))
&\geq& \inf_\kappa \rho(P_\kappa,Q_\kappa)\cr
&\geq& \inf_\lambda \rho(P_{(-\infty, a + \epsilon]}(A_\lambda),
P_{[b-\epsilon,\infty)}(A_\lambda))
\end{eqnarray*}
and so
\begin{eqnarray*}
\frac{b - a}{\rho(P_{(-\infty, a]}(A), P_{[b,\infty)}(A))}
&\leq& \sup_\lambda \frac{b - a}{\rho(P_{(-\infty, a + \epsilon]}(A_\lambda),
P_{[b- \epsilon,\infty)}(A_\lambda))}\cr
&\leq& \frac{b-a}{b - a - 2\epsilon} \sup_\lambda L_s(A_\lambda).
\end{eqnarray*}
Taking $\epsilon \to 0$ and the supremum over $a$ and $b$ then
yields $L_s(A) \leq \sup L_s(A_\lambda)$. This completes the proof
that $L_s$ is a quantum Lipschitz gauge.

Next let $\Lc$ be any quantum Lipschitz gauge. We verify that $\rho_\Lc$
is a quantum distance function. Property (i)
follows by taking $A = 0$ in the definition of $\rho_\Lc$,
property (ii) is immediate, and property (iii) follows from the fact
that $\Lc(aI - A) = \Lc(A)$. For property (iv), suppose there exist
self-adjoint operators $A, \tilde{A} \in \Mx \overline{\otimes} \Bc(l^2)$
such that $\Lc(A), \Lc(\tilde{A}) \leq 1$, $R \leq P_{(-\infty, 0]}(A)$,
$P \leq P_{[a,\infty)}(A)$, $R \leq P_{(-\infty, 0]}(\tilde{A})$, and
$Q \leq P_{[a,\infty)}(\tilde{A})$. Then we have $\Lc(A \vee \tilde{A})
\leq 1$,
$$R \leq P_{(-\infty, 0]}(A \vee \tilde{A}),$$
and
$$P \vee Q \leq P_{[a,\infty)}(A \vee \tilde{A}),$$
and taking the supremum over $a$ yields $\rho_\Lc(P \vee Q, R) \geq
\min\{\rho_\Lc(P,R), \rho_\Lc(Q,R)\}$. The reverse inequality is trivial,
so this verifies property (iv). For property (v) let $\epsilon > 0$
and find a self-adjoint operator $A \in \Mx \overline{\otimes} \Bc(l^2)$
such that $\Lc(A) \leq 1$, $P \leq P_{(-\infty, 0]}(A)$, and
$R \leq P_{[b,\infty)}(A)$
where $b = \rho_\Lc(P,R) - \epsilon$. Let $a \in \Rb$ be the largest value
such that $Q \leq P_{[a,\infty)}(A)$. Then $\rho_\Lc(P,Q) \geq a$, and
letting $\tilde{Q} = P_{(-\infty, a + \epsilon]}(A)$ we have
$Q\tilde{Q} \neq 0$ and $\rho_\Lc(\tilde{Q}, R) \geq b - a - \epsilon$. So
$$\rho_\Lc(P,R) = b + \epsilon \leq \rho_\Lc(P,Q) + \rho_\Lc(\tilde{Q}, R) + 2\epsilon,$$
and taking the supremum over $\tilde{Q}$ and $\epsilon \to 0$ yields
property (v).

To establish property (vi), let $\epsilon > 0$ and find a self-adjoint
operator $A \in \Mx \overline{\otimes} \Bc(l^2)$ such that $\Lc(A) \leq 1$,
$[B^*P] \leq P_{(-\infty, 0]}(A)$, and $Q \leq P_{(a, \infty)}(A)$
where $a = \rho_\Lc([B^*P], Q) - \epsilon$. By replacing $A$ with
$A^+ = A \vee 0$ we may assume it is positive. We now have
$\Lc(A_{[B]}) \leq 1$ and
$[BQ] \leq P_{(a, \infty)}(A_{[B]})$. We claim that
$P \leq P_{(-\infty, 0]}(A_{[B]})$, that is,
$P$ is orthogonal to $P_{(0,\infty)}(A_{[B]}) =
[BP_{(0,\infty)}(A)]$. For if $v \in {\rm ran}(P)$ and $w \in
{\rm ran}(P_{(0,\infty)}(A))$ then $\langle v, Bw\rangle =
\langle B^*v, w\rangle = 0$ because $[B^*P] \leq P_{(-\infty, 0]}(A)$,
and this proves the claim. Thus
$$\rho_\Lc(P,[BQ]) \geq a = \rho_\Lc([B^*P],Q) - \epsilon,$$
and taking $\epsilon \to 0$, we conclude that $\rho_\Lc(P, [BQ]) \geq
\rho_\Lc([B^*P], Q)$. The reverse inequality follows by symmetry
(property (iii)), interchanging $P$ with $Q$ and $B$ with $B^*$.
For property (vii), suppose $P_\lambda \to P$ and $Q_\lambda \to Q$
and let $a < \limsup \rho_\Lc(P_\lambda,Q_\lambda)$.
As we frequently have $\rho_\Lc(P_\lambda, Q_\lambda) > a$, find
self-adjoint operators $A_\lambda \in \Mx \overline{\otimes} \Bc(l^2)$ such
that $\Lc(A_\lambda) \leq 1$, $P_\lambda \leq P_{(-\infty, 0]}(A_\lambda)$,
and $Q_\lambda \leq P_{[a,\infty)}(A_\lambda)$. Replacing $A_\lambda$
with $(A_\lambda \vee 0) \wedge aI$, we may assume $0 \leq A_\lambda
\leq aI$. Now pass to a weak operator convergent subnet and let $A =
\lim A_\lambda$. Then $A$ is positive, $P_\lambda \leq P_{\{0\}}(A_\lambda)$,
and for any $v \in {\rm ran}(P)$ we have $\langle P_\lambda v,v\rangle \to
\langle Pv,v\rangle = \|v\|^2$, which implies that
$\langle Av,v\rangle = \lim \langle A_\lambda v,v\rangle = 0$ and
hence that $Av = 0$. This shows that $P \leq P_{\{0\}}(A)$, and
applying the same argument to $aI - A$ yields $Q \leq P_{\{a\}}(A)$. So
$\rho_\Lc(P,Q) \geq a$, and taking $a \to \limsup \rho(P_\lambda,Q_\lambda)$
yields the desired inequality. This completes the proof that $\rho_\Lc$ is
a quantum distance function.

Now let $\rho$ be any quantum distance function, let $L_s$ be
the associated spectral Lipschitz gauge, and let $\rho_{L_s}$
be the quantum distance function derived from $L_s$. Then
$\rho = \rho_{L_s}$ by Theorem \ref{distrecov}.

Finally, let $\Lc$ be any quantum Lipschitz gauge and let $L_s$ be
the spectral Lipschitz gauge associated to $\rho_\Lc$. We immediately have
that $\Lc(A) \leq 1$ implies $\rho_\Lc(P_{(-\infty, a]}(A), P_{[b,\infty)}(A))
\geq b - a$ for any $a,b \in \Rb$, $a < b$, and hence $L_s(A) \leq 1$;
this shows that $L_s(A) \leq \Lc(A)$ for all $A$. Conversely, suppose
$L_s(A) < 1$. For each $a,b \in \Rb$, $a < b$, we have
$\rho_\Lc(P_{(-\infty, a]}(A),P_{[b,\infty)}(A)) > b-a$ so we can
find a self-adjoint
operator $A_{ab} \in \Mx \overline{\otimes} \Bc(l^2)$ such that
$\Lc(A_{ab}) \leq 1$, $P_{(-\infty, a]}(A) \leq
P_{(-\infty, a]}(A_{ab})$, and $P_{[b,\infty)}(A) \leq
P_{[b,\infty)}(A_{ab})$. Then
$$A = \bigwedge_{|a| \leq \|A\|} \bigvee_{a < b \leq \|A\|}
((A_{ab} \vee aI) \wedge bI),$$
which shows that $\Lc(A) \leq 1$. We conclude that $\Lc(A) \leq L_s(A)$
for all $A$, and hence the two are equal.
\end{proof}

\begin{coro}\label{lipequiv}
Let $\Mx \subseteq \Bc(H)$ be a von Neumann algebra. If $\Vb$ is a
quantum pseudometric on $\Mx$ then $L_s$, defined by
$$L_s(A) = \sup\left\{\frac{b-a}{t}:a,b \in \Rb, a < b,
P_{(-\infty, a]}(A)(\Vc_t\otimes I)P_{[b,\infty)}(A) \neq 0\right\}$$
(for $A \in \Mx \overline{\otimes} \Bc(H)$ self-adjoint), is a quantum
Lipschitz gauge on $\Mx$. Conversely, if $\Lc$ is a quantum Lipschitz
gauge on $\Mx$ then $\Vb = \{\Vc_t\}$ with
\begin{eqnarray*}
\Vc_t &=& \{B \in \Bc(H): P_{(-\infty, 0]}(A)(B\otimes I)P_{[a,\infty)}(A)
= 0\cr
&&\hbox{ for all }a > t\hbox{ and all self-adjoint }A \in
\Mx \overline{\otimes} \Bc(l^2) \hbox{ with }\Lc(A) \leq 1\}
\end{eqnarray*}
is a quantum pseudometric on $\Mx$. The two constructions are inverse
to each other.
\end{coro}

The corollary follows straightforwardly from Theorems \ref{abch} and
\ref{abstractlip}.

As we mentioned in Section \ref{abcase}, the spectral Lipschitz gauge
is not a seminorm in general (although it is in the abelian case by
Proposition \ref{seqab} and (\cite{W6}, Corollary 1.21)). We conclude
this section with a simple example which demonstrates this; in fact,
we show that a sum of spectrally Lipschitz operators need not be
spectrally Lipschitz.

\begin{exam}\label{countersum}
Let $\Mx = M_2(\Cb)$ and let $n \in \Nb$. Define a quantum metric on
$\Mx$ by letting $a = 2/n$ and $b = c = 1$ in Proposition \ref{mtwo}.
Let $A = \left[
\begin{matrix}
1&0\cr
0&0
\end{matrix}\right]$
and $B = \frac{1}{n}\cdot\left[
\begin{matrix}
1&1\cr
1&1
\end{matrix}\right]$. Then $A$ has eigenvalues $0$ and $1$ and
the distance between the corresponding spectral subspaces is $1$, so
$L_s(A) = 1$. The operator $B$ has eigenvalues $0$ and $2/n$ and the
distance between the corresponding spectral subspaces is $2/n$, so
$L_s(B) = 1$. But the operator
$$A + B = \frac{1}{n}\cdot\left[
\begin{matrix}
n + 1&1\cr
1&1
\end{matrix}\right]$$
has eigenvalues $(n + 2 \pm \sqrt{n^2 + 4})/2n$ and the distance
between the corresponding spectral subspaces is $2/n$, so
$L_s(A + B) = \sqrt{n^2 + 4}/2$.
This witnesses the failure of the seminorm property $L(A + B) \leq
L(A) + L(B)$. Moreover, by taking the $l^\infty$ direct sum of this sequence
of examples as $n$ ranges over $\Nb$, we obtain operators $\tilde{A}$
and $\tilde{B}$ such that $L_s(\tilde{A}) = L_s(\tilde{B}) = 1$
and $L_s(\tilde{A} + \tilde{B}) = \infty$. Thus a sum
of two spectrally Lipschitz operators need not be spectrally Lipschitz.
\end{exam}

\subsection{Commutation Lipschitz numbers}
We have just seen that spectral Lipschitz gauges are algebraically very
poorly behaved. However, there is a related alternative notion that has
good algebraic properties. Recall that $[\Vc]_1$ denotes the closed
unit ball of the Banach space $\Vc$.

\begin{defi}\label{comlip}
Let $\Vb$ be a quantum pseudometric on a von Neumann algebra $\Mx
\subseteq \Bc(H)$. We define the {\it commutation Lipschitz number}
of $A \in \Mx$ to be
$$L_c(A) =
\sup\left\{\frac{\|[A,C]\|}{t}: t \geq 0, C \in [\Vc_t]_1\right\},$$
where $[A,C] = AC - CA$ and
with the convention that $\frac{0}{0} = 0$. We say that $A$ is
{\it commutation Lipschitz} if $L_c(A) < \infty$ and we call $L_c$ the
{\it commutation Lipschitz gauge}. We define
$${\rm Lip}(\Mx) = \{A \in \Mx: L_c(A) < \infty\}$$
and equip ${\rm Lip}(\Mx)$ with the norm $\|A\|_L = \max\{\|A\|, L_c(A)\}$.
\end{defi}

Note that taking $t = 0$ shows that $L_c(A) < \infty$ implies $A \in
\Vc_0' \subseteq \Mx$. This is a quantum version of the fact that
Lipschitz functions on a pseudometric space respect the equivalence
relation which makes points equivalent if their distance is zero.

We can define the commutation Lipschitz number of any operator in
$\Mx \overline{\otimes} \Bc(l^2)$ by stabilization (see Section \ref{refstab}).
Explicitly, we set
$$L_c(A) =
\sup\left\{\frac{\|[A,C\otimes I]\|}{t}:
t \geq 0, C \in [\Vc_t]_1\right\}$$
for $A \in \Mx \overline{\otimes} \Bc(l^2)$.

There is an analogue of the measurable de Leeuw map (\cite{W6}, Definition
1.19) for commutation Lipschitz operators. We use it to establish
the basic properties of commutation Lipschitz numbers.

\begin{defi}\label{opdeLeeuw}
Let $\Vb$ be a quantum pseudometric on a von Neumann algebra
$\Mx \subseteq \Bc(H)$. The {\it operator de Leeuw map} is the
map $\Phi: A \mapsto \bigoplus_\alpha \frac{1}{t}[A,C]$,
where $\alpha$ ranges over all pairs $(t,C)$ such that $t \geq 0$ and
$C \in [\Vc_t]_1$, from ${\rm Lip}(\Mx)$ into the $l^\infty$ direct
sum $\bigoplus \Bc(H)$. Also define $\pi: \Mx \to \bigoplus \Bc(H)$ by
$\pi(A) = \bigoplus A$.
\end{defi}

\begin{prop}\label{opdlprop}
Let $\Vb$ be a quantum pseudometric on a von Neumann algebra
$\Mx \subseteq \Bc(H)$ and let $\Phi$ be the operator de Leeuw map.

\noindent (a) For all $A \in {\rm Lip}(\Mx)$ we have $L_c(A) =
\|\Phi(A)\|$.

\noindent (b) $\Phi$ is linear and we have $\Phi(AB) = \pi(A)\Phi(B)
+ \Phi(A)\pi(B)$ for all $A,B \in {\rm Lip}(\Mx)$.

\noindent (c) The graph of $\Phi$ is weak* closed in $\Mx \oplus
\bigoplus \Bc(H)$.
\end{prop}

\begin{proof}
Parts (a) and (b) are straightforward. For part (c),
let $\{A_\lambda\}$ be a net in ${\rm Lip}(\Mx)$ and suppose
$A_\lambda \oplus \bigoplus \frac{1}{t}[A_\lambda, C] \to
A \oplus B$ weak*; we must show that $A \in {\rm Lip}(\Mx)$ and
$\Phi(A) = B$. But
\begin{eqnarray*}
\langle [A_\lambda, C]w,v\rangle
&=& \langle A_\lambda Cw,v\rangle - \langle A_\lambda w,C^*v\rangle\cr
&\to& \langle ACw,v\rangle - \langle Aw,C^*v\rangle\cr
&=& \langle [A,C]w,v\rangle
\end{eqnarray*}
for all $C$ and all $v,w \in H$, and this implies
that $B = \bigoplus \frac{1}{t}[A,C]$. Thus $A \in {\rm Lip}(\Mx)$
and $\Phi(A) = B$, as desired.
\end{proof}

\begin{coro}\label{basicqlip}
Let $\Vb$ be a quantum pseudometric on a von Neumann algebra
$\Mx \subseteq \Bc(H)$.

\noindent (a) $L_c(aA) = |a|\cdot L_c(A)$, $L_c(A^*) = L_c(A)$,
$L_c(A + B) \leq L_c(A) + L_c(B)$,
and $L_c(AB) \leq \|A\| L_c(B) + \|B\| L_c(A)$ for all $A, B \in
{\rm Lip}(\Mx)$ and $a \in \Cb$.

\noindent (b) If $\{A_\lambda\} \subseteq \Mx$ is a net that
converges weak* to $A \in \Mx$ then $L_c(A) \leq \sup L_c(A_\lambda)$.

\noindent (c) ${\rm Lip}(\Mx)$ is a self-adjoint unital subalgebra of $\Mx$.
It is a dual Banach space for the norm $\|\cdot\|_L$.
\end{coro}

\begin{proof}
Part (a) is straightforward. For part (b) we use the fact that
$A_\lambda \to A$ weak* implies $[A_\lambda, C] \to
[A,C]$ weak* (and that weak* limits cannot increase norms).
The fact that ${\rm Lip}(\Mx)$ is a unital subalgebra of $\Mx$
follows from part (a), and the fact that it is a dual space
follows from Proposition \ref{opdlprop} (c) since the map
$A \mapsto A \oplus \Phi(A)$ is an isometric isomorphism between
${\rm Lip}(\Mx)$ and the graph of $\Phi$.
\end{proof}

The operator de Leeuw map does not respect adjoints. In order to
ensure $\Phi(A^*) = \Phi(A)^*$ for all $A \in \Mx$ we could change
the definition to a direct sum of derivations into $\Bc(H \oplus H)$
of the form
$$A \mapsto \left[
\begin{matrix}
0&i[A,C]\cr
i[A,C^*]&0
\end{matrix}
\right].$$
Proposition \ref{opdlprop} would still hold and $\Phi$ would then be
a W*-derivation in the sense of (\cite{W4}, Definition 7.4.1) or
(\cite{W5}, Definition 10.3.7).

We now prove our main result about commutation Lipschitz numbers,
which relates them to spectral Lipschitz numbers.

\begin{theo}\label{lipineq}
Let $\Vb$ be a quantum pseudometric on a von Neumann algebra
$\Mx \subseteq \Bc(H)$. Let $A \in \Mx$ be self-adjoint.

\noindent (a) Let $C \in \Bc(H)$.
If $P_{(-\infty, a]}(A)CP_{[b, \infty)}(A) = 0$ for all
$a,b \in \Rb$, $a < b$, such that $b-a > t$, then $\|[A,C]\| \leq t\|C\|$.

\noindent (b) $L_c(A) \leq L_s(A)$.
\end{theo}

\begin{proof}
Part (b) follows from part (a) because we have
$P_{(-\infty, a]}(A)CP_{[b, \infty)}(A) = 0$ for all $t > 0$, all
$C \in [\Vc_t]_1$, and all $a,b \in \Rb$, $a < b$, such that $b-a > tL_s(A)$
(since the latter implies $\rho(P_{(-\infty, a]}(A), P_{[b, \infty)}(A)) > t$).
So part (a) allows us to infer that $\|[A,C]\| \leq tL_s(A)$ for all
$t > 0$ and all $C \in [\Vc_t]_1$. This shows that $L_c(A) \leq L_s(A)$.

We prove part (a). Since $A$ is self-adjoint,
we may suppose $H = L^2(X,\mu)$ and $A = M_f$ for
some real-valued $f \in L^\infty(X,\mu)$. Let $\epsilon > 0$ and find
a real-valued simple function $g \in L^\infty(X,\mu)$ such
that $\|f - g\|_\infty \leq \epsilon$; then we still have
$$P_{(-\infty, a]}(M_g)CP_{[b,\infty)}(M_g) = 0$$
when $b-a > t + 2\epsilon$ because $P_{(-\infty, a]}(M_g) \leq
P_{(-\infty, a + \epsilon]}(M_f)$ and $P_{[b,\infty)}(M_g) \leq
P_{[b-\epsilon, \infty)}(M_f)$. Since $[M_g,C] \to [M_f,C]$
as $\epsilon \to 0$, it will suffice to show that $\|[M_g,C]\|
\leq (t + 2\epsilon)\|C\|$.

Let
\begin{eqnarray*}
\Vc &=& \{B \in \Bc(H): P_{(-\infty, a]}(M_g)BP_{[b,\infty)}(M_g) = 0\cr
&&\hbox{ for all $a,b \in \Rb$, $a < b$, such that }b-a > t + 2\epsilon\},
\end{eqnarray*}
observe that $\Vc$ is a W*-bimodule over the von Neumann algebra of bounded
multiplication operators,
and define $\Phi: \Vc \to \Vc$ by $\Phi(B) = [M_g,B]$.
Say $g = \sum_{i=1}^k a_i\chi_{S_i}$ such that the $S_i$ partition $X$
and write $P_i = M_{\chi_{S_i}}$.
Then $\Phi^n(B) = \sum_{i,j} (a_i - a_j)^nP_iBP_j$, so if we define
$e^{is\Phi}$ by a power series we get
$$e^{is\Phi}(B) = \sum_{i,j = 1}^k e^{is(a_i - a_j)}P_iBP_j =
M_{e^{isg}}BM_{e^{-isg}}.$$
Thus $\|e^{is\Phi}(B)\| = \|B\|$ for
all $s \in \Rb$, which implies that $\Phi$ is a ``Hermitian''
operator (\cite{BD1}, Definition 5.1) on the complex Banach
space $\Vc$  by (\cite{BD1}, Lemma 5.2). It then follows from
Corollary 26.6 of \cite{BD2} that the norm of $\Phi$ equals its
spectral radius $\lim \|\Phi^n\|^{1/n}$.

Since $|a_i - a_j| > t + 2\epsilon$ implies $P_i B P_j = 0$ the expression
$$\Phi^n(B) = \sum_{i,j = 1}^k (a_i - a_j)^nP_iBP_j$$
yields the estimate
$$\|\Phi^n\| \leq k^2(t + 2\epsilon)^n.$$
Thus $\|\Phi\| = \lim \|\Phi^n\|^{1/n} \leq t + 2\epsilon$, and we
conclude that $\|[M_g,C]\| = \|\Phi(C)\| \leq (t + 2\epsilon)\|C\|$,
as desired.
\end{proof}

\begin{coro}\label{spcomcor}
Let $\Vb$ be a quantum pseudometric on a von Neumann algebra
$\Mx \subseteq \Bc(H)$. Then any spectrally Lipschitz self-adjoint
element of $\Mx$ is commutation Lipschitz.
\end{coro}

The converse fails: in general not every commutation Lipschitz operator
is spectrally Lipschitz. This follows from Example \ref{countersum}
since the operators $\tilde{A}$ and $\tilde{B}$
in that example will both be commutation
Lipschitz by the preceding corollary, and hence their sum will be too by
Corollary \ref{basicqlip} (a).

However, in the measure theory setting the spectral and commutation Lipschitz
gauges agree.

\begin{coro}\label{agree}
Let $(X,\mu)$ be a finitely decomposable measure space, let
$\rho$ be a measurable pseudometric on $X$, and let $\Vb_\rho$
be the associated quantum pseudometric on $\Mx \cong L^\infty(X,\mu)$
(Theorem \ref{qmeasmet}). Then for any real-valued $f \in L^\infty(X,\mu)$
we have $L_c(M_f) = L_s(M_f) = L(f)$.
\end{coro}

\begin{proof}
The equality $L_s(M_f) = L(f)$ was Proposition \ref{seqab}, and we
have $L_c(M_f) \leq L_s(M_f)$ by Theorem \ref{lipineq}. For the reverse
inequality let $p,q \in L^\infty(X,\mu)$ be nonzero projections
and let $\epsilon > 0$. Say $p = \chi_S$, $q = \chi_T$ and apply the
equality $\Rc = \Rc_{\Vc_\Rc}$ in Theorem \ref{abelianrel} to the
measurable relation $\Rc = \{(p',q'): \rho(p,q) < \rho(p,q) + \epsilon\}$
(\cite{W6}, Lemma 1.16) to find $C \in \Bc(L^2(X,\mu))$ such that
$M_pCM_q \neq 0$ but $M_{p'}CM_{q'} = 0$
whenever $\rho(p',q') \geq \rho(p,q) + \epsilon$.

Decompose $p$ and $q$ as $p = \sum p_i$ and
$q = \sum q_j$ so that ${\rm ran}_{p_i}(f)$ and ${\rm ran}_{q_j}(f)$
have diameter at most $\epsilon$ for all $i$ and $j$. Fix values of
$i$ and $j$ such that $B = M_{p_i}CM_{q_j} \neq 0$. We may assume that
$\|B\| = 1$. Let $a \in {\rm ran}_{p_i}(f)$ and $b \in {\rm ran}_{q_j}(f)$.
Then $B \in \Vc_{\rho(p,q) + \epsilon}$ and
$$\|[M_f,B] - (a-b)B\| \leq \|(M_f - aI)M_{p_i}B\| +
\|BM_{q_j}(M_f - bI)\| \leq 2\epsilon$$
so
$$L_c(M_f) \geq \frac{\|[M_f,B]\|}{\rho(p,q) + \epsilon}
\geq \frac{|a-b| - 2\epsilon}{\rho(p,q) + \epsilon} \geq
\frac{d({\rm ran}_p(f), {\rm ran}_q(f))- 2\epsilon}{\rho(p,q) + \epsilon}.$$
Taking $\epsilon \to 0$ and the supremum over $p$ and $q$ then yields
$L(f) \leq L_c(M_f)$.
\end{proof}

Theorem \ref{lipineq} is nontrivial even in the atomic abelian case.
For example, let $H = l^2(\Zb)$ and let $U \in \Bc(l^2(\Zb))$ be the
bilateral shift. Let $f(z) = \sum_{k = -n}^n a_k e^{ikx}$ be a
trigonometric polynomial of degree $n$, let $C = \sum_{k=-n}^n a_k U^k$
be the corresponding polynomial in $U$, and for $N \in \Nb$ define
$$g_N(k) =
\begin{cases}
N&\hbox{if }k > N\cr
k&\hbox{if }-N \leq k \leq N\cr
-N&\hbox{if }k < -N
\end{cases}.$$
Then giving $\Zb$ the standard metric, we have $L_s(M_{g_N}) = L(g_n) = 1$
and $D(C) = n$.
So Theorem \ref{lipineq} (b) implies that $\|[M_{g_N},C]\| \leq n\|C\|$.
Taking inner products against standard basis vectors shows that
the weak operator limit of $[M_{g_N},C]$ as $N \to \infty$ is
the operator $\tilde{C} = \sum_{k=-n}^n ka_k U^k$. Thus we conclude that
$\|\tilde{C}\| \leq n\|C\|$. Taking the Fourier transform, we get
$$\|f'\|_\infty = \|\tilde{C}\| \leq n\|C\| = n\|f\|_\infty.$$
This is Bernstein's inequality from classical complex analysis
(see \cite{Bur}).

We include one more general result about ${\rm Lip}(\Mx)$ which states that
it is weak* dense in $\Mx$ if $\Vb$ is a quantum metric. This is not
surprising, but the proof is interesting because it uses some of
the machinery that we have built up in the last two sections.

\begin{prop}\label{lipdensity}
Let $\Vb$ be a quantum metric on a von Neumann algebra $\Mx
\subseteq \Bc(H)$. Then ${\rm Lip}(\Mx)$ is weak* dense in $\Mx$.
\end{prop}

\begin{proof}
The weak* closure of ${\rm Lip}(\Mx)$ is a von Neumann subalgebra of
$\Mx$ by Corollary \ref{basicqlip} (c). By the double commutant
theorem, to prove equality we must show that ${\rm Lip}(\Mx)'
\subseteq \Mx'$. Thus let $C \in \Bc(H) - \Mx'$; we must find an
operator in ${\rm Lip}(\Mx)$ that does not commute with $C$.

Since $\Vb$ is a quantum metric we have $C \not\in \Vc_0$. Thus
$D(C) > 0$ and so by Proposition \ref{recover} there exist projections
$P,Q \in \Mx \overline{\otimes} \Bc(l^2)$ such that $\rho(P,Q) > 0$
and $P(C \otimes I)Q \neq 0$. Now let $A$ be the spectral join in
$\Mx \overline{\otimes} \Bc(l^2)$
defined in Lemma \ref{distfn} with $R = P$ and $c = \rho(P,Q)$. Then
$PA= 0$ and $AQ = \rho(P,Q)\cdot Q$ so
$$P[A,C \otimes I]Q = -\rho(P,Q)P(C \otimes I)Q \neq 0.$$

Since $[A,C \otimes I]$ is nonzero there exists a rank one projection
$P_0 \in \Bc(l^2)$ such that if $B = (I \otimes P_0)A(I \otimes P_0)$ then
$$[B,C \otimes I]
= (I \otimes P_0)[A,C \otimes I](I \otimes P_0) \neq 0.$$
Say $B = B_0 \otimes P_0$; then $[B, C \otimes I] = [B_0,C]\otimes P_0$
and this implies that $[B_0,C]\neq 0$.
Finally, for any $t$ and any $D \in [\Vc_t]_1$ we have
$$\|[B_0,D]\| = \|[B,D \otimes I]\| =
\|(I \otimes P_0)[A, D\otimes I](I \otimes P_0)\|
\leq \|[A,D \otimes I]\| \leq t$$
by Theorem \ref{lipineq} (b)
since $L_s(A) \leq 1$ (relative to the quantum pseudometric
$\Vb \otimes I$ on $\Mx \overline{\otimes} \Bc(l^2)$).
This shows that $L_c(B_0) \leq 1$, so that
$B_0 \in {\rm Lip}(\Mx)$. Thus we have found an operator in
${\rm Lip}(\Mx)$ that does not commute with $C$.
\end{proof}

Finally, we relate ${\rm Lip}(\Mx)$ to $C^*(U_\hbar, V_\hbar)$ in
the quantum tori.

\begin{prop}\label{qtorilip}
Let $\hbar \in \Rb$ and let $d$ be a translation invariant metric on
$\Tb^2$ that is quasi-isometric (i.e., homeomorphic via a bijection which
is Lipschitz in both directions) to the standard metric. Equip
$W^*(U_\hbar, V_\hbar)$ with the quantum metric $\Vb_0$ defined in
Theorem \ref{qtori} (b). Then ${\rm Lip}(W^*(U_\hbar, V_\hbar))$
is (operator norm) densely contained in the C*-algebra $C^*(U_\hbar, V_\hbar)$
generated by $U_\hbar$ and $V_\hbar$.
\end{prop}

\begin{proof}
If $A \in W^*(U_\hbar, V_\hbar)$ is commutation Lipschitz then
$$\|A - \theta_{x,y}(A)\| = \|[A, M_{e^{i(mx + ny)}}]\| \to 0$$
as $(x,y) \to (0,0)$, which implies that $A \in C^*(U_\hbar, V_\hbar)$
by (\cite{W5}, Proposition 6.6.5). This shows that
${\rm Lip}(W^*(U_\hbar, V_\hbar)) \subseteq C^*(U_\hbar, V_\hbar)$.

For density, it will be enough to prove that $U_\hbar$ and $V_\hbar$ belong
to ${\rm Lip}(W^*(U_\hbar, V_\hbar))$ since this will entail that every
polynomial in $U_\hbar$, $V_\hbar$, $U_\hbar^{-1} = U_\hbar^*$, and
$V_\hbar^{-1} = V_\hbar^*$ is in ${\rm Lip}(W^*(U_\hbar, V_\hbar))$
by Corollary \ref{basicqlip} (c).
We will prove that the real and imaginary parts of $U_\hbar$
(actually, any self-adjoint Lipschitz element of $C^*(U_\hbar) \cong
C(\Tb)$) are spectrally Lipschitz, and hence commutation Lipschitz
by Corollary \ref{spcomcor}. This implies that $U_\hbar$ is commutation
Lipschitz by Corollary \ref{basicqlip} (c). The analogous statements for
$V_\hbar$ are proven similarly.

Identify $C^*(U_\hbar)$ with $C(\Tb)$, let $A \in C^*(U_\hbar)$ be
self-adjoint, and suppose $A \in {\rm Lip}(\Tb) \subset C(\Tb)$.
Recall (Definition \ref{qtdefs}) that $\Vc_t = \Vc_{\Ec_0(S_t)}$ consists
of the operators whose $(k,l)$ Fourier term belongs to
$\Ec_0(S_t)\cdot U^k_{-\hbar}V^l_{-\hbar}$, for all $k$ and $l$,
where $\Ec_0(S_t)$ is the weak* closed span of the operators
$M_{e^{i(mx + ny)}}$ with $(x,y) \in S_t$.

Now conjugate all operators in $\Bc(l^2(\Zb^2))$ by the unitary
$M_{e^{i\hbar mn/2}}$. Then $U_\hbar$ will still commute with both
$U_{-\hbar}$ and $V_{-\hbar}$ (an easy computation directly from Definition
\ref{qtdf}), $\Ec_0(S_t)$ is unaffected, and $U_\hbar$ becomes
the shift $e_{m,m} \mapsto e_{m+1,n}$. In the $L^2(\Tb^2)$ picture,
$A$ now becomes multiplication by a Lipschitz function in the
first variable and the operators $M_{e^{i(mx + ny)}}$ with $(x,y) \in S_t$,
which generate $\Ec_0(S_t)$, become translations by vectors of length
at most $t$. Thus if $A = M_f$, $f \in {\rm Lip}(\Tb)$, then for any
$t > 0$ we have
$$P_{(-\infty, a]}(A)\Ec_0(S_{t/L(f)})P_{[b,\infty)}(A) = 0$$
for any $a,b \in \Rb$, $a < b$, such that $b - a > t$.
But since the spectral projections of $A$ commute with $U_{-\hbar}$
and $V_{-\hbar}$, this implies that
$$P_{(-\infty, a]}(A)\Vc_{t/L(f)}P_{[b,\infty)}(A) = 0$$
for any $a,b \in \Rb$, $a < b$, such that $b - a > t$. So
$\rho(P_{(-\infty, a]}(A), P_{[b,\infty)}(A)) \geq (b-a)/L(f)$, and we
conclude that $L_s(A) \leq L(f)$. Thus, we have shown that $A$ is
spectrally Lipschitz, as claimed.
\end{proof}

\subsection{Little Lipschitz spaces}
Classically, little Lipschitz functions are Lipschitz functions which
satisfy a kind of ``local flatness'' condition (see Chapter 3 of \cite{W4}).
On nice spaces like connected Riemannian manifolds the only little Lipschitz
functions are constant functions, but on totally disconnected or H\"older
spaces they are abundant.

We can formulate spectral and commutation versions of the little
Lipschitz condition in our setting.

\begin{defi}
Let $\Vb$ be a quantum pseudometric on a von Neumann algebra
$\Mx \subseteq \Bc(H)$ and let $\rho$ be the associated quantum
distance function.

\noindent (a) A self-adjoint operator $A \in \Mx$ is {\it spectrally
little Lipschitz} if it is spectrally Lipschitz and for every
$\epsilon > 0$ there exists $\delta > 0$ such that
$$\frac{b - a}{\rho(P_{(-\infty, a]}(A), P_{[b,\infty)}(A))} \leq \epsilon$$
for any $a,b \in \Rb$, $a < b$, such that
$\rho(P_{(-\infty, a]}(A), P_{[b,\infty)}(A)) \leq \delta$.

\noindent (b) An operator $A \in \Mx$ is {\it commutation little
Lipschitz} if it is commutation Lipschitz and for every $\epsilon > 0$
there exists $\delta > 0$ such that
$$\frac{\|[A,C]\|}{t} \leq \epsilon$$
whenever $t \leq \delta$ and $C \in [\Vc_t]_1$. We let ${\rm lip}(\Mx)$ be
the set of elements of ${\rm Lip}(\Mx)$ that are commutation little
Lipschitz, equipped with the inherited norm $\|\cdot\|_L$.
\end{defi}

This generalizes the atomic abelian case; see Corollary \ref{aall} below.

\begin{prop}
Let $\Vb$ be a quantum pseudometric on a von Neumann algebra
$\Mx \subseteq \Bc(H)$. Then ${\rm lip}(\Mx)$ is a closed unital
self-adjoint subalgebra of ${\rm Lip}(\Mx)$.
\end{prop}

\begin{proof}
All of the assertions follow from the observation that $A \in {\rm Lip}(\Mx)$
belongs to ${\rm lip}(\Mx)$ if and only if $\Phi_\alpha(A) \to 0$ as
$t \to 0$, where $\alpha$ ranges over all pairs $(t,C)$ such that
$t > 0$ and $C \in [\Vc_t]_1$, and $\Phi_\alpha(A) = \frac{1}{t}[A,C]$.
\end{proof}

We omit the proofs of the next two results; they are straightforward
adaptations of the proofs of Propositions \ref{altform} and \ref{compo}.

\begin{prop}
Let $\rho$ be a quantum distance function on a von Neumann algebra $\Mx$
and let $A \in \Mx$ be self-adjoint and spectrally little Lipschitz.
Then for every $\epsilon > 0$ there exists $\delta > 0$ such that
$$\frac{d(S,T)}{\rho(P_S(A), P_T(A))} \leq \epsilon$$
for any Borel sets $S,T \subseteq \Rb$ such that $\rho(P_S(A), P_T(A))
\leq \delta$.
\end{prop}

\begin{prop}
Let $\rho$ be a quantum distance function on a von Neumann algebra
$\Mx$, let $A \in \Mx$ be self-adjoint and spectrally little Lipschitz,
and let $f: \Rb \to \Rb$ be Lipschitz. Then $f(A)$ is spectrally little
Lipschitz.
\end{prop}

\begin{prop}
Let $\rho$ be a quantum distance function on a von Neumann algebra $\Mx$
and let $A,\tilde{A} \in \Mx$ be self-adjoint and spectrally little Lipschitz.
Then their spectral join and meet are also spectrally little Lipschitz.
\end{prop}

\begin{proof}
The statement about joins follows from the inequality
\begin{eqnarray*}
&&\frac{b-a}{\rho(P_{(-\infty, a]}(A \vee \tilde{A}), P_{[b,\infty)}(A \vee \tilde{A}))}\cr
&&\phantom{\limsup}\leq
\max\left\{\frac{b-a}{\rho(P_{(-\infty, a]}(A),P_{[b-\epsilon,\infty)}(A))},
\frac{b-a}{\rho(P_{(-\infty, a]}(\tilde{A}),P_{[b-\epsilon,\infty)}(\tilde{A}))}\right\}
\end{eqnarray*}
established in the course of showing that $L_s$ satisfies property (iii)
of Definition \ref{abspeclip} in the proof of Theorem \ref{abstractlip}.
(Whichever term on the right dominates the left side must have a smaller
denominator, so the spectral little Lipschitz condition can be applied.)
The statement about meets can either be proven similarly or reduced to
the statement about joins via the identity $A \wedge \tilde{A} =
-(-A \vee -\tilde{A})$.
\end{proof}

\begin{prop}\label{sclit}
Let $\Vb$ be a quantum pseudometric on a von Neumann algebra
$\Mx \subseteq \Bc(H)$. Then any spectrally little Lipschitz
self-adjoint element of $\Mx$ is commutation little Lipschitz.
\end{prop}

\begin{proof}
Let $A \in \Mx$ be self-adjoint and spectrally little Lipschitz.
Given $\epsilon > 0$, find $\delta > 0$ witnessing the spectral
little Lipschitz condition. Fix $0 \leq t \leq \delta$ and $C \in [\Vc_t]_1$.
Suppose $P_{(-\infty, a]}(A)CP_{[b,\infty)}(A) \neq 0$; then
$\rho(P_{(-\infty, a]}(A), P_{[b,\infty)}(A)) \leq t \leq \delta$,
so the spectral little Lipschitz condition implies that
$t \geq \rho(P_{(-\infty, a]}(A), P_{[b,\infty)}(A)) \geq (b-a)/\epsilon$.
This shows that if $b - a > \epsilon t$ then
$P_{(-\infty, a]}(A)CP_{[b,\infty)}(A) = 0$,
and so Theorem \ref{lipineq} (a) yields $\|[A,C]\| \leq \epsilon t$.
We conclude that $A$ is commutation little Lipschitz.
\end{proof}

\begin{coro}\label{aall}
Let $X$ be a set, let $d$ be a pseudometric on $X$, let
$\Mx \cong l^\infty(X)$ be the von Neumann algebra of bounded
multiplication operators on $l^2(X)$, and let $\Vb_d$ be the
quantum pseudometric on $\Mx$ corresponding to $d$ (Proposition \ref{aa}).
If $f \in l^\infty(X)$ is real-valued then $M_f$ is spectrally little
Lipschitz if and only if $M_f$ is commutation little Lipschitz if and
only if for every $\epsilon > 0$ there exists
$\delta > 0$ such that
$$d(x,y) \leq \delta\quad\Rightarrow\quad
\frac{|f(x) - f(y)|}{d(x,y)} \leq \epsilon.$$
\end{coro}

\begin{proof}
Let $f \in l^\infty(X)$ be real-valued. By Corollary \ref{agree} we may
assume $f$ is Lipschitz. First, suppose $f$ is little Lipschitz, i.e.,
it satisfies the $\epsilon$-$\delta$ condition stated in the proposition.
Let $\epsilon > 0$ and find $\delta > 0$ satisfying this condition.
Then if $a,b \in \Rb$, $a < b$, satisfy $\rho(P_{(-\infty,a]}(M_f),
P_{[b,\infty)}(M_f)) < \delta$, we can find sequences $\{x_n\}
\subseteq f^{-1}((-\infty,a])$ and $\{y_n\} \subseteq f^{-1}([b,\infty))$
such that $d(x_n,y_n) \to \rho(P_{(-\infty,a]}(M_f), P_{[b,\infty)}(M_f))$
and $d(x_n,y_n) \leq \delta$ for all $n$.
Then $|f(x_n) - f(y_n)| \geq b-a$ for all $n$, so
$$\frac{b-a}{\rho(P_{(-\infty,a]}(M_f), P_{[b,\infty)}(M_f))} \leq
\limsup \frac{|f(x_n) - f(y_n)|}{d(x_n,y_n)} \leq \epsilon.$$
This verifies the spectral little Lipschitz condition for $M_f$.

Next, if $M_f$ is spectrally little Lipschitz then it is commutation
little Lipschitz by Proposition \ref{sclit}.

Finally, suppose $M_f$ is commutation little Lipschitz, let $\epsilon > 0$,
and find $\delta > 0$ satisfying the commutation little Lipschitz
condition. For any $x,y \in X$ with $t = d(x,y) \leq \delta$, the operator
$V_{xy}$ then belongs to $[\Vc_t]_1$ with $t \leq \delta$, so
$$\frac{|f(x) - f(y)|}{d(x,y)} =
\frac{\|[M_f,V_{xy}]\|}{t} \leq \epsilon.$$
This shows that $f$ is little Lipschitz.
\end{proof}

Finally, we note that just as in the abelian case, little Lipschitz
operators are abundant when the underlying metric space is H\"older.
This result is a straightforward consequence of the definitions of
spectral and commutation little Lipschitz operators, together with
the fact that if $0 < \alpha < 1$ then $t/t^\alpha \to 0$ as $t \to 0$.

\begin{prop}
Let $\Vb$ be a quantum pseudometric on a von Neumann algebra $\Mx
\subseteq \Bc(H)$ and let $0 < \alpha < 1$. Let ${\rm lip}^\alpha(\Mx)$
denote the little Lipschitz space relative to the H\"older quantum
pseudometric $\Vb^\alpha$ (Section \ref{hold}). Then
${\rm Lip}(\Mx) \subseteq {\rm lip}^\alpha(\Mx)$. Any self-adjoint
element of $\Mx$ that is spectrally Lipschitz relative to $\Vb$ will
be spectrally little Lipschitz relative to $\Vb^\alpha$.
\end{prop}

The most significant substantive result about little Lipschitz spaces
states that ${\rm lip}^\alpha(X)^{**} \cong {\rm Lip}^\alpha(X)$ for
any compact metric space $X$ and any $0 < \alpha < 1$. We conjecture
that this remains true for general quantum metrics, with the hypothesis
``$X$ is compact'' modified to ``the closed unit ball of
${\rm Lip}^\alpha(\Mx)$ is compact for the operator norm topology''.

\section{Quantum uniformities}\label{uniformities}

In this brief final chapter we propose a quantum analog of the notion
of a uniform space \cite{Pag}. A classical uniformity on a
set can be defined in terms of a family of relations called ``entourages''.
We can give a natural quantum generalization of this definition which
is representation independent (Theorem \ref{uniftrepindep}) and
effectively reduces to the classical definition in the atomic abelian
case (Proposition \ref{uniftaa}). We find that the basic theory of
uniformities, including their presentability in terms of families of
pseudometrics, generalizes to the quantum setting (Theorem \ref{uniftqpm}),
and we also develop some basic material on
quantum uniform continuity (which, like the Lipschitz condition,
bifurcates into two distinct but related notions). However, we do
not attempt to mine the subject in detail.

\subsection{Basic results}
We start with our definition of a quantum uniformity. It is not
overtly expressed in terms of dual operator bimodules, but we
immediately show that there is an equivalent reformulation in these
terms.

\begin{defi}\label{qudef}
A {\it quantum uniformity} is a family $\Ub$ of dual operator systems
contained in some $\Bc(H)$ that satisfies the following conditions:
\begin{quote}
(i) any dual operator system that contains a member of $\Ub$ belongs
to $\Ub$

\noindent (ii) if $\Uc, \tilde{\Uc} \in \Ub$ then $\Uc \cap \tilde{\Uc} \in \Ub$

\noindent (iii) for every $\Uc \in \Ub$ there exists $\tilde{\Uc} \in \Ub$
such that $\tilde{\Uc}^2 \subseteq \Uc$.
\end{quote}
The elements of $\Ub$ are {\it quantum entourages}. $\Ub$ is a
{\it quantum uniformity on the von Neumann algebra $\Mx \subseteq \Bc(H)$}
if $\Mx' \subseteq \bigcap \Ub$, and it is {\it Hausdorff} if
$\Mx' = \bigcap \Ub$. A subfamily $\Ub_0 \subseteq \Ub$ {\it generates}
$\Ub$ if every member of $\Ub$ contains some member of $\Ub_0$.
\end{defi}

Equivalently, we could work with dual unital operator spaces and require
that $\Uc \in \Ub$ $\Rightarrow$ $\Uc \cap \Uc^* \in \Ub$.

Note that the intersection $\bigcap \Ub$ is always a von Neumann
algebra. (It is clearly a dual operator system,
and it is an algebra by property (iii).) So if $\Ub$ is a quantum
uniformity on the von Neumann algebra $\Mx$ then $\bigcap \Ub$ is a
von Neumann algebra containing $\Mx'$
and we can ensure the Hausdorff property by passing from $\Mx$ to
the commutant of $\bigcap \Ub$ (a possibly smaller von Neumann algebra).

\begin{prop}
Let $\Ub$ be a quantum uniformity on a von Neumann algebra $\Mx
\subseteq \Bc(H)$. Then
$\Ub$ is generated by the subfamily
$$\Ub_0 = \{\Uc \in \Ub: \Uc\hbox{ is a quantum relation on }\Mx\}.$$
\end{prop}

\begin{proof}
Let $\Uc \in \Ub$ and apply property (iii) of Definition \ref{qudef} twice
to obtain $\tilde{\Uc} \in \Ub$ such that $\tilde{\Uc}^3 \subseteq \Uc$.
Then $\overline{\Mx'\tilde{\Uc}\Mx'}^{wk^*}$
is a quantum relation on $\Mx$ that contains $\tilde{\Uc}$, and hence is
a quantum entourage, and it is contained in $\tilde{\Uc}^3 \subseteq \Uc$. 
\end{proof}

Thus, we could just as well define a quantum uniformity on $\Mx$ to be a
family $\Ub$ of quantum relations on $\Mx$ such that
\begin{quote}
(i) $\Mx' \subseteq \Uc = \Uc^*$ for all $\Uc \in \Ub$

\noindent (ii) any quantum relation $\Uc$ that contains a member
of $\Ub$ and satisfies $\Uc = \Uc^*$ belongs to $\Ub$

\noindent (iii) if $\Uc, \tilde{\Uc} \in \Ub$ then
$\Uc \cap \tilde{\Uc} \in \Ub$

\noindent (iv) for every $\Uc \in \Ub$ there exists $\tilde{\Uc} \in \Ub$
such that $\tilde{\Uc}^2 \subseteq \Uc$.
\end{quote}

Given the preceding, the next two results follow from,
respectively, Theorem \ref{repindep} and Proposition \ref{atomiccase}.
Order the quantum uniformities on a von Neumann
algebra by inclusion.

\begin{theo}\label{uniftrepindep}
Let $H_1$ and $H_2$ be Hilbert spaces and let $\Mx_1 \subseteq \Bc(H_1)$
and $\Mx_2 \subseteq \Bc(H_2)$ be isomorphic von Neumann algebras. Then
any isomorphism induces an order preserving 1-1 correspondence between the
quantum uniformities on $\Mx_1$ and the quantum uniformities on $\Mx_2$.
\end{theo}

\begin{prop}\label{uniftaa}
Let $X$ be a set and let $\Mx \cong l^\infty(X)$ be the von Neumann
algebra of bounded multiplication operators on $l^2(X)$. If $\Phi$ is a
uniformity on $X$ then
$$\Ub_\Phi = \{\Uc: \Vc_R \subseteq \Uc\hbox{ for some }R \in \Phi\}$$
($\Vc_R$ as in Proposition \ref{atomiccase}, $\Uc$ ranging over dual
operator systems) is a quantum uniformity
on $\Mx$; conversely, if $\Ub$ is a quantum uniformity on $\Mx$ then
$$\Phi_\Ub = \{U \subseteq X^2:
R_\Uc \subseteq U\hbox{ for some }\Uc \in \Ub\}$$
($R_\Uc$ as in Proposition \ref{atomiccase}) is a uniformity on $X$.
The two constructions are inverse to each other.
\end{prop}

Finally, we show that every quantum uniformity arises from a family of
quantum pseudometrics.

\begin{defi}\label{assocqu}
Let $\{\Vb_\lambda\}$ with $\Vb_\lambda = \{\Vc_t^\lambda\}$
be a family of quantum pseudometrics on a von Neumann
algebra $\Mx \subseteq \Bc(H)$. The {\it associated quantum uniformity}
on $\Mx$ is the family of dual operator systems $\Uc \subseteq \Bc(H)$
such that
$$\Vc^{\lambda_1}_\epsilon \cap \cdots \cap \Vc^{\lambda_n}_\epsilon
\subseteq \Uc$$
for some $\epsilon > 0$, some $n \in \Nb$, and some $\lambda_1, \ldots,
\lambda_n$. Thus, it is the smallest quantum uniformity that contains
$\Vc^\lambda_t$ for every $\lambda$ and every $t > 0$.
\end{defi}

\begin{theo}\label{uniftqpm}
Every quantum uniformity on a von Neumann algebra $\Mx \subseteq \Bc(H)$
is the quantum uniformity associated to some family of quantum
pseudometrics on $\Mx$.
\end{theo}

\begin{proof}
Let $\Ub$ be a quantum uniformity on $\Mx$ and let $\Fc$ be the
family of all quantum pseudometrics $\Vb$ on $\Mx$ with the property
that $\Vc_t$ is a quantum entourage for all $t > 0$. We claim that
$\Ub$ is the quantum uniformity associated to $\Fc$. The inclusion
$\supseteq$ is easy because $\Vc_\epsilon^{\lambda_1} \cap
\cdots \cap \Vc_\epsilon^{\lambda_n}$ is a quantum entourage for
all $\epsilon > 0$ and all $\Vb_{\lambda_1}, \ldots, \Vb_{\lambda_n}
\in \Fc$.

To prove the reverse inclusion, let $\Uc_1 \in \Ub$; we will find
a quantum pseudometric $\Vb \in \Fc$ such that $\Vc_t \subseteq \Uc_1$
for some $t > 0$. To do this, first find a sequence $\{\Uc_n\}$ of
quantum entourages such that $\Uc_{n+1}^3 \subseteq \Uc_n$ for all $n$.
For each $s > 0$ define
$$\Wc_s = \overline{\rm span}^{wk*}\left\{A_1\cdots A_k: k \in \Nb
\hbox{ and }A_i \in \Uc_{n_i}
\hbox{ ($1 \leq i \leq k$) where }\sum_{i=1}^k 2^{-n_i} \leq s\right\},$$
and then define a W*-filtration $\Vb$ of $\Bc(H)$ by setting
$\Vc_t = \bigcap_{s > t} \Wc_s$ for all $t \geq 0$. It is straightfoward
to check that $\Vb$ is a quantum pseudometric on $\Mx$.
We claim that $\Wc_{2^{-n}} = \Uc_n$ for all $n$. It is clear that
$\Uc_n \subseteq \Wc_{2^{-n}}$. For the reverse inclusion,
fix $A_1\cdots A_k \in \Wc_{2^{-n}}$; we want to show that
$A_1\cdots A_k \in \Uc_n$. If $k = 1$ the assertion is trivial, so
we may inductively assume it holds for all $n$ and all smaller values of $k$.
Suppose $k \geq 2$ and split the product up into three segments
$A_1 \cdots A_{j_1}$, $A_{j_1+1}\cdots A_{j_2}$, and $A_{j_2+1} \cdots A_k$
such that the corresponding sums
$\sum_1^{j_1} 2^{-n_i}$, $\sum_{j_1+1}^{j_2} 2^{-n_i}$, and
$\sum_{j_2+1}^k 2^{-n_i}$ are each at most $2^{-n-1}$. Then each of the
three subproducts is in $\Uc_{n+1}$ by the induction hypothesis, and hence the
entire product is in $\Uc_{n+1}^3 \subseteq \Uc_n$. This completes the
proof of the claim.

It follows that $\Vb \in \Fc$ (since for each $t > 0$, $\Vc_t$ contains
$\Wc_{2^{-n}} = \Uc_n$ for any $n$ such that $2^{-n} \leq t$) and that
$\Vc_t \subseteq \Wc_1 \subseteq \Uc_1$ for any $t < 1$, as desired.
\end{proof}

\subsection{Uniform continuity}

The natural morphisms between uniform spaces are the uniformly
continuous maps. As with the Lipschitz condition, in the quantum
setting we have both a spectral version and a commutator version
of this notion.

\begin{defi}\label{uc}
Let $\Ub$ be a quantum uniformity on a von Neumann algebra
$\Mx \subseteq \Bc(H)$.

\noindent (a) A self-adjoint operator $A \in \Mx$ is {\it spectrally
uniformly continuous} if for every $\epsilon > 0$ there exists
$\Uc \in \Ub$ such that
$$P_{(-\infty, a]}(A)\Uc P_{[b,\infty)}(A) = 0$$
for all $a,b \in \Rb$, $a < b$, such that $b-a > \epsilon$.

\noindent (b) An operator $A \in \Mx$ is {\it commutation uniformly
continuous} if for every $\epsilon > 0$ there exists $\Uc \in \Ub$
such that
$$\|[A,C]\| \leq \epsilon$$
for every $C \in [\Uc]_1$. We let $UC(\Mx)$ be the set of commutation
uniformly continuous operators in $\Mx$, with the inherited operator
norm.
\end{defi}

This generalizes the atomic abelian case; see Corollary \ref{quaa}
below.

For quantum uniformities arising from quantum pseudometrics we
can characterize spectral and commutation uniform continuity
directly in terms of the W*-filtration.

\begin{prop}\label{qufilt}
Let $\Mx \subseteq \Bc(H)$ be a von Neumann algebra equipped with a
quantum pseudometric $\Vb$, and let $\Ub$ be the quantum uniformity
generated by the quantum relations $\Vc_t$ for $t > 0$.

\noindent (a) A self-adjoint operator $A \in \Mx$ is spectrally uniformly
continuous relative to $\Ub$ if and only if for every $\epsilon > 0$ there
exists $\delta > 0$ such that
$$\rho(P_{(-\infty, a]}(A), P_{[b,\infty)}(A)) \geq \delta$$
for every $a, b \in \Rb$, $a < b$, with $b - a > \epsilon$.

\noindent (b) An operator $A \in \Mx$ is commutation uniformly
continuous relative to $\Ub$ if and only if for every $\epsilon > 0$
there exists $\delta > 0$ such that $\|[A,C]\| \leq \epsilon$
for every $C \in [\Vc_\delta]_1$.
\end{prop}

The proof of this proposition is straightforward.

Next we observe that, just as for Lipschitz conditions, spectral
uniform continuity is stronger than commutation uniform continuity.
This result follows immediately from Theorem \ref{lipineq} (a).

\begin{theo}\label{quspcom}
Let $\Mx \subseteq \Bc(H)$ be a von Neumann algebra equipped with
a quantum uniformity $\Ub$ and let $A \in \Mx$ be self-adjoint.
If $A$ is spectrally uniformly continuous then it is commutation
uniformly continuous.
\end{theo}

\begin{coro}\label{quaa}
Let $X$ be a set, let $\Phi$ be a uniformity on $X$, let
$\Mx \cong l^\infty(X)$ be the von Neumann algebra of bounded
multiplication operators on $l^2(X)$, and let $\Ub_\Phi$ be the quantum
uniformity on $\Mx$ corresponding to $\Phi$ (Proposition \ref{uniftaa}).
If $f \in l^\infty(X)$ is real-valued then $M_f$ is spectrally uniformly
continuous if and only if $M_f$ is commutation uniformly continuous if
and only if for every $\epsilon > 0$ there exists $R \in \Phi$ such
that $(x,y) \in R$ implies $|f(x) - f(y)| \leq \epsilon$.
\end{coro}

\begin{proof}
Suppose that $f$ is uniformly continuous in the sense stated in the
corollary, let $\epsilon > 0$, and find an entourage $R$ witnessing
uniform continuity of $f$. Then $\Vc_R \in \Ub_\Phi$, and $(x,y) \in R$
implies $|f(x) - f(y)| \leq \epsilon$, so that
$P_{(-\infty, a]}(M_f)V_{xy} P_{[b,\infty)}(M_f) = 0$ whenever $b - a >
\epsilon$, for every $(x,y) \in R$.
Since $\Vc_R$ is generated by $\{V_{xy}: (x,y) \in R\}$, this
shows that $P_{(-\infty, a]}(M_f)\Vc_R P_{[b,\infty)}(M_f) = 0$
whenever $b - a > \epsilon$, and this demonstrates that $M_f$ is
spectrally uniformly continuous.

Spectral uniform continuity implies commutation uniform continuity
by Theorem \ref{quspcom}.

Finally, if $f$ is not uniformly continuous then there exists $\epsilon > 0$
such that for every entourage $R \in \Phi$ there is a pair
$(x,y) \in R$ with $|f(x) - f(y)| > \epsilon$. Then the operator
$V_{xy}$ belongs to $[\Vc_R]_1$, and we have $\|[M_f,V_{xy}]\| =
|f(x) - f(y)| > \epsilon$. Since every quantum entourage contains a
quantum entourage of the form $\Vc_R$, this shows that $M_f$ is not
commutation uniformly continuous. So commutation uniform continuity of
$M_f$ implies uniform continuity of $f$.
\end{proof}

Next we look at algebra and lattice properties of spectral and
commutation uniform continuity.

\begin{prop}
Let $\Mx \subseteq \Bc(H)$ be a von Neumann algebra equipped with a
quantum uniformity $\Ub$ and let $A, \tilde{A} \in \Mx$ be self-adjoint
and spectrally uniformly continuous. Then their spectral join and meet
are also spectrally uniformly continuous.
\end{prop}

\begin{proof}
Let $\epsilon > 0$ and find quantum entourages $\Uc$ and $\tilde{\Uc}$
such that
$$P_{(-\infty, a]}(A)\Uc P_{(b,\infty)}(A) =
P_{(-\infty, a]}(\tilde{A})\tilde{\Uc} P_{(b,\infty)}(\tilde{A}) = 0$$
for all $a,b \in \Rb$ with $b - a > \epsilon$. Then
$$P_{(-\infty,a]}(A \vee \tilde{A}) =
P_{(-\infty,a]}(A) \wedge P_{(-\infty,a]}(\tilde{A})$$
and
$$P_{(b,\infty)}(A \vee \tilde{A}) =
P_{(b,\infty)}(A) \vee P_{(b,\infty)}(\tilde{A}),$$
so
$$P_{(-\infty,a]}(A \vee \tilde{A})(\Uc \cap \tilde{\Uc})
P_{(b,\infty)}(A \vee \tilde{A}) =0.$$
This shows that $A \vee \tilde{A}$ is uniformly continuous. The
fact that $A \wedge \tilde{A}$ is uniformly continuous can either
be proven analogously or inferred from the equality
$A \wedge \tilde{A} = -((-A) \vee (-\tilde{A}))$.
\end{proof}

\begin{prop}\label{ucalg}
Let $\Mx \subseteq \Bc(H)$ be a von Neumann algebra equipped with a
quantum uniformity $\Ub$. Then $UC(\Mx)$ is a unital C*-algebra.
\end{prop}

The proof of this proposition is routine.

The sum of two spectrally uniformly continuous operators need not be
spectrally uniformly continuous. Indeed, this is the case for the
operators constructed in Example \ref{countersum}, as one can easily
check using the characterization of spectral uniform continuity given
in Proposition \ref{qufilt}. Since $\tilde{A}$ and $\tilde{B}$ are both
spectrally Lipschitz, it follows that they are spectrally uniformly
continuous, but their sum fails spectrally uniform continuity
because $\rho(P_{(-\infty, 1/2]}(\tilde{A} + \tilde{B}),
P_{[1,\infty)}(\tilde{A} + \tilde{B})) = 0$.

Recall that a quantum uniformity $\Ub$ on a von Neumann algebra $\Mx$
is Hausdorff if $\Mx = \bigcap \Ub$ (Definition \ref{qudef}). We now
show that under this hypothesis $UC(\Mx)$ is weak* dense in $\Mx$; this
result is analogous to, and easily deduced from, the corresponding result
about weak* density of ${\rm Lip}(\Mx)$ in $\Mx$ (Proposition
\ref{lipdensity}).

\begin{prop}\label{ucdense}
Let $\Ub$ be a Hausdorff quantum uniformity on a von Neumann algebra
$\Mx \subseteq \Bc(H)$. Then $UC(\Mx)$ is weak* dense in $\Mx$.
\end{prop}

\begin{proof}
We must show that $UC(\Mx)' \subseteq \Mx'$. Thus let $C \in \Bc(H) -
\Mx'$; we must find an operator in $UC(\Mx)$ that does not commute
with $C$.

Since $\Ub$ is Hausdorff and $C \not\in \Mx'$, we must have
$C \not\in \Uc$ for some quantum entourage $\Uc$. By
Theorem \ref{uniftqpm} there is a quantum pseudometric
$\Vb$ on $\Mx$ such that every $\Vc_t$ is a quantum entourage and
$\Vc_t \subseteq \Uc$ for some $t > 0$. Then $C \not\in \Vc_0$, so by
Proposition \ref{lipdensity} there is an operator $B \in \Vc_0' \subseteq \Mx$
that is commutation Lipschitz relative to $\Vb$, and hence commutation
uniformly continuous relative to $\Ub$, and does not commute
with $C$. 
\end{proof}

We conclude with a simple result about commutation uniform continuity
in the quantum tori. Recall that on a compact space every continuous
function is uniformly continuous.

\begin{prop}
Let $\hbar \in \Rb$ and let $d$ be a translation invariant metric on
$\Tb^2$ that is equivalent to the flat
Euclidean metric. Equip $W^*(U_\hbar,V_\hbar)$
with the quantum metric $\Vb_0$ defined in Theorem \ref{qtori} (b). Then
$UC(W^*(U_\hbar,V_\hbar)) = C^*(U_\hbar, V_\hbar)$.
\end{prop}

\begin{proof}
The proof is similar to, but slightly simpler than, the proof of
Proposition \ref{qtorilip}. If $A \in W^*(U_\hbar, V_\hbar)$ is
commutation uniformly continuous then we must have
$$\|A - \theta_{x,y}(A)\| = \|[A, M_{e^{i(mx + ny)}}]\| \to 0$$
as $(x,y) \to (0,0)$, and this implies that $A \in C^*(U_\hbar, V_\hbar)$
by (\cite{W5}, Proposition 6.6.5).

Conversely, by Proposition \ref{ucalg}, to establish that every
operator in $C^*(U_\hbar, V_\hbar)$ is commutation uniformly
continuous it will suffice to show this for $U_\hbar$ and $V_\hbar$.
We will prove that the real and imaginary parts of $U_\hbar$
(actually, any self-adjoint element of $C^*(U_\hbar)$) are
spectrally uniformly continuous, and hence commutation uniformly
continuous by Theorem \ref{quspcom}. The analogous statements for
$V_\hbar$ are proven similarly.

Let $A \in C^*(U_\hbar)$ be self-adjoint. As in the proof of
Proposition \ref{qtorilip}, $\Vc_t = \Vc_{\Ec_0(S_t)}$ consists
of the operators whose $(k,l)$ Fourier term belongs to
$\Ec_0(S_t)\cdot U^k_{-\hbar}V^l_{-\hbar}$, for all $k$ and $l$.

Now conjugate $\Bc(l^2(\Zb^2))$ by the operator $M_{e^{i\hbar mn/2}}$.
Then $U_\hbar$ will still commute with both $U_{-\hbar}$ and
$V_{-\hbar}$, $\Ec_0(S_t)$ is unaffected, and $U_\hbar$ becomes
the shift $e_{m,m} \mapsto e_{m+1,n}$. In the $L^2(\Tb^2)$ picture,
$A$ now becomes multiplication by a continuous (hence uniformly continuous)
function in the first variable and the operators $M_{e^{i(mx + ny)}}$ with
$(x,y) \in S_t$, which generate $\Ec_0(S_t)$, become translations by vectors
of length at most $t$. Thus given $\epsilon > 0$ we can find $\delta > 0$
such that
$$P_{(-\infty, a]}(A)\Ec_0(S_\delta)P_{[b,\infty)}(A) = 0$$
for any $a,b \in \Rb$, $a < b$, such that $b - a > \epsilon$.
But since the spectral projections of $A$ commute with $U_{-\hbar}$
and $V_{-\hbar}$, this implies that
$$P_{(-\infty, a]}(A)\Vc_\delta P_{[b,\infty)}(A) = 0$$
for any $a,b \in \Rb$, $a < b$, such that $b - a > \epsilon$.
Thus, we have shown that $A$ is spectrally uniformly continuous,
as claimed.
\end{proof}


\end{document}